\documentclass[a4paper,10pt,reqno]{amsart}
 
\usepackage{amsmath}
\usepackage{amsthm}
\usepackage{amsfonts}
\usepackage{amssymb}
\usepackage{mathrsfs}
\usepackage[shortlabels]{enumitem}
\usepackage{graphicx}
\usepackage{subfigure}
\usepackage{color}
\usepackage[dvipsnames]{xcolor}
\newtheorem{theorem}{Theorem}[section]
\numberwithin{theorem}{section}
\newtheorem{prop}[theorem]{Proposition}
\newtheorem{lemma}[theorem]{Lemma}
\newtheorem{corollary}[theorem]{Corollary}
\theoremstyle{definition}
\newtheorem{definition}[theorem]{Definition}
\newtheorem{rem}[theorem]{Remark}

\newtheorem{example}[theorem]{Example}

\numberwithin{equation}{section}

\newcommand{\N}{\mathbb{N}}

\newcommand{\R}{\mathbb{R}}

\newcommand{\C}{\mathbb{C}}

\newcommand\e{\mathrm{e}}
\newcommand\I{\mathrm{i}}
\newcommand\re{\operatorname{Re}}
\newcommand\im{\operatorname{Im}}

\newcommand{\cH}{\mathbb H}
\newcommand{\cV}{\mathcal V}
\newcommand{\rd}{\mathrm{d}}

\newcommand\dom{\mathcal D}
\newcommand\ran{\mathcal R}

\newcommand\cQ{\mathcal Q}

\newcommand\cD{\mathcal D}

\newcommand\lbar\overline

\newcommand\eps\varepsilon
\renewcommand\epsilon\varepsilon
\renewcommand\rho\varrho
\newcommand\al\alpha
\newcommand\lm\lambda

\newcommand\ds\displaystyle

\newcommand\p\partial

\newcommand{\tolong}{\longrightarrow}

\newcommand{\gsr}{\stackrel{gsr}{\rightarrow}}
\newcommand{\s}{\stackrel{s}{\rightarrow}}

\newcommand{\gsrlong}{\stackrel{gsr}{\longrightarrow}}
\newcommand{\slong}{\stackrel{s}{\longrightarrow}}

\newcommand{\dist}{\operatorname{dist}}

\newcommand{\beq}{\begin{equation}}
\newcommand{\eeq}{\end{equation}}
\newcommand{\be}{\begin{equation*}}
\newcommand{\ee}{\end{equation*}}
\newcommand{\bmat}{\begin{pmatrix}}
\newcommand{\emat}{\end{pmatrix}}

\author{Sabine B\"ogli}
\address[S.\,B.]{\!
Department of Mathematics, Imperial College London, South Kensington Campus, London SW7 2AZ, UK}
\email{s.boegli@imperial.ac.uk}

\author{Marco Marletta}
\address[M.\,M.]{\!
School of Mathematics,
Cardiff University,
21–-23 Senghennydd Road, Cardiff CF24 4AG, UK}
\email{MarlettaM@cardiff.ac.uk}

\author{Christiane Tretter}
\address[C.\,T.]{\!
Mathematisches Institut,
Universit\"at Bern, Sidlerstr.\ 5, 3012 Bern, Switzerland}
\email{tretter@math.unibe.ch}

\thanks{
}

\usepackage[update,prepend]{epstopdf} 

\title[The essential numerical range for unbounded linear operators]{The essential numerical range for unbounded linear operators}

\begin{document}



\date{\today}

\begin{abstract}
We introduce the concept of essential numerical range $W_{\!e}(T)$ for unbounded Hil\-bert space operators $T$ 
and study its fundamental properties including possible equivalent characterizations and perturbation results. 
Many of the properties known for the bounded case do \emph{not} carry over to the unbounded case,
and new interesting phenomena arise which we illustrate by some striking examples. A key feature 
of the essential numerical range $W_{\!e}(T)$ is that it captures spectral pollution in a unified and minimal way 
when approximating $T$ by projection methods or domain truncation methods for PDEs. 
\end{abstract}

\maketitle

\section{Introduction}
The main object of this paper is the \emph{essential numerical range} $W_{\!e}(T)$ which we introduce for unbounded operators $T$ in a Hilbert space. 
This concept is of great importance in the spectral analysis of non-normal operators and, in particular, in the numerical analysis of differential operators and approximations thereof. Our principal results include the analysis of several alternative, but only partly equivalent characterizations of $W_{\!e}(T)$, a series of perturbation theorems, and results showing that $W_{\!e}(T)$  captures spectral pollution in a unified and minimal way when 
$T$ is approximated by projection and/or domain truncation methods. Diverse examples and applications, e.g.\ to non-symmetric strongly elliptic PDEs, illustrate the sharpness and wide range of applicability of our results.
    
There are good reasons for the long time elapsed between this article and the first papers on the essential numerical range for bounded operators, 
dating back to Stampfli and Williams \cite{Stampfli-Williams} in 1968 and subsequent joint work with Fillmore \cite{fillmore}.
 The unbounded case is significantly different from the bounded case in several respects.  We show that definitions which are equivalent in the bounded case may yield very different sets in the unbounded case. It was not clear at the outset which would be most appropriate to regard as the canonical essential numerical range.  Moreover, none of the usual tools such as graph norms or mapping theorems can be used to reduce the unbounded case to the bounded case, so that a gamut of new ideas and tools had to be developed. The pay-off has exceeded our most optimistic expectations, both on the abstract level and for applications.

The original idea of the essential numerical range was to give a convex enclosure of the essential spectrum, just as the  (closure of the) numerical range gives a convex enclosure for the approximate point spectrum. However we became interested in the essential numerical range also because of an ambitious aim to establish an abstract tool for capturing \emph{spectral pollution}, independent of the particular type of approximation method and not limited to special operator classes such as selfadjoint, close-to-selfadjoint, or second-order-differential. The key connection between the essential numerical range $W_{\!e}(T)$ and spectral pollution is  the new concept of \emph{limiting essential numerical range} (Definition \ref{defWeTn} below).

Some of the earliest descriptions of the phenomenon of spectral pollution were motivated by finite element approximations in plasma physics, see, e.g.\ \cite{MR642457, MR642458, MR800851}, which analysed sequences of eigenvalues of the approximating problems converging to a limit that is not a true eigenvalue. Already there it was noted that such spurious eigenvalues can only occur in gaps of the essential spectrum of the selfadjoint operators considered. An interesting reverse perspective on spectral pollution is that approximating large, finite-domain PDE problems by infinite-domain problems may result in a loss of spectral information, with the lost spectrum being termed absolute spectrum in \cite{MR2341158, MR1782392}.

Although a substantial literature on spectral pollution is now available, most of it concerns selfadjoint operators
and deals with particular methods to approximate spectra such as projection and/or domain truncation methods, see, e.g.\ \cite{MR3022252,boulton,Marletta-2010,lewin-sere,davies-plum-2004,levitin,MR1661578}; some works discuss methods to avoid spectral pollution, while others try to characterize the sets in $\mathbb C$ in which spectral pollution may be present. Dauge and Suri \cite{MR1935966,MR2215994} follow Descloux \cite{descloux} to circumvent the unboundedness in their selfadjoint problems by considering a concept of essential numerical range with respect to a coercive form (a change of topology); their essential numerical range is then the convex hull of the essential spectrum. The only step away from selfadjointness without recourse to perturbation arguments, is the generalization of the classical Titchmarsh-Weyl nesting analysis for $M$-functions of Sturm-Liouville operators to non-selfadjoint cases in~\cite{MR1701776}.

While the main applications presented concern spectral pollution, we emphasize that this article is really about the essential numerical range itself.  In the first part, we start in Section~\ref{section:3} by fixing our definition for the concept of \emph{essential numerical range}
and examining fundamental issues, including geometric properties related to the numerical range and the question of when $W_{\!e}(T)$ is empty, 
which can only occur for unbounded operators. In Section~\ref{section:equiv}, our first main result, Theorem \ref{thmequivdefforWe}, introduces 
four further possible definitions of essential numerical range. Unlike the case of bounded operators studied by Fillmore, Stampfli and Williams~\cite{fillmore}, Salinas~\cite{salinas}, Pokrzywa \cite{pokrzywa-1,pokrzywa} and Descloux~\cite{descloux}, in general these are not the same and the conditions under which at least some of them coincide are non-trivial. For example, an important role is played by the domain intersection $\dom (T) \cap \dom (T^*)$ which, even for m-accretive operators, can be anything from $\{0\}$ to a dense set, see~\cite{AT}. We study the relationship between the essential numerical range and the convex hull of the 
%\marginpar{\ct{\tiny hull correct, or hulls?}}
various different types of essential spectrum. In general, the latter may be a much smaller set and only in particular cases, 
e.g.\ if the operator is selfadjoint and semibounded, do they coincide; for non-semibounded selfadjoint operators, $W_{\!e}(T)$  
coincides with the convex hull of the \emph{extended essential spectrum} of Levitin and Shargorodsky \cite{levitin}.
In Section~\ref{sectionperturb} we derive several perturbation results and describe some startling examples
showing that some results which one may have expected to be true, are actually false, see, e.g.\ Remark \ref{rem:Wetilde} and Example \ref{xmas}.
We also establish some useful results which may be used to compute the essential numerical range when an operator can be decomposed into real and imaginary~parts. 

In the second part of our paper, Section~\ref{sec:limessnr} introduces the notion of limiting essential numerical range for a sequence of operators 
$(T_n)_{n\in\N}$, which is 
%\marginpar{\ct{\tiny not correct?}}
even new in the bounded case, and derives conditions on approximation methods under which it coincides with the essential numerical range of the approximated operator $T$. Section~\ref{sectiongalerkin} studies spectral pollution arising from approximation of operators by projection methods in a Hilbert space, while 
%\marginpar{\ct{\tiny `Section 7 studies' correct?}} Section~\ref{sectiondomaintruncation}
studies approximation by domain truncation of strongly elliptic, not necessarily selfadjoint partial differential operators on domains in 
${\mathbb R}^d$. Here our main results, Theorems \ref{thmsummary} and \ref{thmtruncation}, describe how closely the essential numerical range captures spectral pollution, without any recourse whatsoever to hypotheses of selfadjointness or perturbation-from-selfadjointness.
This is illustrated by applications to non-selfadjoint neutral delay differential equations and differential equations with non-real essential spectrum 
of advection-diffusion type. 

Throughout this paper we denote by $H$ a separable infinite-dimensional Hilbert space. 
The notations $\|\cdot\|$ and $\langle\cdot,\cdot\rangle$ refer to the norm and scalar product of $H$.
Strong and weak convergence of elements in $H$ is denoted by $x_n\to x$ and $x_n\stackrel{w}{\to}x$, respectively.
By $L(H)$ we denote the space of all bounded linear operators acting in $H$, and by $C(H)$  the space of all closed linear operators in $H$.
Norm and strong operator convergence in $L(H)$ is denoted by $T_n\to T$ and $T_n\s T$, respectively.
%\marginpar{\tiny \ct{there is not just one}}
Identity operators are denoted by $I$; scalar multiples $\lm I$ are written as $\lm$.
The domain, range, spectrum, point spectrum and resolvent set
%\marginpar{\tiny\ct{ notations also used for $T\!\notin\!C(H)$}}
of an operator $T$ in $H$ are denoted by $\dom(T)$, $\ran(T)$, $\sigma(T)$, $\sigma_p(T)$, $\rho(T)$, respectively, 
and $T^*$ denotes the Hilbert space adjoint of $T$;  
note that whenever we assume that an operator has non-empty resolvent set, the operator is automatically closed. 
The numerical range is $W(T):=\{\langle T x,x\rangle:\,x\in\dom(T),\,\|x\|=1\}$. 
For non-selfadjoint operators there exist (at least) five different definitions for the essential spectrum which all coincide in the selfadjoint case; for a discussion see  Edmunds and Evans~\cite[Chapter~IX]{edmundsevans}.
Here we use
$$\sigma_e(T):=\left\{\lm\in\C:\,\exists\, (x_n)_{n\in\N}\subset\dom(T) \text{ with }\|x_n\|=1,\,x_n\stackrel{w}{\to}0,\,\|(T-\lm)x_n\|\to 0\right\},$$
which is $\sigma_{e,2}$ in~\cite{edmundsevans}. 
%
%\marginpar{\tiny \ct{sectorial need not be closed in Kato}}
Following  Kato~\cite[Section V.3.10]{Kato}, we call a linear operator $T$ in $H$ 
sectorial if $W(T)\subset\{\lm\in\C:\,|\arg(\lm-\gamma)|\leq\theta\}$ for some sectoriality semi-angle $\theta\in [0,\pi/2)$ and sectoriality vertex $\gamma\in\R$. {$T$} is $m$-sectorial if, in addition, $\lm\in\rho(T)$ for some (and hence all) $\lm\in\C\backslash \overline{W(T)}$.
For a sesquilinear form $t$ in $H$ with domain $\dom(t)$, sectoriality is defined analogously.
%\marginpar{\tiny\ct{$\Phi$ not good for subspaces}}
A subspace $\cD\subset \dom(T)$ is called a core of a closable operator $T$ if $T|_{\cD}$ is closable with closure $\overline{T}$; a core of a closable sequilinear form is defined analogously, see~\cite[Sections III.5.3,~IV.1.4]{Kato} (note that here we do not restrict ourselves to sectorial forms).
For a subset $\Omega\subset\C$ we denote its interior by ${\rm int}\,\Omega$, its convex hull by ${\rm conv}\,\Omega$, its complex conjugated set by $\Omega^*:=\{\overline{z}:\,z\in\Omega\}$, and the distance of $z\in\C$ to $\Omega$ {by} ${\rm dist}(z,\Omega):=\inf_{w\in \Omega}|z-w|$. Finally, $B_{r}(\lm):=\{z\in\C:\,|z-\lm|<r\}$ is the open disk of radius $r$ around $\lm\in\C$.

\section{The essential numerical range {of unbounded operators}}
\label{section:3}

The essential numerical range $W_{\!e}(T)$ was introduced by Stampfli and Williams in~\cite{Stampfli-Williams} for a bounded linear operator $T$ 
in a Hilbert space $H$ as the closure of the numerical range of the image of $T$ in the Calkin algebra, $W_{\!e}(T):= \bigcap\, \{\overline{W(T+K)}: K\,\text{compact}\}$. Various equivalent characterizations were established in the sequel in~\cite{fillmore}. It is immediate from the definition that $W_{\!e}(T)$ is a compact convex subset of $\C$, and one can show that $W_{\!e}(T)\ne \emptyset$ in the bounded case.

The generalization to the unbounded case is not as straightforward as one might expect and leads to interesting new phenomena. In particular, 
for some characterizations seemingly obvious generalizations might fail; e.g.\ in the original definition one cannot replace compact perturbations by relatively compact ones. Moreover, the different characterizations are no longer equivalent in general and some questions only arise in the unbounded case, e.g.\ when is $W_{\!e}(T) \ne \emptyset$.

In the unbounded case, the characterization established in~\cite[Theorem (5.1)\,(3)]{fillmore} turns out to be a good starting point. 
Note that in general, if not stated otherwise, we consider unbounded linear operators $T$ that do not need to be closable or closed. 

\begin{definition}\label{defWe}
For a linear operator $T$ with domain $\dom(T) \subset H$
we define the \emph{essential numerical range} of $T$ by 
% sorted?
% \marginpar{\tiny{\ct{\bf here is the first abuse of notation, shall we leave it or change it?}}} 
$$
W_{\!e}(T):=\left\{\lm\in\C:\,\exists\, (x_n)_{n\in\N}\subset\dom(T) 
\text{ with }
%\begin{array}{c}
\|x_n\|\!=\!1,\,x_n\!\stackrel{w}{\to}\!0,\\[1mm] 
\langle Tx_n,x_n\rangle \to  %\stackrel{\ct{n\to\infty}}{\longrightarrow}
\lm
%\end{array}
\right\}.$$
\end{definition}

Clearly, $W_{\!e}(T) \subset \overline{W(T)}$ by definition and %, as for the numerical range of an un\-bounded operator, 
$W_e(zT)=zW_{\!e}(T)$ and $W_e(T\!+\!z)=W_{\!e}(T)\!+\!z$ for~$z\in\C$.

Our first aim is to investigate the equivalence of other possible definitions of $W_{\!e}(T)$, including the original one in~\cite{Stampfli-Williams}, see Theorem~\ref{thmequivdefforWe}. To this end we need some geometric properties of $W_{\!e}(T)$, which are of independent interest.

First we show that $W_{\!e}(T)$ continues to be closed and convex in the unbounded case. Secondly, 
we investigate some relations between the geometry of the numerical range $W(T)$ and that of $W_{\!e}(T)$; 
they will also provide criteria for $W_{\!e}(T)$ to be unbounded or non-empty. 

\begin{prop}\label{propWeclosed-convex}
The essential numerical range $W_{\!e}(T)$ is closed and convex, and ${\rm conv}\, \sigma_e(T) \subset W_{\!e}(T)$.
\end{prop}

\begin{proof}
The closedness of $W_{\!e}(T)$ follows by a standard diagonal sequence argument.
To show that $W_{\!e}(T)$ is convex, let $\lm,\mu\in W_{\!e}(T)$. Then there
exist two sequences $(x_n)_{n\in\N}$, $(y_n)_{n\in\N}\subset\dom(T)$
with $\|x_n\|=\|y_n\|=1$, $x_n\stackrel{w}{\to} 0$,
$y_n\stackrel{w}{\to}0$ as $n\to\infty$ and
$$|\langle x_n,y_n\rangle|<\frac{1}{n}, \quad |\langle
Tx_n,x_n\rangle-\lm|<\frac{1}{n}, \quad |\langle
Ty_n,y_n\rangle-\mu|<\frac{1}{n},\quad n\in\N.$$
Let $t\in [0,1]$ and $\nu:= t \lm +(1-t)\mu\in{\rm conv}\,\{\lm,\mu\}$.
For $n\in\N$, denote by $P_n:H\to {\rm span}\,\{x_n,y_n\}$ the orthogonal projection in $H$ onto
${\rm span}\,\{x_n,y_n\}$ and define the compression $T_n:=P_nT|_{\ran(P_n)}$.
Since $\langle Tx_n,x_n\rangle, \langle Ty_n,y_n\rangle\in W(T_n)$ and
the latter is convex, there exists
$z_n\in\ran(P_n)$ with $\|z_n\|=1$ and $$|\langle Tz_n,z_n\rangle
-\nu|=|\langle T_n z_n,z_n\rangle -\nu|<\frac{1}{n}, \quad n\in\N.$$
Now $x_n\stackrel{w}{\to} 0$, $y_n\stackrel{w}{\to}0$ as $n\to\infty$
and $|\langle x_n,y_n\rangle|<1/n$, $n\in\N$,  imply
$z_n\stackrel{w}{\to}0$ as $n\to\infty$ and so $\nu\in W_{\!e}(T)$.

The inclusion $\sigma_e(T) \subset W_{\!e}(T)$ is immediate from the definitions 
and so the last claim follows since $W_{\!e}(T)$ is convex.
\end{proof}

Next we give criteria for $W_{\!e}(T)\ne \emptyset$ in terms of the numerical range. 
It turns out that the case where $W(T)$ is a half-plane is different from all others, see Corollary~\ref{corlines} and Example~\ref{corlines}. 

\begin{prop}
\label{WeTnonempty}
If $\,\overline{W(T)}$ is a line or a strip or if $\,W(T)=\C$, then $W_{\!e}(T)\neq\emptyset$.
In particular, $W_{\!e}(T)\neq\emptyset$ if \,$T$ is densely defined and not closable.
\end{prop}

\begin{proof}
In the case when $\overline{W(T)}$ is a strip, or a line which we regard as a special case of a strip of zero width, we can always assume without loss of generality that $W(T)$ is a strip containing $\R$.

Let $(x_n)_{n\in\N}\subset\dom(T)$ with $\|x_n\|=1$ and $x_n\stackrel{w}{\to}0$ as $n\to\infty$. If $(\langle Tx_n,x_n\rangle)_{n\in\N}$ is bounded, then it has a convergent subsequence whose limit belongs to $W_{\!e}(T)$ and hence $W_{\!e}(T)\neq\emptyset$.
If $(\langle Tx_n,x_n\rangle)_{n\in\N}$ is unbounded, we can assume without loss of generality that
$0<\re\langle Tx_n,x_n\rangle\to\infty$ in all cases. To prove the existence of some $\lm\in W_{\!e}(T)$, 
we proceed in (at most) two steps, one to control the real part and, in the case when $W(T)=\C$, one for the imaginary part.

Since $\overline{W(T)}$ is either a strip containing $\R$ or $W(T)=\C$, we can choose $(y_n)_{n\in\N}\subset \dom(T)$ with $\|y_n\|=1$ and 
$$0> \langle Ty_n,y_n\rangle\tolong -\infty, \quad \frac{\langle Tx_n,x_n\rangle}{\langle Ty_n,y_n\rangle}\tolong 0, \quad n\to\infty.$$
It is not difficult to verify that, for every $n\in\N$ there exists $\theta_n\in [0,2\pi)$ such that 
$$\alpha_n:=\e^{\I\theta_n}\langle Ty_n,x_n\rangle+\e^{-\I\theta_n}\langle Tx_n,y_n\rangle\in \I\R.$$
Define
$$r_n: = \sqrt{\frac{\re\langle Tx_n,x_n\rangle}{|\langle Ty_n,y_n\rangle|}}>0, \quad u_n:=x_n+r_n\e^{\I\theta_n}y_n,\quad n\in\N.$$
Then $r_n\to 0$ and hence $\|u_n\|\to 1$ and $u_n\stackrel{w}{\to}0$ as $n\to\infty$.
Moreover, 
$$\langle Tu_n,u_n\rangle=\im \langle Tx_n,x_n\rangle+r_n \alpha_n \in\I\R,\quad n\in\N.$$
Now, with $x_n':=u_n/\|u_n\|$, $n\in\N$, it is easy to see that $\|x_n'\|=1$ and $x_n'\stackrel{w}{\to}0$  as $n\to\infty$.
If $\langle Tx_n',x_n'\rangle=\langle Tu_n,u_n\rangle / \|u_n\|^2\in\I\R$, $n\in\N$, are uniformly bounded, then again $W_{\!e}(T)\neq\emptyset$. \
This is always the case if $W(T)$ is a strip and hence the proof is complete in this case.

So it remains to consider the case that $W(T)=\C$ and $\langle Tx_n',x_n'\rangle \in \I\R$, $n\in\N$, is not uniformly bounded,
without loss of generality $\langle Tx_n',x_n'\rangle\to \I\infty$ as $n\to\infty$.
Since $W(T)=\C$, there exists $(y_n')_{n\in\N}\subset\dom(T)$ such that  $\|y_n\|=1$ and
\beq\label{eq.yn}\langle Ty_n',y_n'\rangle\in \I (-\infty,0), \quad \langle Ty_n',y_n'\rangle\tolong -\I\infty, \quad \frac{\langle Tx_n',x_n'\rangle}{\langle Ty_n',y_n'\rangle}\tolong 0, \quad n\to\infty.\eeq
One may check that, for every $n\in\N$, there exist unique $\theta_n'\in [0,2\pi)$, $r_n'>0$~with
\begin{align*}
& \alpha_n':=\e^{\I\theta_n'}\langle Ty_n',x_n'\rangle+\e^{-\I\theta_n'}\langle Tx_n',y_n'\rangle \in \I(-\infty,0],\\
& \frac{1}{r_n'}-r_n'\left|\frac{\langle Ty_n',y_n'\rangle}{\langle Tx_n',x_n'\rangle}\right|=\left|\frac{\alpha_n'}{\langle Tx_n',x_n'\rangle}\right|.
\end{align*}
Using the last convergence in~\eqref{eq.yn}, we deduce that $r_n'\to 0$ as $n\to\infty$.
Now define $u_n':=x_n'+r_n'\e^{\I\theta_n'}y_n'$ and $v_n:=u_n'/\|u_n'\|$ for $n\in\N$. 
Then it is straightforward to check that $\|v_n\|=1$, $v_n\stackrel{w}{\to}0$ and 
$\langle Tv_n,v_n\rangle =0$, $n\in\N$; hence $0\in W_{\!e}(T)$.

The last claim is immediate from the first claim and the fact that if $T$ is densely defined and $W(T)\neq \C$, then $T$ is closable, see 
\cite[Thm.~V.3.4]{Kato}.
\end{proof}

\begin{prop}
\label{proplinesinWeT}
If there exist $z\!\in\!W_{\!e}(T)$ and $w\in\C\backslash\{0\}$ with $z+w(0,\infty)\!\subset\!W(T)$,  then $z+w [0,\infty)\!\subset\! W_{\!e}(T)$. 
\end{prop}

\begin{proof}
%i)
Without loss of generality take $z=0$ and $w=1$, which can always be arranged by shift of origin and rotation.
Hence there exists $(x_n)_{n\in\N}\subset\dom(T)$ with $\|x_n\|=1$, $x_n\stackrel{w}{\to}0$ and $\langle Tx_n,x_n\rangle \to 0$ as $n\to\infty$.
Let $\lm\in[0,\infty)$ be arbitrary. 
By the assumption $(0,\infty)\!\subset\!W(T)$, there exists $(y_k)_{k\in\N}\!\subset\!\dom(T)$ with $\|y_k\|=1$ and $0\!<\!\langle Ty_k,y_k\rangle\!\to\!\infty $.
Since $x_n\stackrel{w}{\to}{0}$  as $n\to\infty$, we can choose a strictly increasing sequence $(n_k)_{k\in\N}\subset\N$ such that
\beq\label{eq.ykxnk}|\langle y_k,x_{n_k}\rangle|<\frac 1 k, \quad |\langle T y_k,x_{n_k}\rangle|<\frac 1 k, \quad k\in\N.\eeq
Now define
$$u_k:=x_{n_k}+r_k\e^{\I\theta_k} y_k, \quad k\in\N,$$
with $r_k\geq 0$ and $\theta_k\in [0,2\pi)$ such that
$$r_k^2\langle Ty_k,y_k\rangle +r_k |\langle Tx_{n_k},y_k\rangle|=\lm, \quad \e^{-\I\theta_k}\langle Tx_{n_k},y_k\rangle\geq 0, \quad k\in\N.$$
Note that $r_k\to 0$ since $\langle Ty_k,y_k\rangle\to \infty$, and hence $\|u_k\|\to 1$, $u_k\stackrel{w}{\to}0$ as $k\to\infty$.
By direct calculation,
$$\langle Tu_k,u_k\rangle =\langle Tx_{n_k},x_{n_k}\rangle + \lm + r_k \e^{\I \theta_k}\langle Ty_k,x_{n_k}\rangle,\quad k\in\N,$$
and the latter converges to $\lm$ by $\langle Tx_{n_k},x_{n_k}\rangle\to 0$, $r_k\to 0$  as $k\to\infty$ and by the second estimate in~\eqref{eq.ykxnk}.
Now, with $v_k:=u_k/\|u_k\|$, $k\in\N$, it is easy to see that $\|v_k\|=1$, $v_k\stackrel{w}{\to}0$ and $\langle Tv_k,v_k\rangle\to\lm\in W_{\!e}(T)$.
\end{proof}

\begin{corollary}\label{corlines}
\begin{enumerate}[leftmargin=19pt,itemindent=0pt]
\item[\hspace{-2mm}\rm i)]
If \,$W(T)$ is a line, then so is $W_{\!e}(T)$ and thus $W_{\!e}(T)\!=\!W(T)$.
\item[\rm ii)]
If \,$\overline{W(T)}$ is a strip, then $W_{\!e}(T)$ is a strip or a line.
\item[\rm iii)]
If \,$\overline{W(T)}$ is a half-plane and $W_{\!e}(T)\!\neq\!\emptyset$, then $W_{\!e}(T)$ is a~half-plane.
\item[\rm iv)]
If \,$W(T)=\C$, then $W_{\!e}(T)=\C$, and vice versa.
\end{enumerate}
\end{corollary}

\begin{proof}
Using the convexity of $W_{\!e}(T)$ by Proposition~\ref{propWeclosed-convex} and Corollary~\ref{corlines}, the claims follow from Proposition~\ref{proplinesinWeT} if we know that $W_{\!e}(T)\neq\emptyset$.
The latter was proved in Proposition \ref{WeTnonempty} for cases i), ii) and iv), and it is assumed in case iii). 
The converse in iv) follows from $W_{\!e}(T)\subset \overline{W(T)}$. 
\end{proof}

The following example shows that $W_{\!e}(T)=\emptyset$ is possible if $W(T)$ is a half-plane.

\begin{example}
\label{ex:WeTempty} 
For the diagonal operator $T={\rm diag}\,(n+\I (-1)^n n^2: n\in\N_0)$ in the Hilbert space $H=l_2(\N_0)$, we have 
$\overline{W(T)} = \{ z\in\C: \re z \ge 0\}$  but $W_{\!e}(T)=\emptyset$; the latter follows from 
the equivalent characterization $W_e(T)=W_{e2}(T)= \bigcap\, \{\overline{W(T+K)}: K\in L(H),\,\text{rank}<\infty\}$ which we will 
prove in Theorem~\ref{thmequivdefforWe} below.
\end{example}

The following technical lemma for the case that $W(T)=\C$ is needed for the proof of 
two of the main results of this paper, Theorem \ref{thmequivdefforWe} and Theorem \ref{thmdescloux2}.

\begin{lemma} 
\label{newlem}
Suppose that $W(T)=\C$. 
\begin{enumerate}
\item[\rm i)] Let $x,y\in\dom(T)$, $\|x\|=\|y\|=1$, be linearly independent. Then,
for all but at most three $t\in\C$, 
\begin{equation}\label{eq:C}
W(T|_{(y+tx)^{\perp}\cap\dom(T)})=\C.
\end{equation}
\item[\rm ii)]
Let $y\in \dom(T)$, $\|y\|=1$, be such~that $\{y\}^{\perp}\cap\dom(T) \ne \{0\}$.
Then, for every $\epsilon>0$, there exists a $w_\epsilon \in\dom (T)$, $\|w_\epsilon\|=1$, with
\begin{equation}\label{eq.c2}
W(T|_{\{w_\epsilon\}^{\perp}\cap\dom(T)})=\C,  
\quad
| \langle Tw_\epsilon,w_\epsilon\rangle - \langle Ty,y\rangle |< \epsilon.
\end{equation}
\end{enumerate}
\end{lemma}

\begin{proof}
If~\eqref{eq:C} holds for every $t\in\C$, there is nothing to show.
Hence assume that~\eqref{eq:C} is false for some $t\in\C$; without loss of generality $t=0$.
Then there exists 
\begin{equation}\label{eq.lmW}
\lm\notin\overline{W(T|_{\{y\}^{\perp}\cap\dom(T)})}.
\end{equation} 
Since $W(T) = \C$, we also have $W_{\!e}(T)=\C$ by Corollary \ref{corlines} iv) and hence
there exists $(x_k)_{k\in\N}\subset \dom(T)$ with $\|x_k\|=1$, $x_k\stackrel{w}{\to}0$ and $\langle Tx_k,x_k\rangle\to\lm$ 
as $k\to\infty$.

Suppose first that $\sup_{k\in\N}\left|\langle Tx_k,y\rangle\right|<\infty$. If, for every $k\in\N$, we write $x_k=x_k^{(1)}+x_k^{(2)}
\in{\rm span}\{y\}\oplus \{y\}^{\perp}$, then $x_k^{(2)}\in \{y\}^{\perp} \cap \dom(T)$ since $x_k$, $y \in \dom(T)$.
Since $x_k^{(1)}=\langle x_k,y\rangle y\to 0$ and $\|x_k^{(2)}\|\to 1$ as $k\to\infty$,  we arrive at
\begin{align*}
\langle Tx_k,x_k\rangle&=\langle Tx_k^{(1)},x_k\rangle+\langle Tx_k^{(2)},x_k^{(1)}\rangle+\langle Tx_k^{(2)},x_k^{(2)}\rangle\\
&=\langle x_k,y \rangle \langle Ty,x_k \rangle + \langle y, x_k\rangle \big(\langle Tx_k,y\rangle-\langle x_k,y\rangle\langle Ty,y\rangle \big)
+\langle Tx_k^{(2)},x_k^{(2)}\rangle
\end{align*}
for $k\in\N$. Observe that all terms on the right hand side except the last tend to $0$ as $k\to\infty$. 
Since the left hand side has limit $\lambda$ and $x_k^{(2)} \in \{y\}^{\perp}\cap\dom(T)$, we obtain that
$\lambda = \lim_{k\to\infty} \langle Tx_k^{(2)},x_k^{(2)}\rangle \in\overline{W(T|_{\{y\}^{\perp}\cap\dom(T)})}$, 
a contradiction to~\eqref{eq.lmW}. 

Hence $\sup_{k\in\N}\left|\langle Tx_k,y\rangle\right|\!=\!\infty$; without loss of generality $\left|\langle Tx_k,y\rangle\right|\!\to\!\infty$ 
as $k\!\to\!\infty$. 
Since $x,y$ are linearly independent, there exists $u\in{\rm span}\{x,y\}$ with $u\perp y$ and $\|u\|=1$. 
Replacing $(x_k)_{k\in\N}$ by a subsequence if necessary, we may assume that the sequence 
$ \left(\frac{\langle Tx_k,u \rangle}{\langle Tx_k, y\rangle} \right)_{k\in\mathbb N} $
converges to a limit $\alpha\in {\mathbb C}\cup\{\infty\}$ as $k\to\infty$.
Then, for all $s\in\C$ with $s\in\C\setminus\{0,\alpha\}$,
\begin{equation}\label{eq.uref0}
\frac{\langle Tx_k,(y+s u)\rangle}{\langle Tx_k,(u-sy)\rangle} 
\tolong \frac{1+s\alpha}{\alpha-s} \in {\mathbb C},
\quad k\to\infty,
\end{equation} 
with the convention that the right hand side is 
$s$ if $\alpha=\infty$.
We define 
$$
y_s:=y+s u\in \dom(T), \quad u_s:=u-sy\in \{y_s\}^{\perp}\cap\dom(T), \quad s\in\C.
$$
Since $|\langle Tx_k,y\rangle|\to\infty$ we see from (\ref{eq.uref0}) that 
\begin{equation}\label{eq.ueps}
|\langle Tx_k,y_s \rangle|\tolong \infty,\quad |\langle Tx_k,u_s \rangle|\tolong \infty, \quad k\to\infty. 
\end{equation}
Now we fix $s\in\C\setminus\{0,\alpha\}$. By~\eqref{eq.uref0},
\begin{equation}\label{eq.ueps2}
\sup_{k\in\N} \left| \frac{\langle Tx_k,y_s\rangle}{\langle Tx_k,u_s\rangle} \right| < \infty.  
\end{equation}
 For arbitrary $z\in\C$ and $k\in\N$, set
\begin{align*}
\beta_k&:=\frac{\|u_s\|^2 z-\langle Tu_s,u_s\rangle}{\langle Tx_k,u_s\rangle} \in\C, \\
v_k&:=u_s+\beta_k\left(x_k-\frac{\langle x_k,y_s\rangle}{\|y_s\|^2}y_s\right)\in\{y_s\}^{\perp}\cap\dom(T).
\end{align*}
By~\eqref{eq.ueps}, $|\langle Tx_k,u_s\rangle|\to\infty$ and hence $\beta_k\to 0$ as $k\to\infty$, from which it follows at once that  $\|v_k\|^2\to \|u_s\|^2$
as $k\to\infty$. Also by \eqref{eq.ueps2}, $(\beta_k\langle Tx_k,y_s\rangle)_{k\in\N}$ is bounded, and of course we already know 
that $x_k \stackrel{w}{\to} 0$ as $k\to\infty$.
Using these two facts, together with the convergence $\langle Tx_k,x_k\rangle\to\lambda$  as $k\to\infty$, a laborious direct calculation shows that
$\langle Tv_k,v_k\rangle\to \|u_s\|^2 z$ as $k\to\infty$.
Hence $z\in\overline{W(T|_{\{y_s\}^{\perp}\cap\dom(T)})}.$
Since $z$ was chosen arbitrarily, we arrive at %the first claim in~\eqref{eq.c2}.
$$
W(T|_{\{y_s\}^{\perp}\cap\dom(T)})=\C, \quad s\in\C\setminus\{0,\alpha\}.
$$
Since $y,x\in\dom(T)$ are linearly independent, we can write $x=a y+b u$ for some $a,b\in\C$, $b\neq 0$.
Now we obtain~\eqref{eq:C} for all $t\in\C\setminus\{0,-1/a,(b/\alpha -a)^{-1}\}$ since
$$
\{y+tx\}^{\perp}=\{y_s\}^{\perp}\quad\text{with}\quad s=\frac{b}{a+1/t}.
$$

ii)
If $W(T|_{\{y\}^{\perp}\cap\dom(T)})=\C$, we choose $w_\epsilon=y$. Otherwise, by i), we can choose $w_{\eps}=\frac{y+t x}{\|y+t x\|}$ with $t>0$ so small that the second assertion in~\eqref{eq.c2} is satisfied.
\end{proof}

\section{Equivalent characterizations of $W_{\!e}(T)$}\label{section:equiv}

Next we show that two of the other characterizations of $W_{\!e}(T)$ established in~\cite{fillmore}
are equivalent to the definition of $W_{\!e}(T)$ given in the previous section also in the unbounded case, and another one 
is equivalent for densely defined operators.

However, there is one characterization which, in the unbounded case, is equivalent only under some additional conditions, 
even if $T$ is densely defined and closable; a counter-example will show that these conditions are also necessary, 
see Remark~\ref{newrem}~iv) and Example \ref{exequivWe}.

\begin{theorem}\label{thmequivdefforWe}
Let $\cV$ be the set of all finite-dimensional subspaces $V\subset H$. 
\vspace{1mm} Define
   \begin{align*}
    W_{e1}(T)&:=\hspace{2mm} \underset{V \in\cV}{\bigcap}\hspace{2mm}\overline{W(T|_{V^{\perp}\cap\dom(T)})},\\
    W_{e2}(T)&:=\hspace{-1mm}\underset{K\in L(H)\atop {\rm rank }\,K<\infty}{\bigcap}\hspace{-1mm}\overline{W(T+K)},\\
    W_{e3}(T)&:=\hspace{0mm}\underset{K\in L(H)\atop K \text{ compact}}{\bigcap}\!\overline{W(T+K)},\\
    W_{e4}(T)&:=\hspace{2mm}\left\{\lm\in\C:\,\exists\,(e_n)_{n\in\N}\subset\dom(T)\text{ orthonormal with }
		                             \langle Te_n,e_n\rangle\stackrel{n\to\infty}{\longrightarrow}\lm\right\}. \vspace{2mm}
    %,\\[3mm]
    %W_{e5}(T)&:=\big\{\lm\in\C:\,\exists\,Q\in\cQ_T\,\text{ with }\,Q(T-\lm)Q \text{ compact}\big\}.
   \end{align*}
Then, in general,
\begin{align}
\label{Weinclgeneral}
&W_{e1}(T)\subset W_{e4}(T) \subset W_{e2}(T)= W_{e3}(T)=W_{\!e}(T).\\
\intertext{%
If $\,\overline{\dom(T)}=H$, then
}
\label{Weincldd}
&W_{e1}(T)\subset W_{e4}(T)= W_{e2}(T)= W_{e3}(T)= W_{\!e}(T).
\end{align}% 
If $\,\overline{\dom(T)\cap\dom(T^*)}=H$ or if $\,W(T)\neq\C$, then
%\marginpar{\tiny\ct{order changed in (3.3) for consistency}}
\begin{equation}
\label{Weinclspecial}
W_{ei}(T)=W_{\!e}(T), \quad i=1,2,3,4.
\end{equation}
\end{theorem}

\smallskip

\begin{rem}
\label{newrem}
If $\overline{\dom(T)}=H$, then
\begin{enumerate}
\item[i)] $\overline{\dom(T)\cap\dom(T^*)}= H$ necessitates that $T$ is closable, see \cite[Thm.\ 5.3]{MR566954};
\item[ii)] $W(T)\neq\C$ necessitates that $T$ is closable, see \cite[Thm.~V.3.4]{Kato};
\item[iii)] if $\dom(T)\subset\dom(T^*)$, then $\overline{\dom(T)\cap\dom(T^*)}=H$ is satisfied;
\end{enumerate}
in particular, \eqref{Weinclspecial} holds if $T$ is a symmetric operator. Unlike the bounded case, 
\begin{enumerate}
\item[iv)] the inclusion $W_{e1} \subset W_{\!e}(T) \ (=\!W_{ei}(T),\, i\!=\!2,3,4)$ in \eqref{Weinclgeneral} can be strict 
if $\overline{\dom(T)\cap\dom(T^*)}\ne H$ and $W(T)=\C$, see~Example \ref{exequivWe}.
\end{enumerate}
\end{rem}

\begin{proof}[Proof of Theorem {\rm \ref{thmequivdefforWe}}]
%First we show $W_{e5}(T)\subset W_{e1}(T)$.
%Let $\lm\in W_{e5}(T)$. 
%Then there exists $Q\in\cQ_T$ such that $Q(T-\lm)Q$ is compact.
% Let $V\in \cV$ be arbitrary and let $P\in L(H)$ be the orthogonal projection onto $\ran(Q)\cap V^{\perp}$, which is an infinite-dimensional space since $\ran(Q)$ is infinite-dimensional and $V$ has finite dimension. Note that $QPQ=P$.
% Then $P(T-\lm)P=QPQ(T-\lm)QPQ$ is compact, and hence so is $P(T-\lm)|_{\ran(P)}$. We conclude 
% $$0\in \overline{W(P(T-\lm)|_{\ran(P)})}=\overline{W((T-\lm)|_{\ran(P)})}.$$ Now $\ran(P)\subset V^{\perp}\cap\dom(T)$ implies $\lm\in \overline{W(T|_{V^{\perp}\cap\dom(T)})}$. Therefore, $\lm\in W_{e1}(T)$.
%
%
In the following sequence of inclusions 
\begin{equation}
\label{bern1}
W_{e4}(T)\subset W_{\!e}(T)\subset W_{e3}(T) = W_{e2}(T)
\end{equation}
all `$\subset$' are obvious and the reverse inclusion $W_{e3}(T)\supset W_{e2}(T)$ follows since every compact operator is the norm limit of 
finite rank operators. 

Now we prove $ W_{e1}(T)\subset W_{e4}(T).$
Let $\lm \in W_{e1}(T)$ and $e_0\in\dom(T)$ with $\|e_0\|=1$.
To show that $\lm \in W_{e4}(T)$, we inductively construct a sequence $(e_n)_{n\in\N}\subset \dom(T)$ such that, for $n\in\N$, 
$$ \|e_n\|=1, \quad \langle e_n,e_k\rangle =0, \quad k<n, \quad |\langle Te_n,e_n\rangle-\lm| <\frac{1}{n}.$$
For this, let $n\in\N$ and assume that $e_0,\dots,e_{n-1}$ with the described properties have been constructed.
Set $V_n:={\rm span}\{e_0,\dots,e_{n-1}\}\in\cV$.
Since $\lm \in W_{e1}(T)$, we have $\lm \in\overline{W(T|_{V_n^{\perp}\cap\dom(T)})}$ and hence there exists $e_n\in V_n^{\perp}\cap\dom(T)$
with $\|e_n\|=1$ and $|\langle T e_n,e_n\rangle-\lm|<1/n$. Now the claim follows by induction over $n\in\N$.

Next we show $W_{e3}(T)\subset W_{\!e}(T)$.
First we consider the case that $W(T)$ is contained in a half-plane, without loss of generality $\overline{W(T)} \!\subset\! \left\{ z\!\in\!\C\!: \right.$ $\left. 0 \le {\rm Re\,} z\right\}$. 
If $W_{\rm e3}(T) \subset W_{\rm e}(T)$ were false, there would exist a $\lambda_0 \in W_{\rm e3}(T) \setminus W_{\!e}(T)$. 
By Proposition \ref{propWeclosed-convex}, $W_{\!e}(T)$ is closed and convex. Hence, by the strong separation property, see e.g.\ \cite[Thm.~3.6.9]{MR3497790}, there exists a  closed half-plane $\cH$ with $\cH \supset W_{\rm e}(T)$ but $\lambda_0 \notin \cH$.
Then there exist $\theta_{\cH} \!\in\! (-\pi, \pi]$ and $z_0 \!\in\! \C$ with $\cH = z_0 + \{ z\in\C :\theta_{\cH} - \frac \pi 2 \le {\rm arg\,} z \le \theta_{\cH} + \frac \pi 2 \}$. 
For every angle $\theta \in (\frac \pi 2, \frac {3\pi} 2)$, we take an arbitrary positive compact operator $K$ and find $(x_n)_{n\in\N} \subset D(T)$, $\|x_n\|=1$, without loss of generality $x_n \stackrel{w}{\to} x$ as $n\to \infty$, such that
\[\lambda_n =  \langle Tx_n,x_n\rangle -  {\rm e}^{{\rm i} \theta} n \langle Kx_n,x_n\rangle \tolong \lambda_0, \quad n\to \infty.\]
If $\theta_{\cH} \ne \pi$, we can choose $\theta \!\in\! (\frac \pi 2, \frac {3\pi} 2)$ so that $\theta \!\in\! [\theta_{\cH} \!+\! \frac \pi 2, \theta_{\cH} \!+\! \frac {3 \pi}2 ]$. Then
$\cos \theta < 0$ and thus
%\marginpar{\tiny\ct{3 $n$ were in rm}}
\[ 0 \le {\rm Re\,} \langle Tx_n,x_n\rangle  = \cos \theta \, n \langle Kx_n,x_n\rangle + {\rm Re}\, \lambda_n  \le {\rm Re}\, \lambda_n \tolong {\rm Re}\, \lambda_0, \quad n\to \infty. \]
Hence there is a convergent subsequence $({\rm Re\,} \langle Tx_{n_k},x_{n_k}\rangle)_{k\in \N}$, which implies that also $n_k \langle Kx_{n_k},x_{n_k}\rangle \to \mu \ge 0$, $k\to\infty$; in particular, $\langle Kx_{n_k}, x_{n_k}\rangle \to 0$, $k\to \infty$, which necessitates $x=w\!-\!\lim_{k\to\infty} x_{n_k} = 0$. Altogether we obtain
the contradiction
\[ \langle Tx_{n_k},x_{n_k}\rangle  \tolong {\rm e}^{{\rm i} \theta}  \mu + \lambda_0 \in W_{\rm e}(T) \setminus {\cH}, \quad k \to\infty. \]
If $\theta_{\cH} \!=\! \pi$, i.e.\ $W_{\rm e}(T)$ is contained in a vertical strip, then the same is true for $\overline{W(T)}$ by Proposition~\ref{proplinesinWeT}.
In this case we can rotate everything such that we are in the case already proved.
Now we assume that $W(T)=\C$. Then $W_{\!e}(T)=\C$ by Proposition~\ref{proplinesinWeT} and hence $W_{e3}(T)=\C=W_{\!e}(T)$ by \eqref{bern1}. 
This completes the proof of \eqref{Weinclgeneral}.

For \eqref{Weincldd} it remains to be proved that $W_{\!e}(T)\subset W_{e4}(T)$ if $T$ is densely defined, which is the most difficult part. 
The inclusion will be a consequence of the following two properties. 

\textit{Claim.} If $\overline{\dom(T)}=H$ and

\begin{quote}
1) if the inclusion $W_{e1}(T)\subset W_{\!e}(T)$ is strict, then $W(T)=\C$;\\
2) if $W(T)=\C$, then $W_{e4}(T)=\C$.
\end{quote}

In fact, if  $W(T)\neq\C$, then Claim 1) implies that $W_{\!e}(T)=W_{e1}(T)\subset W_{e4}(T)$; if $W(T)=\C$, then 
Corollary \ref{corlines} iv) yields $W_{\!e}(T)=\C$ and Claim 2) shows that also $W_{e4}(T)=\C$.

\textit{Proof of Claim} 1):
Suppose there exists $\lm\in W_{\!e}(T) \setminus W_{e1}(T)$. Then there is a sequence $(x_n)_{n\in\N}\subset\dom(T)$ with $\|x_n\|=1$, $x_n\stackrel{w}{\to}0$ and $\langle Tx_n,x_n\rangle\to\lm$ as $n\to\infty$ 
and $V\in\cV$ with $\lm\notin\overline{W(T|_{V^{\perp}\cap\dom(T)})}$.
We introduce a (not necessarily orthogonal) projection $P$ with $\ran(P)=V$  and $\ran(P^*)\subset\dom(T)$.
To this end, choose any basis $\{\phi_1,\dots,\phi_k\}$ of $V$. Now, since $\overline{\dom(T)}=H$, 
we can choose a biorthogonal set $\{\psi_1,\dots,\psi_k\}$ in $\dom(T)$ so that $\langle \phi_n,\psi_m\rangle=\delta_{n,m}$.
Define $P\in L(H)$~by
\begin{align*}
Px:=\sum_{n=1}^k\langle x,\psi_n\rangle \phi_n \in V, \quad x\in H.\\[-2mm]
\intertext{Then $P^2=P$ and so $P$ is a projection. 
\vspace{-2mm}Also,}
P^*x=\sum_{n=1}^k\langle x,\phi_n\rangle\psi_n\in \dom(T), \quad x\in H.
\end{align*}
Note that
\beq\label{eq.Txnxn} \langle Tx_n,x_n\rangle=\langle TP^*x_n,x_n\rangle+\langle T(I-P^*)x_n,(I-P^*)x_n\rangle+\langle T(I-P^*)x_n,P^*x_n\rangle.\eeq
Since $P^*$ and $TP^*$ are compact, we conclude $P^*x_n\to 0$ and $\langle TP^*x_n,x_n\rangle\to 0$ as $n\to\infty$.
For an arbitrary $z\in\C$ define 
$$y_n:=z P^*x_n+(I-P^*)x_n\in\dom(T), \quad n\in\N.$$ 
Note that $\|y_n\|\to 1$ and $y_n\stackrel{w}{\to}0$, and hence $\langle TP^*y_n,y_n\rangle\to 0$ as $n\to\infty$.

First assume that $(\langle T(I-P^*)x_n,(I-P^*)x_n\rangle)_{n\in\N}$ is bounded; without loss of generality it is convergent, 
with limit $\mu\in \overline{W(T|_{V^\perp \cap \dom(T)})}$.
Since, by assumption, the latter set does not contain $\lm$, we have 
$$c:=\lim_{n\to\infty}\langle T(I-P^*)x_n,P^*x_n\rangle=\lm-\mu\neq 0.$$
Moreover, for $n\in\N$,
\beq\label{eq.Tynyn}
\langle Ty_n,y_n\rangle=\langle TP^*y_n,y_n\rangle+\langle T(I-P^*)x_n,(I-P^*)x_n\rangle+\overline{z}\langle T(I-P^*)x_n,P^*x_n\rangle
\eeq
and hence $(\langle Ty_n,y_n \rangle)_{n\in\N}$ converges to $ \mu+\overline{z}c\in W_{\!e}(T)$. Since $z\in\C$ was chosen arbitrarily, 
we arrive at $W_{\!e}(T)=W(T)=\C$.

Now assume that $(\langle T(I-P^*)x_n,(I-P^*)x_n\rangle)_{n\in\N}$ is unbounded. 
Using~\eqref{eq.Txnxn}, together with $\langle Tx_n,x_n\rangle\to \lm$, $\langle TP^*x_n,x_n\rangle\to 0$  as $n\to\infty$, 
we conclude that also $(\langle T(I-P^*)x_n,P^* x_n\rangle)_{n\in\N}$ is unbounded.
Taking the difference of~\eqref{eq.Tynyn} and \eqref{eq.Txnxn} and using that 
$\langle TP^*y_n,y_n\rangle\to 0$ as $n\to\infty$, we infer that  
$(\langle Ty_n,y_n\rangle)_{n\in\N}$ is unbounded for every $z\in\C\backslash\{1\}$. The arbitrary argument of $z\in\C\backslash\{1\}$ is reflected in an arbitrary angle at which $(\langle Ty_n,y_n\rangle)_{n\in\N}$ diverges. This implies $W(T)=\C$.

\textit{Proof of Claim} 2):
Assume that $W(T)=\C$ and let $\lm\in \C$ be arbitrary.  
To show that $\lm\in W_{e4}(T)$, we shall prove, by induction, that there exist sequences 
$(e_n)_{n\in\mathbb N}$ and $(y_n)_{n\in\mathbb N}$
such that, for all $n\in\mathbb N$, we have  $\| e_n \| = \| y_n \| = 1$, the orthogonality conditions
 $y_{n+1}\in\{e_1,\ldots,e_{n}\}^\perp$ and $e_{n+1}\in\{e_1,\ldots,e_n\}^\perp$ hold, 
the condition
\[ W(T|_{\{e_1,\ldots,e_n\}^\perp\cap \dom(T)}) = {\mathbb C} \]
is satisfied and 
\[ \langle Ty_n,y_n\rangle =  \lambda, \;\;\; \left| \langle Te_n,e_n\rangle - \langle Ty_n,y_n\rangle \right| < 1/n. \]
To this end, we will employ Lemma \ref{newlem}~ii); here we will use that for every finite codimensional subspace $N\subset H$ the intersection $N\cap \dom(T)$ is dense in $N$, and hence in particular not $\{0\}$, since $T$ is densely defined, see \cite[Lemma~2.1]{MR0113146}.

At the first step of the induction we use the fact that $W(T)=\mathbb C$ to choose a unit vector $y_1\in\dom(T)$ such that
$\lambda = \langle Ty_1,y_1\rangle$. We apply Lemma \ref{newlem}~ii) with $\epsilon=1$ and $y_1$ in the r\^{o}le of $y$ to 
deduce the existence of a unit vector $e_1\in \dom(T)$ such that 
\[ W(T|_{\{e_1\}^\perp\cap \dom(T)}) = {\mathbb C}, \;\;\; \left| \langle Te_1,e_1\rangle - \langle Ty_1,y_1\rangle \right| < 1. \]
Since $W(T|_{\{e_1\}^\perp\cap \dom(T)}) = {\mathbb C}$ we can choose a unit vector $y_2$ orthogonal to $e_1$ such that 
$ \langle Ty_2,y_2\rangle = \lambda$, and the first step of the induction is complete.

Now suppose we have constructed $e_1,\ldots,e_{n-1}$ and  $y_1,\ldots,y_n$. 
Define the space $X_{n-1}:=\{e_1,\ldots,e_{n-1}\}^\perp$ of codimension $n-1$ and let
$\widehat{T} := P_{X_{n-1}}T|_{X_{n-1}\cap \dom(T)}$ where $P_{X_{n-1}}:H\to X_{n-1}$ is the orthogonal projection in $H$ onto $X_{n-1}$. We apply Lemma \ref{newlem}~ii) with $\widehat{T}$ in $X_{n-1}$ playing the 
r\^{o}le of $T$ and $y_n$ playing the r\^{o}le of $y$, together with the choice $\epsilon=\frac{1}{n}$, to deduce the existence of a unit vector 
$e_n\in\dom(T)$ orthogonal to all of $e_1,\ldots,e_{n-1}$, such that 
\begin{equation}
\label{ind} 
W(\widehat{T}|_{\{e_n\}^\perp\cap \dom(T)}) = W(T|_{\{e_1,\ldots,e_n\}^\perp\cap \dom(T)}) = {\mathbb C}, 
\end{equation}
and
\[ \left| \langle Te_n,e_n\rangle - \langle Ty_n,y_n\rangle \right| < 1/n. \]
Now \eqref{ind} allows us to choose a unit vector $y_{n+1}\in\dom(T)$, orthogonal to all of $e_1,\ldots, e_n$, such that 
$\langle \widehat{T}y_{n+1},y_{n+1}\rangle \!=\! \langle Ty_{n+1},y_{n+1}\rangle\!=\! \lambda$. The induction is thus~complete.

Finally, to show that \eqref{Weinclspecial} holds if $\overline{\dom(T)\cap\dom(T^*)}= H$ or if $W(T)\neq\C$, it remains to be proved that the inclusion 
$W_{e1}(T) \subset W_{ei}(T) = W_{\!e}(T)$, $i=2,3,4$, is an equality as well.

If $W(T)\neq\C$, this is immediate from Claim 1) above.
If $\overline{\dom(T)\cap\dom(T^*)}= H$, we will show that $W_{e2}(T) \subset W_{e1}(T)$. 
Otherwise, there would exist $\lm\in W_{e2}(T)$ and $V\in \cV$ so that $\lm \notin \overline{W(T|_{V^{\perp}\cap\dom(T)})}$. 
After a possible shift and rotation we may assume that $\re\lm<0$ and 
$$\overline{W(T|_{V^{\perp}\cap\dom(T)})}\subset \cH^+:=\{z\in\R:\,\re z\geq 0\}.$$
Since ${\rm dim }\,V\!<\!\infty$ and $\dom(T)$ is dense in $H$, 
$V^{\perp}\!\cap\dom(T)$ 
is dense in $V^\perp$ and thus, in particular, $V^{\perp}\!\cap\dom(T)\neq \{0\}$. 
Let $\mu\in\overline{W(T|_{V^{\perp}\cap\dom(T)})}$.
As in the proof of Claim~1) above, now using that $\overline{\dom(T)\cap\dom(T^*)}= H$, there exists a (not ne\-cessarily orthogonal) 
projection $P\in L(H)$ with $\ran(P)\!=\!V$ and $\ran(P^*)\!\subset\!\dom(T)\cap\dom(T^*)$.~Define
$$K_0:=-TP^*-PT(I-P^*)+\mu P P^*, \quad \dom(K_0):=\dom(T).$$
Since $\ran(K_0)\subset {\rm span}(\ran(P)\cup\ran(TP^*))$ is finite-dimensional, the operator $K_0$ has finite rank.
The assumption that $\overline{\dom(T)\cap\dom(T^*)}= H$ implies that $\overline{\dom(T^*)}= H$ and hence $T$ is closable, see Remark \ref{newrem} i). This and $\ran(P^*)\subset\dom(T)\cap\dom(T^*)$ imply that the operators $\overline{T} P^* \supset TP^* $ and $(T^* P^*)^* \supset PT$ 
are bounded and hence so is~$K_0$. 
Since $\dom(K_0)=\dom(T)$ is dense, $K_0$ is closable and its closure $K:=\overline{K_0}\in L(H)$ has finite rank as well.
Note that $\dom(T+K)=\dom(T)=\dom(K_0)$ 
 and
\begin{align*}
T+K&=T+K_0
=\big((I-P)T(I-P^*)+\mu P P^*\big)|_{\dom(T)}.
\end{align*}
For arbitrary $x\in\dom(T)$, $\|x\|=1$, we set 
\begin{align*}
u\!:=\!P^*x, \quad v\!:=\!(I-P^*)x \!\in\! \ran(I-P^*)\cap\dom(T)  
\!=\! \ran(P)^{\perp}\cap\dom(T)=\!V^{\perp}\cap\dom(T).
\end{align*}
Then 
$$\langle (T+K)x,x\rangle=\langle T v,v\rangle+\mu \|u\|^2\in\{tz+s\mu:\,z\in W(T|_{V^{\perp}\cap\dom(T)}),\,t,s\geq 0\}.$$ 
Since $\mu$ was chosen in $\overline{W(T|_{V^{\perp}\cap\dom(T)})}\subset \cH^+\!$, we conclude $\re \langle (T+K)x,x\rangle\geq 0$. 
This implies that
$\overline{W(T+K)}\!\subset\! \cH^+$ and so $\lm \in W_{e2}(T)\!\subset\!\overline{W(T+K)}\!\subset\! \cH^+$, 
%As $\lm \in W_{e2}(T)$, we deduce $\re\lm \ge 0$, 
a contradiction to $\re \lm \!<\!0$.
This proves~$W_{e2}(T) \subset W_{e1}(T)$ and hence \eqref{Weinclspecial}. 
\end{proof}

\begin{rem}\label{remequivchar}
\begin{enumerate}[label=\rm{\roman{*})},leftmargin=19pt,itemindent=0pt] 
\item  For bounded $T$ the identity $W_{\!e}(T)=W_{e1}(T)$ is not explicitly stated in~\cite[Theorem~(5.1)]{fillmore}, 
but it may be read off from the proof.

\item In the bounded case there is yet another characterization of $W_{\!e}(T)$, see~\cite[Theorem~(5.1) (5)]{fillmore}, 
$$W_{\!e}(T)=W_{e5}(T):=\big\{\lm\in\C:\,\exists\,Q\in\cQ\,\text{ with }\,Q(T-\lm)Q \text{ compact}\big\}$$
where $\cQ$ is the set of all projections $Q\!\in\! L(H)$ with ${\rm rank}\, Q \!=\! \infty$.
Note that if $T$ is un\-bounded, one has to add the condition $\ran(Q)\!\subset\!\dom(T)$ for $Q\in\cQ$ in the definition of $W_{e5}(T)$.
For the purpose of this paper this characterization does not play a~r\^{o}le. We only mention that, if $\overline{\dom(T)\cap\dom(T^*)}=H$, 
then $W_{e5}(T)\!=\!W_{ei}(T)\!=\!W_{\!e}(T)$, $i=1,2,3,4$,  also in the unbounded case.
\end{enumerate}
\end{rem}

The next observation is useful for determining the essential numerical range in concrete examples such as the following Example~\ref{exequivWe}.

\begin{lemma}\label{lemma.V0}
Let $V_0\in\mathcal V$. Then
$W_{e1}(T)=W_{e1}(T|_{V_0^{\perp}\cap\dom(T)})$.
\end{lemma}

\begin{proof}
The inclusion `$\supset$' is obvious from the definition in Theorem \ref{thmequivdefforWe}. 
The inclusion `$\subset$' follows from 
\begin{align*}
W_{e1}(T)
&=\underset{V\in\mathcal V}{\bigcap}\overline{W(T|_{V^{\perp}\cap\dom(T)})}
\subset \underset{V\in\mathcal V \atop V\supset V_0}{\bigcap}\overline{W(T|_{V^{\perp}\cap\dom(T)})}\\
&=\underset{V\in\mathcal V}{\bigcap}\overline{W(T|_{V^{\perp}\cap V_0^{\perp}\cap\dom(T)})}
=W_{e1}(T|_{V_0^{\perp}\cap\dom(T)}).
\qedhere
\end{align*}
\end{proof}

The next example shows that the strict inclusion $W_{e1}(T) \subsetneq W_{\!e}(T)$ in \eqref{Weincldd} may occur if $\overline{\dom(T)}=H$, but $\overline{\dom(T) \cap \dom(T^*)} \ne H$ and $W(T)=\C$.

\begin{example}\label{exequivWe}
Let $T_0$ be a selfadjoint operator in a Hilbert space $H$ with domain $\dom(T_0)$ and $\sigma(T_0)=\sigma_{e}(T_0)=\R$.
We perturb $T_0$ by an unbounded linear operator $S$ with $\dom(S)=\dom(T_0)$ of the form $S=Q\Phi$ where 
$\Phi : H\to\C$ is an unbounded linear functional which is $T_0$-bounded and $Q:\C\to H$, $Qz=zg$ where $g\in H$ is a fixed element.
Note that $\dom(S)$ is dense since so is $\dom(T_0)$.

Then $S$ is $T_0$-compact since $S(T_0+\I)^{-1}=Q\Phi(T_0+\I)^{-1}$ is the product of the bounded finite rank operator $Q$ with the bounded operator $\Phi(T_0+\I)^{-1}$. Hence for 
$T:=T_0+S$ with $\dom(T)=\dom(T_0)=\dom(S)$ we also have $\sigma_{e}(T)=\sigma_{e}(T_0)=\R$.

Since $\Phi$ is unbounded, $f\mapsto \langle Sf,y \rangle = (\Phi f) \langle g,y  \rangle$ is continuous if and only if $y \in \{g\}^\perp$. Thus $\dom(S^*)=\{g\}^\perp$ is not dense and so $S$ is not closable. 
Together with the fact that $S$ is densely defined it follows that $W(S)=\C$, see \cite[Thm.~V.3.4]{Kato}. 

Now we show that $\overline{\dom(T) \cap \dom(T^*)}\ne H$. This follows from $\overline{\dom(S^*)}\ne H$ if we prove that 
$\dom(T) \cap \dom(T^*)= \dom(T_0) \cap \dom(S^*)$. For the latter we use the inclusions $\dom(T_0) \cap \dom(S^*) = \dom(T) \cap \dom(S^*) \subset \dom(T)$, $T_0+S^* \subset (T_0+S)^* = T^*$ and that, for 
$y \in \dom(T) \cap \dom(T^*) = \dom(T_0) \cap \dom(T^*) $ and $x\in \dom(S)=\dom(T_0)$,
\[
  x\mapsto \langle Sx,y \rangle = \langle (T-T_0) x,y \rangle = \langle T x,y \rangle - \langle T_0 x,y \rangle  
\]
is continuous on $\dom(S)$ and hence $y\in\dom(S^*)$.

We claim that $W(T)=\C$, which implies that $W_{\!e}(T)=\C$ by Corollary \ref{corlines}~iv). Otherwise, $\overline{W(T)}$ would be contained in some closed half-plane ${\mathbb H}$. 
Since $\R=\sigma_{e2}(T)\subset \overline{W(T)}\subset {\mathbb H}$, this half-plane must be of the form ${\mathbb H}=\{z\in\C: \im z \le h\}$ or ${\mathbb H}=\{z\in\C: \im z \ge -h\}$ with $h\ge 0$. 
E.g.\ in the former case, let $h_0 > h$. Since $W(S)=\C$, there exists $f\in\dom(S)=\dom(T)$, $\|f\|=1$, so that $\langle Sf,f \rangle= \I h_0$. But then $\langle Tf,f\rangle =\langle T_0f,f\rangle +\langle Sf,f\rangle$ 
and since $T_0$ is selfadjoint $\im \langle Tf,f\rangle = \im \langle Sf,f \rangle = h_0 > h$, a contradiction to $W(T)\subset {\mathbb H}$.

Applying Lemma \ref{lemma.V0} with $V_0={\rm span} \{g\}$, we obtain $W_{e1}(T)=W_{e1}(T|_{V_0^\perp \cap \dom (T)}) \!=\! W_{e1}(T_0|_{V_0^\perp \cap \dom (T)}) \!=\! W_{e1}(T_0)$. Since $W(T_0)\!=\!\sigma(T_0)\!=\!\R$, Corollary \ref{corlines} implies $W_e(T_0)\!=$ $W(T_0)\!=\!\R$. Theorem \ref{thmequivdefforWe} for $T_0$ shows $W_{e1}(T)\!=\!W_{e1}(T_0)\!=\!W_e(T_0)\!=\!\R$.

Altogether, in this example, $\overline{\dom(T)\cap\dom(T^*)}\ne H$, $W(T)=\C$ and 
\beq
\label{nikolaus}
  \R = W_{e1}(T) \subsetneq W_{e4}(T) = W_{e2}(T) = W_{e3}(T) =  W_{\!e}(T) = \C,
\eeq
which shows that the conditions for \eqref{Weinclspecial} in Theorem \ref{thmequivdefforWe} are necessary.

A concrete example of operators $T_0$ and $\Phi$ as above is $T_0 f \!=\! \I f'$ in $L^2(\R)$ with $\dom(T_0)\!=\!H^1(\R)$ and $\Phi=\delta_0$, 
i.e.\ $\Phi f := f(0)$ with $\dom (\Phi) = H^1(\R)$, so that $Sf\!=\!f(0) g$ with some fixed $g\in L^2(\R)$. In this case $Tf=\I f' + f(0) g$ with $\dom (T)=H^1(\R)$ is an example for the above abstract model
for which \eqref{nikolaus} holds.
\end{example}

It is well-known that for a non-selfadjoint operator $T$ in a Hilbert space there are several different, and in general not equivalent, 
definitions of essential spectrum, denoted by $\sigma_{ek}(T)$, $k=1,\dots5$, see e.g.\ \cite[Chapter~IX]{edmundsevans}, which satisfy the inclusions 
$$
\sigma_{e1}(T) \subset \sigma_{e2}(T) \subset \sigma_{e3}(T) \subset \sigma_{e4}(T) \subset \sigma_{e5}(T).
$$ 
By Proposition \ref{propWeclosed-convex} we already know that for the essential spectrum 
$\sigma_{e2}(T)= \sigma_e(T)$ from \cite[Chapter~IX]{edmundsevans}, which we use here, 
$ {\rm conv}\, \sigma_e(T)  \subset W_{\!e}(T)$, and hence
\begin{equation} 
\label{sigmae2incl}
{\rm conv}\, \sigma_{e1}(T) \subset {\rm conv}\, \sigma_{e2}(T) = {\rm conv}\, \sigma_{e}(T) \subset W_{\!e}(T).
\end{equation}
The following remark collects the inclusions for all the essential spectra.
 
\begin{rem}\label{remrelcompchar} 
\begin{enumerate}
\item[i)] If $\dom(T)\cap\dom(T^*)$ is a core of $T^*$, then 
$${\rm conv}\,(\sigma_e(T)\cup\sigma_e(T^*)^*)\subset W_{\!e}(T).$$
\item[ii)] 
If $T$ is closed, then $\sigma_{e3}(T)=\sigma_e(T)\cup\sigma_e(T^*)^*$ and hence if, in addition, 
$\dom(T)\cap\dom(T^*)$ is a core of $T^*$, then
$$ {\rm conv}\,\sigma_{e3}(T)={\rm conv}\,\big(\sigma_e(T)\cup\sigma_e(T^*)^*\big)  \subset W_{\!e}(T).$$
\item[iii)] 
If $T$ is closed, %and $\overline{\dom(T)}=H$, 
then $\sigma_{e4}(T) \subset \hspace{-3mm} {\displaystyle \bigcap\limits_{K\in L(H), \atop K\,\text{compact}}} \hspace{-2mm}\sigma(T+K)$ and \vspace{-4mm}
if, in addition, $\sigma(T) \!\subset\! \overline{W(T)}$, \vspace{3mm}then 
$$ {\rm conv}\,\sigma_{e4}(T) \subset W_{e3}(T) = W_{\!e}(T);$$
in this case $\sigma_{e5}(T)=\sigma_{e4}(T)$ and hence also
$${\rm conv}\, \sigma_{e5}(T) \subset W_{\!e}(T).$$
\end{enumerate}
\end{rem}

\begin{proof}
i) By Proposition \ref{propWeclosed-convex}, see also \eqref{sigmae2incl}, it suffices to show that $\sigma_e(T^*)^*\!\subset\! W_{\!e}(T)$. If $\lm\!\in\!\sigma_e(T^*)^*$, there exist  $x_n\!\in\!\dom(T^*)$, $n\in \N,$ with $\|x_n\|=1$, $x_n\stackrel{w}{\to}0$ and $\|(T^*-\overline{\lm})x_n\|\to 0$ as $n\to\infty$. 
Since  $\dom(T)\cap\dom(T^*)$ is a core of $T^*$, we can construct  $\widetilde x_n\in\dom(T)\cap \dom(T^*)$, $n\in \N,$ with $\|\widetilde x_n\|=1$, $\widetilde x_n\stackrel{w}{\to}0$ and $\|(T^*-\overline{\lm}) \widetilde x_n\|\to 0$ as $n\to\infty$.  
This implies $\langle (T-\lm)\widetilde x_n,\widetilde x_n\rangle=\langle \widetilde x_n, (T^*-\overline{\lm})\widetilde x_n\rangle\to 0$ as $n\to\infty$ and hence $\lm\in W_{\!e}(T)$.

ii) The claim follows from the stated formula for $\sigma_{e3}(T)$ for closed $T$, which is a consequence of the Closed Range Theorem, see \cite[Thm.\ I.3.7]{edmundsevans}, and from claim i).

iii) The claim will follow from the stated inclusion for $\sigma_{e4}(T)$ for closed $T$, see the first part of the proof of \cite[Thm.\ IX.1.4]{edmundsevans} 
and note that $T$ need not have dense domain for the inclusion ``$\subset$'' therein, and from the equality $W_{\!e}(T)=W_{e3}(T)$ by Theorem \ref{thmequivdefforWe} if we show that the spectral inclusion $\sigma(T) \subset \overline{W(T)}$ implies the spectral inclusion $\sigma(T+K) \subset \overline{W(T+K)}$. 

To show the latter, suppose first that $\overline{W(T)}=\C$. 
%\marginpar{\tiny\ct{use notation from intro to save space}}
Then also $\overline{W(T\!+\!K)}=\C$, and so the claim is immediate, since otherwise $\overline{W(T+K)}$ were contained in a half-plane and hence, since $K$ is bounded, also 
$\overline{W(T)}\!\subset\! \overline{W(T\!+\!K)}\!+\!\overline{B_{\|K\|}(0)}$ would be contained in a half-plane, a contradiction. 
If $\overline{W(T)}\!\ne\!C$, the convexity of $W(T)$ implies that the complement $\C\backslash \overline{W(T)}$ consists either of one component or of two components in which case $\overline{W(T)}$ is a strip. 
Then the same is true for $\overline{W(T\!+\!K)}$ since $K$ is bounded and so $W(T\!+\!K)\!\subset\!W(T)\!+\!\overline{B_{\|K\|}(0)}$. A Neumann series ar\-gu\-ment and the resolvent estimate 
$\|(T-\lambda)^{-1}\| \!\le\! 1/{\rm dist}\,(\lambda, W(T))$, $\lambda \!\notin\! \overline{W(T)}$, yield $\sigma(T+K)\!\subset\! \overline{W(T)}\!+ \!\overline{B_{\|K\|}(0)}$. 
The latter yields that in each component of $\C\setminus \overline{W(T\!+\!K)}$ there exists at least one point in $\rho(T+K)$ which implies $\sigma(T\!+\!K)\subset \overline{W(T\!+\!K)}$.
\end{proof}

For a bounded selfadjoint operator $T$, the essential numerical range is the convex hull of the essential spectrum, 
$W_{\!e}(T)={\rm conv}\,\sigma_e(T)$, see \cite[Corollary~5.1]{salinas}. 

An analogous result for unbounded selfadjoint operators 
does not hold; it may even happen that $\sigma_e(T)=\emptyset$ but $W_{\!e}(T)=\R$. 
In order to formulate the result for arbitrary selfadjoint operators,
we need the notion of extended essential spectrum. 

Note that there are different notions of the latter for closed operators using the one-point compactification of $\C$ or $\R$, see \cite{MR0291846}, 
and for selfadjoint operators using the two-point compactification of $\R$,  \cite{levitin}, which is needed here. 

\begin{definition}\label{defextended}
If $T$ is selfadjoint, we define the \emph{extended essential spectrum} $\widehat{\sigma}_e(T) \subset \R \cup \{+\infty,-\infty\}$ of $T$ 
as the set $\sigma_e(T)$ with $+\infty$ and/or $-\infty$ added if $T$ is unbounded from above and/or from below, and as $\sigma_e(T)$ if $T$ is bounded.
\end{definition}

\begin{theorem}\label{thm.selfadjoint}
If $\,T$ is unbounded and selfadjoint, then
$$W_{\!e}(T)={\rm conv}\big(\widehat{\sigma}_e(T)\big)\backslash\{\pm\infty\}.$$
\end{theorem}

\begin{proof}
If $T$ is not semibounded, then $W(T)\!=\!\R$ and thus $W_{\!e}(T)\!=\!W(T)\!=\!\R$ by Corollary \ref{corlines}.  
Since ${\rm conv}\big(\widehat{\sigma}_e(T)\big) \!=\! \R\cup\{\pm\infty\}$ by Definition \ref{defextended}, the claim follows.

Now suppose that $T$ is semibounded, say bounded from below; if $T$ is bounded from above, we consider $-T$. Then $s_e\!:=\!\inf \widehat \sigma_e(T)\!=\! \inf \sigma_e(T) \in \R \cup \{+\infty\}$. 
If $s_e< +\infty$, we have $s_e \in W_{\!e}(T)\subset W(T)$, $(s_e,+\infty)\subset W(T)$ and hence $(s_e,+\infty)\subset W_{\!e}(T)$ 
by Proposition \ref{proplinesinWeT}. Therefore the claim is proved if we show that $(-\infty,s_e) \cap W_{\!e}(T)=\emptyset$. Let $\lambda < s_e = \inf \sigma_e(T)$. 
Then $(\lambda,\lambda+\epsilon)\subset \rho(T)$ for some $\epsilon>0$.
If $E_T(\Delta)$ denotes the spectral projection of $T$ corresponding to some Borel set $\Delta\subset \R$, 
we have $\dim \ran (E_T((-\infty,\lambda+\epsilon))) < \infty$,  
$K:=(-T +s_e) E_T((-\infty,\lambda+\epsilon))$ is compact and $T+K \ge \lambda+\epsilon$. Hence
Theorem~\ref{thmequivdefforWe}~ii) yields that $W_{\!e}(T)=W_{e2}(T)\subset \overline{W(T+K)} \subset [\lambda+\epsilon,\infty)$ which implies
$\lm \notin W_{\!e}(T)$.
\end{proof}

The definition of the essential numerical range of a linear operator $T$ involves its quadratic form $t[f]:=\langle Tf,f\rangle$, $\dom(t):=\dom(T)$. This motivates the following definition.

\begin{definition}\label{defWesesq}
Let $t$ be a sesquilinear form with $\dom(t)\subset H$.
Define the \emph{essential numerical range} of~$t$ by 
%\marginpar{\tiny\ct{I created a mess with $T$ in many places for $t$ with a replace command, please check!}}
$$
W_{\!e}(t):= \left\{\lm\in\C:\,\exists\, (x_n)_{n\in\N}\subset\dom(t) \text{ with }\|x_n\|=1,\,x_n\stackrel{w}{\to}0,\,t[x_n]\to \lm\right\}.
$$
\end{definition}

Clearly, if $t$ is the quadratic form of a linear operator $T$, then $W_{\!e}(t)\!=\!W_{\!e}(T)$.
Analogously as in Proposition~\ref{propWeclosed-convex} one may show that $W_{\!e}(t)$ is closed and convex.~For 
an extension $\widetilde t$ of $t$ we have $W_{\!e}(t)\!\subset\! W_e(\widetilde t)$, and equality prevails if $\dom(t)$ is a core of~$\widetilde t$.

\begin{theorem}\label{thm.form}
Let $\,T$ be a linear operator with associated quadratic form $t$, and let $t^*$ be the adjoint form of $t$. 
\begin{enumerate}
\item[\rm i)] %If  $t$ is the quadratic form associated with $T$, 
Then, with the quadratic forms $\re\,t:= \frac 12 (t+t^*)$, $\im\,t:= \frac 1{2\I} (t-t^*)$,
\begin{align*}
\re \,W_{\!e}(T)&\subset W_e(\re\, t),\quad
\im \,W_{\!e}(T)\subset W_e(\im\, t).
\end{align*}
\item[\rm ii)]
If $\,T$ is $m$-sectorial with semi-angle $<\pi/2$ and $\re T$ is the selfadjoint operator induced by the $($symmetric non-negative$)$ form $\re\,t$, then
\beq\label{eq.msect}
\re\,W_{\!e}(T)=W_e(\re T)={\rm conv}(\widehat \sigma_e(\re T))\backslash \{\infty\};
\eeq
in particular, $W_{\!e}(T)=\emptyset$ if $T$ has compact resolvent.
\end{enumerate}
\end{theorem}

\begin{proof}
i)
The claim follows from the fact that if $(t[x_n])_{n\in\N}$ is convergent, then so are $((\re\,t)[x_n])_{n\in\N}$ and $((\im\,t)[x_n])_{n\in\N}$.

ii)
By the assumption on $T$, the associated sequilinear form $t$ is closed and sectorial with semi-angle $<\pi/2$
and there exists a non-negative selfadjoint operator $\re T$ associated with the symmetric non-negative  
quadratic form $h:=\frac 12 (t+t^*)$ with $\dom(h)=\dom(t)=\dom(t^*)$, see  \cite[Sect.\ 3]{MR0138005}.
The inclusion $\re\,W_{\!e}(T)\subset W_e(\re t)=W_e(\re T)$ follows from claim~i). 
Conversely, let $\lm \in W_e(\re T)$. Then there exists $(x_n)_{n\in\N}\subset\dom(\re T)\subset\dom(t)$ with $\|x_n\|=1$, $x_n\stackrel{w}{\to}0$ and
$$\re t[x_n] =\langle \re T x_n,x_n\rangle \longrightarrow\lambda\in W_e(\re T);$$
in particular, $(\re t[x_n])_{n\in\N}$ is bounded. 
Since $t$ is sectorial, this implies that $(\im t[x_n])_{n\in\N}$ is bounded and, thus, has a convergent subsequence.
Hence $(t[x_n])_{n\in\N}$ has a convergent subsequence whose limit has real part $\lm\in\re\,W_{\!e}(T)$.
Since $\dom(T)$ is a core of $t$ by~\cite[Theorem~VI.2.1]{Kato}, we obtain $\lm\in\re\,W_{\!e}(T)=\re\,W_{\!e}(T)$.

The second equality in~\eqref{eq.msect} follows from Theorem~\ref{thm.selfadjoint} since $\re T$ is selfadjoint.
If $T$ has compact resolvent, then so has $\re T$, see~\cite[Theorem~VI.3.3]{Kato}, and hence ${\rm conv}(\widehat \sigma_e(\re T))\backslash \{\infty\} =\emptyset$ because $\re\,T$ is non-negative.
\end{proof}

\begin{rem}
i) As a consequence of Theorem \ref{thm.form}~i), we obtain the inclusion
\begin{align*}
 W_{\!e}(T)=\!\!\!\underset{\phi\in[0,\pi/2)}{\bigcap}\!\!\!\e^{-\I\phi}W_e\left(\e^{\I\phi}T\right)
 \subset\!\!\!\underset{\phi\in[0,\pi/2)}{\bigcap}\!\!\!\e^{-\I\phi}\Big(W_e\left(\re\left(\e^{\I\phi}t\right)\right)+\I\,W_e\left(\im\left(\e^{\I\phi}t\right)\right)\Big).
\end{align*} 
ii) Note that, while for a bounded linear operator $T$, the real and imaginary part can be defined by the formulas $\frac 12 (T+T^*)$ and $\frac 1{2\I} (T-T^*)$, respectively, 
%\marginpar{!!! {\footnotesize forms? ask Arlinskii!}}
this is not always possible if $T$ is unbounded since $\dom(T)\cap \dom (T^*)$ may not be large enough; 
even if $T$ is m-sectorial with semi-angle $<\pi/2$ the operator $\frac 12 (T+T^*)$ may not be selfadjoint and hence it need not coincide with $\re T$ as defined above 
via forms, see \cite[Sect.~3]{MR0138005}. 
\end{rem}

The last claim in Theorem \ref{thm.form} is sharp, i.e.\ there are $m$-accretive, 
non-sectorial operators with compact resolvent and $W_{\!e}(T)\ne \emptyset$, as the following example illustrates. 
%We illustrate this \cb{effect} for Schr\"odinger operators \cb{with complex potentials}. 

\begin{example}[Schr\"odinger operators with complex potentials]

\begin{itemize}
\item[\phantom{a)}] 
\end{itemize}

\vspace{1mm}

a) Consider a potential $Q\in L_{\rm loc}^1(\R^d)$ with 
\beq\label{eq.div}
|Q(x)|\tolong \infty, \quad |x|\to\infty.
\eeq
If $Q$ is sectorial, then, by~\cite[Proposition~2.2]{BST}, the quadratic form 
$$t[f]:=\|\nabla f\|^2+\langle Q f,f\rangle, \quad \dom(t):=\{f\in H^1(\R^d):\,Q|f|^2\in L^2(\R^d)\},$$
is closed, densely defined and sectorial in $L^2(\R^d)$, and the $m$-sectorial operator $T$ uniquely determined by $t$ has compact resolvent.
Thus $W_{\!e}(T)=\emptyset$ by Theorem~\ref{thm.form}~ii).

b) If the potential is not sectorial but only accretive, then $W_{\!e}(T)\ne \emptyset$ is possible, even under the assumption~\eqref{eq.div}.
As an example in dimension $d=1$, consider the complex Airy \vspace{-2mm} operator
$$T:=-\frac{\rd^2}{\rd x^2}+Q, \quad Q(x):=\I x, \quad x\in \R,$$ 
in $L^2(\R)$. Here we will show that
$$ W_{\!e}(T) = \overline{W(T)}=\{\lm\in\C:\,\re\lm\geq 0\}. $$
%and hence equality holds everywhere.
The inclusions ``$\subset$'' are obvious. Since $W_{\!e}(T)$ is closed, it remains to be proved that
$\{\lm\in\C:\,\re\lm > 0\}\subset W_{\!e}(T)$. Let $\lm=u+\I v\in \C$ with $u=\re \lm> 0$ be arbitrary. 
%For this $\lm$ we construct a sequence $(f_n)_{n\in\N}\subset\dom(T)$ with $\|f_n\|=1$, $f_n\stackrel{w}{\to}0$ and
%$\langle Tf_n,f_n\rangle\to \lm$, hence $\lm\in W_{\!e}(T)$; then the claim follows from the closedness of $W_{\!e}(T)$.
If $\varphi\in C_0^{\infty}(\R)$ is an even or odd function with ${\rm supp}\, \varphi\subset [-1,1]$ and $\|\varphi\|^2=1/2$,  $\|\varphi'\|^2=u/2$, we 
define $(f_n)_{n\in\N}\subset\dom(T)$ by 
$$
f_n(x):=\begin{cases} \varphi(x-(-n))+\varphi(x-(n+2|v|)), &v\geq 0,\\ \varphi(x-(-n-2|v|))+\varphi(x-n), &v<0,\end{cases} \quad x\in\R,\quad n\in\N.
$$
Then it is not difficult to check that $\langle Q f_n,f_n\rangle =\I v$, $\|f_n\|=1$, $f_n\stackrel{w}{\to} 0$ as $n\to\infty$~and
$$
\langle Tf_n,f_n\rangle  =\|f_n'\|^2+\langle Q f_n,f_n\rangle=2\|\varphi'\|^2+\I v=u+\I v=\lm, \quad n\in\N,
$$
which implies that $\lm\in W_{\!e}(T)$, as required.

Note that the above arguments also prove directly that the closure of the numerical range of the complex Airy operator is the closed right half-plane, 
while earlier proofs rely on estimates of the resolvent norm, comp.\ \cite[Sect.~3.1]{MR2807108}.
\end{example}

\section{Perturbation results}\label{sectionperturb}

While the essential spectrum of an unbounded linear operator is invariant both under compact and relatively compact perturbations, 
we will see that, in general, the latter is not true for the essential numerical range. 

First we prove that the essential numerical range $W_{\!e}(T)$ and all other, possibly not coincident sets $W_{ei}(T)$, $i=1,2,3,4$, 
are invariant under compact perturbations.

\begin{prop}\label{propWeprop}
For every compact $K\in L(H)$ we have
\begin{align*}
&W_{ei}(T+K)=W_e(T+K)=W_{\!e}(T)=W_{ei}(T), \quad i=2,3, \\
&W_{e4}(T+K)\!=\!W_{e4}(T), \quad W_{e1}(T+K)\!=\!W_{e1}(T), 
\end{align*}
even if the latter two sets are not equal to $W_{\!e}(T)$.
\end{prop}

\begin{proof}
For $W_{\!e}(T)$ the claim follows readily from Definition \ref{defWe} since compact operators map weakly convergent sequences to strongly convergent ones. Alternatively, as for $W_{e2}(T)$ it follows from the equality $W_{\!e}(T)=W_{e2}(T)=W_{e3}(T)$ by Theorem~\ref{thmequivdefforWe} since the claim for $W_{e3}(T)$ is obvious from its definition. For $W_{e4}(T)$ the claim follows from the property that $\langle Ke_n,e_n\rangle \to 0$
as $n\to\infty$ for a compact operator $K$ and an arbitrary orthonormal system $(e_n)_{n\in\N} \subset H$.

In order to show that $W_{e1}(T)=W_{e1}(T+K)$ for every compact $K\in L(H)$, it suffices to prove that $W_{e1}(T+K) \subset W_{e1}(T)$; then the reverse inclusion follows from $W_{e1}(T)=W_{e1}((T+K)-K) \subset W_{e1}(T+K)$. Let $\lambda \in W_{e1}(T+K)$, and suppose that $V\in \cV$ and $n\in\N$ are arbitrary. 
By Lacey's theorem, see \cite{Lacey-1963-PhD}, \cite[Thm.\ III.2.3]{MR0200692},
there exists a finite codimensional, hence closed, subspace $W_n \subset H$ such that $\|K|_{W_n}\| \le \frac 1{2n}$. 
By Lemma \ref{lemma.V0}, since $V_n:=W_n^\perp$ is finite dimensional, we have $W_{e1}(T+K)=W_{e1}((T+K)|_{W_n \cap \dom (T)})$
and hence, by definition of the latter, $\lambda \in \overline{W((T+K)|_{V^\perp\cap W_n \cap \dom(T)})}$.
This implies that there exists $x_n \in V^\perp\cap W_n \cap \dom(T)$, $\|x_n\|=1$, with $|\lambda - \langle(T+K)x_n,x_n \rangle| \le \frac 1{2n}$.
Since $x_n\in W_n$ and $x_n \in V^\perp\cap \dom(T)$, it follows that
$$
|\lambda - \langle Tx_n,x_n \rangle| \le |\lambda - \langle (T+K)x_n,x_n \rangle | + \underbrace{|\langle Kx_n,x_n \rangle|}_{\le \|K|_{W_n}\|\le \frac 1{2n}} 
\vspace{-4mm}
\le \frac 1n, \quad n\in\N,
$$
which proves that $\lambda \in W_{e1}(T)$ as required.
\end{proof}

The following remark shows that, unlike the essential spectrum, 
the essential numerical range is not invariant under relatively compact perturbations in general.

\begin{rem}
\label{rem:Wetilde}
For an unbounded operator $T$, Example~\ref{xmas} below shows that the definition of the essential numerical range is in general \emph{not} equivalent to 
$$\widetilde W_{\!e}(T):=\underset{K\, T\text{-compact}}{\bigcap}\overline{W(T+K)},$$
not even if $T$ is selfadjoint and thus $W_{\!e}(T)\!=\!W_{ei}(T)$, $i\!=\!1,2,3,4$, see Theorem~\ref{thmequivdefforWe}. 
In general, $\widetilde W_{\!e}(T) \subset W_{e3}(T)=W_{\!e}(T)$ by Theorem~\ref{thmequivdefforWe} since every com\-pact operator is $T$-compact.
Equality only holds under some additional condition: 
\begin{align}
\label{Wetilde-equiv}
W_{\!e}(T)=\widetilde W_{\!e}(T) \quad \iff \quad
W_{\!e}(T) \subset W_{e}(T+\widetilde{K}) \ \mbox{for all $T$-compact } \widetilde{K};
\end{align}
if $T$ is closable, \eqref{Wetilde-equiv} holds with ``$\subset$'' replaced by ``$=$'' on the right hand side. 
In particular, $W_{\!e}(T) \!\ne\! \widetilde W_{\!e}(T)$ if $T$ has compact resolvent, $W(T) \!\ne\! \C$ and $W_{\!e}(T)\!\ne\! \emptyset$.
\end{rem}

\begin{proof}
``$\Longrightarrow$'' 
Suppose first that $W_{\!e}(T)=\widetilde W_{\!e}(T)$ and let $\widetilde K$ be an arbitrary $T$-compact operator. Then 
\begin{align*}
W_{\!e}(T) &= \underset{K\, T\text{-compact}}{\bigcap} \overline{W(T+K)}\; =\underset{K\, T\text{-compact}}{\bigcap}\overline{W(T+\widetilde{K}+K)} \\
  &\subset \hspace{2mm}\underset{K \in L(H) \atop K\,\text{compact}}{\bigcap}\hspace{2mm}\overline{W(T+\widetilde{K}+K)}  = W_{e3}(T+\widetilde{K}) 
	=W_e(T+\widetilde{K}),
\\[-7mm]
\end{align*}
where Theorem~\ref{thmequivdefforWe} was used for $T\!+\!\widetilde K$ in the last step. 
If $T$ is closable, then every $T$-compact $\widetilde K$ is also $(T+\widetilde K)$-compact, see \cite[Prop.~III.8.3]{edmundsevans}.
Using what was already proved, we also obtain the reverse inclusion $W_e(T+\widetilde{K}) \subset W_e(T+\widetilde{K}-\widetilde K)=W_{\!e}(T)$.   

``$\Longleftarrow$''  
Conversely, if $W_{\!e}(T) \!\subset\! W_{e}(T+\widetilde{K})$ for all $T$-compact $\widetilde{K}$, then 
\begin{equation}
\begin{array}{rl}
W_{\!e}(T) \; & \subset  \hspace{-2mm} \underset{\widetilde K\, T\text{-compact}}{\bigcap} \hspace{-2mm}  W_{e}(T+\widetilde{K}) \\
&\subset \hspace{-2mm}  \underset{\widetilde K\, T\text{-compact}}{\bigcap} \hspace{-2mm}  \overline{W(T+\widetilde{K})} \
%\underset{\ch{K \in L(H) \atop K \,\text{compact}}}{\bigcap} \hspace{-2mm}  W_e(T+K) \\
\subset \hspace{-2mm}  \underset{K \in L(H) \atop K \,\text{compact}}{\bigcap} \hspace{-2mm} \overline{W(T+K)} = W_{e3}(T) = W_{\!e}(T),
\end{array} 
\vspace{-1mm}
\end{equation}
where Theorem~\ref{thmequivdefforWe} was used for $T$ in the last step. This shows that equality prevails everywhere and hence, in particular, $W_{\!e}(T)=\widetilde W_{\!e}(T)$.

% Combining (\ref{eq:1}) with (\ref{eq:2}) shows that the inclusion in (\ref{eq:1}) must be an equality, proving the
% second statement.
To prove the last claim, we use that $K=\lambda I$ is $T$-compact and $T$ is closed if $T$ has compact resolvent, and hence the right hand side of 
\eqref{Wetilde-equiv} %(``with ``$\subset$'' replaced by ``$=$'') 
only holds if
\begin{equation}
\label{Wetransl}
W_{\!e}(T) = W_e(T+\lm I)=W_{\!e}(T)+\lm, \quad \lm\in\C,
\end{equation}
%\cb{It is not difficult to check that \eqref{Wetransl}}
which necessitates $W_{\!e}(T)=\emptyset$ or $W_{\!e}(T)=\C$; the latter is equivalent to $W(T)=\C$ by Corollary~\ref{corlines} iv).
\end{proof}

\begin{example}
\label{xmas}
Let $T$  be a selfadjoint non-semibounded operator with compact 
resolvent. 
Then $W(T)=\R$ and thus, by Corollary~\ref{corlines}~i), also $W_{\!e}(T)=W(T)=\R$. 
So in this case $W(T)\ne \C$ and $W_{\!e}(T)\ne \emptyset$, whence $W_{\!e}(T) \ne \widetilde W_{\!e}(T)$ by the last claim of Remark \ref{rem:Wetilde}.
\end{example}

The next proposition shows that the essential numerical range of a selfadjoint operator $T$  remains invariant under 
%\marginpar{\tiny\ct{symmetric should be first, I think}}
symmetric relatively compact  perturbations. Further stability results  for $W_{\!e}(T)$ for non-selfadjoint operators are given~below. 

\begin{prop}\label{propsapert}
Let $T$ be selfadjoint. If $S$ is symmetric and $T$-compact, then $W_{\!e}(T)=W_e(T+S)$.
\end{prop}

\begin{proof}
The assumptions on $S$ imply that  
$\sigma_e(T)=\sigma_e(T+S)$, that $S$ is $T$-bounded with relative bound $0$ and hence 
%By the classical results~\cite[Problem~IV.1.2, Theorem~V.4.11]{Kato}, 
$T+S$ is selfadjoint, and that $T+S$ is bounded from below/above whenever $T$ is, see \cite[Thm.~IX.2.1, Cor.~II.7.7]{edmundsevans},
\cite[Thm.~V.4.3, Thm.~V.4.11]{Kato}.
Now the claim follows from Theorem~\ref{thm.selfadjoint}.
\end{proof}

In the following result both the unperturbed operator $T$ and the perturbation~$S$ may be non-selfadjoint, but we 
assume that they admit decompositions into ``real and imaginary parts''\!, which need not hold for unbounded operators~in general.
%\marginpar{\tiny\ct{there was still the real and imaginary parts mistake mentioned!}}

\begin{theorem}\label{thmABUV}
Let $T=A+\I B$ and $S=U+\I V$ with symmetric operators $A$,~$B$ and $U$, $V$ in $H$ such that
\emph{one} of the following holds:
\begin{enumerate}
 \item[\rm (i)] \hspace{0.5mm}$A$ is selfadjoint and semibounded, $U$, $V$ are $A$-compact, or
 \item[\rm (ii)] $B$ is selfadjoint and semibounded, $U$, $V$ are $B$-compact, or
 \item[\rm (iii)] $A$, $B$ are selfadjoint and semibounded, $U$ is $A$-compact and $V$ is $B$-compact.
\end{enumerate}
\noindent Then $W_{\!e}(T)=W_e(T+S)$.
\end{theorem}

\begin{rem}
\label{rem:sectorial-decompose}
Theorem \ref{thmABUV} does not hold if $S$ does not decompose and we only assume that 
$S$ is $A$-compact in (i) or $S$ is $B$-compact in (ii), see Example \ref{Ex2} below.
\end{rem}

\begin{proof}[Proof of Theorem {\rm \ref{thmABUV}}.]
It is sufficient to consider the cases (i) and (iii); in case (ii) the operators $\I T$ and $\I S$ satisfy the assumptions of (i). 
 
``$\supset$":  Let $\lm\in W_e(T+S)$.
Then there exists a sequence $(x_n)_{n\in\N}\subset\dom(T)$ with $\|x_n\|=1$, $x_n\stackrel{w}{\to} 0$ and $\langle (T+S) x_n,x_n\rangle \to \lm$,
i.e.
\beq \label{eqreimconv}
 \langle A x_n,x_n\rangle + \langle U x_n,x_n\rangle \longrightarrow \re\lm, \ \ 
 \langle B x_n,x_n\rangle + \langle V x_n,x_n\rangle \longrightarrow \im\lm, \quad n\to\infty.
\eeq

First we show that all four sequences occurring in \eqref{eqreimconv} are bounded.
To this end, suppose that $(\langle A x_n,x_n\rangle)_{n\in\N}$ is unbounded,
i.e.\ there exists an infinite subset $I\subset\N$ such that $\langle A x_n, x_n\rangle\!\to\!\infty$ as $n\!\in\! I$, $n\!\to\!\infty$.
In both cases (i) and (iii), $A$ is selfadjoint and semibounded and 
$U$ is $A$-bounded with relative bound~$0$, and~thus \cite[Theorem~VI.1.38]{Kato} 
(or, more generally, \cite[Thm.~3.2 (i)]{MR2501173}) 
implies $\langle (A+U) x_n, x_n\rangle \to\infty$ as $n\in I$, $n\to\infty$, a contradiction to~\eqref{eqreimconv}.
Thus $(\langle A x_n,x_n\rangle)_{n\in\N}$ is bounded, and hence so is $(\langle U x_n,x_n\rangle)_{n\in\N}$ 
by \cite[Theorem~VI.1.38]{Kato}.
In case (iii), the proof that $(\langle B x_n,x_n\rangle)_{n\in\N}$ and $(\langle V x_n,x_n\rangle)_{n\in\N}$ are bounded is analogous.
In case (i), $V$ is $A$-bounded with relative bound~$0$ and hence the boundedness of $(\langle A x_n,x_n\rangle)_{n\in\N}$ implies the 
boundedness of $(\langle V x_n,x_n\rangle)_{n\in\N}$ by \cite[Theorem~VI.1.38]{Kato}. Now \eqref{eqreimconv} yields that $(\langle B x_n,x_n\rangle)_{n\in\N}$ is bounded.
%Similarly one may show that $(\langle B x_n,x_n\rangle)_{n\in\N}$ is bounded.

The boundedness of the four sequences in \eqref{eqreimconv} implies that 
there exist an infinite subset $J\subset\N$ and $\gamma$, $\delta\in\R$ such that
$$
\begin{array}{ll}
\langle U x_n,x_n\rangle \longrightarrow \gamma, \quad &\langle A x_n,x_n\rangle \longrightarrow \re \lm-\gamma=:\alpha, \\[1mm]
\langle V x_n,x_n\rangle \longrightarrow \delta \quad &\langle B x_n,x_n\rangle \longrightarrow \im \lm-\delta=:\beta,
\end{array} \qquad n\in J, \quad n\to\infty.$$
Now suppose that $\lm\!\notin\! W_{\!e}(T)$. Since $\alpha\!+\!\I \beta \in W_{\!e}(T)$, this would imply $\gamma \!\neq\! 0$ or $\delta \!\neq\! 0$.

First we assume that $\gamma\neq 0$. In both cases (i) and (iii), we have 
$W_e(A+t U)=W_e(A)$ by Proposition~\ref{propsapert}. Then, for any $ t\in\R$, 
$$ 
\alpha + t \gamma =\lim_{n\in J\atop n\to\infty} \langle (A+t U)x_n,x_n\rangle \in W_e(A+t U) = W_e(A),
\vspace{-2mm}
$$
%Since $W_e(A+t U)=W_e(A)$ by the assumption~(b) and Proposition~\ref{propsapert}, we obtain 
which implies that $W_e(A)=\R$, a contradiction to the semiboundedness of $A$ assumed in (i) and (iii).

Now assume that $\delta \neq 0$. 
In case (i), in the same way as above, we arrive at $\gamma+t \delta\in W_e(A+t V)=W_e(A)$ 
for every $t\in\R$ implying the contradiction $W_e(A)=\R$.
%By the assumption~(d), we have $W_e(B)\neq\R$ or $V$ is $A$-compact.
%If $W_e(B)\neq\R$, by proceeding as above, we obtain 
In case (iii), in the same way as above, we conclude that 
$\beta+t\delta \in W_e(B+t V)=W_e(B)$ for any $t\in\R$ and hence $W_e(B)=\R$, a contradiction 
the semiboundedness of $B$ assumed in (iii). 

Altogether, we have shown that $\gamma=\delta=0$ and hence $\lm\in W_{\!e}(T)$.

``$\subset$": The  reverse inclusion follows by applying the first part of the proof to the operators $T'\!=T+S$ and $S'\!=-S$.
Here, in case (iii) we have to note that $A+U$ and 
$B+V$ are selfadjoint by~\cite[Theorem~V.4.3]{Kato} and that $U$ is $(A+U)$-compact,
$V$ is $(B+V)$-compact by \cite[Prop.~III.3.8]{edmundsevans}. In case (i), we also have to show that $V$ is $(A+U)$-compact. To this end, suppose that 
$(x_n)_{n\in\N}$, $((A+U)x_n)_{n\in\N}$ are bounded. Because $U$ is $(A+U)$-compact and thus $(A+U)$-bounded, this implies that $(Ux_n)_{n\in\N}$ 
is bounded and hence so is $(Ax_n)_{n\in\N}$. Since $V$ is $A$-compact by assumption in (i), it follows that $(Vx_n)_{n\in\N}$ contains a convergent subsequence.
\end{proof}

In the following theorem, instead of requiring
the perturbation to decompose into real and imaginary parts,
we strengthen the relative compactness assumptions on it. % in Theorem \ref{thmABUV}.
If e.g.\ $A$ is uniformly positive, then instead of assuming that $S$ is $A$-compact, i.e.\ $SA^{-1}$ is compact,
we have to assume that  $A^{-1/2}SA^{-1/2}$ is compact.

That $S$ is $A$-compact implies that $S$ is $A$-form-compact, i.e.\ $|S|^{1/2}A^{-1/2}$ is compact, is not only true for symmetric $S$, 
see \cite[Theorem~VI.1.38]{Kato}, but even for densely defined closed $S$, see \cite[Thm.~3.5]{MR2501173}. Therein it was also shown that, 
if $|S|^{1/2}A^{-1/2}$ is compact and $\dom(A) \subset \dom(S) \cap \dom(S^*)$, then even $A^{-1/2}SA^{-1/2}$ is compact. 
Note that this need not be true for non-symmetric $S$ since $S^*$ may not be $A^{1/2}$-bounded; incidentally,
Example \ref{Ex2} below shows that the condition $\dom(A) \subset \dom(S) \cap \dom(S^*)$ is also necessary.

\begin{theorem}\label{thmSqrt}
Let $T=A+\I B$ with uniformly positive $A$ and symmetric $B$ and let $A^{-1/2}S$ be $A^{1/2}$-compact, i.e.\ $A^{-1/2}SA^{-1/2}$ is compact. Then $W_{\!e}(T)\!=\!W_{\!e}(T+S)$.
In particular, if $S$ is $A$-compact and $\dom(A)\!\subset\!\dom(S) \cap \dom(S^*)$, then $W_{\!e}(T)\!=\!W_{\!e}(T+S)$.
\end{theorem}

\begin{proof}
Since $\dom(T)\!=\!\dom(A)\cap\dom(B)\!\subset\!\dom(A^{1/2})\!\subset\!\dom(S)$, we have~$\dom(T\!+\!S)\!=\!\dom(T)$.

``$\supset$": Let $(x_n)_{n\in\N}\subset\dom(T)$ satisfy $\|x_n\|=1$, $x_n\stackrel{w}{\to}0$ and
\begin{equation}\label{eq.lm.in.W.TplusK}
\langle (T+S)x_n,x_n\rangle \longrightarrow\lambda\in W_e(T+S), \quad n\to\infty.
\end{equation}
First we show that $\big(\|A^{1/2}x_n\|\big)_{n\in\N}$ is a bounded sequence.
To this end, we estimate $|\langle Tx_n,x_n\rangle|\geq  \langle Ax_n,x_n \rangle= \|A^{1/2}x_n\|^2$ to obtain
\begin{equation}
\label{estTA}
\begin{aligned}
\big|\langle (T+S) x_n,x_n\rangle\big|
&\geq |\langle Tx_n,x_n\rangle|\left(1-\frac{|\langle A^{-1/2}Kx_n,A^{1/2}x_n\rangle|}{|\langle Tx_n,x_n\rangle|}\right)\\
&\geq \|A^{1/2}x_n\|^2\left(1- \frac{\|A^{-1/2}Sx_n\|}{\|A^{1/2}x_n\|}\right), \quad n\in\N.
\end{aligned}
\end{equation}
Since $A^{-1/2}S$ is $A^{1/2}$-compact by the assumptions, it is relatively bounded with $A^{1/2}$-bound $0$
and hence $\|A^{-1/2}Sx_n\| / \|A^{1/2}x_n\| \!\to\! 0$ if $\|A^{1/2}x_n\|\!\to\!\infty$ as $n\!\to\!\infty$.
Together with \eqref{estTA} we see that $\|A^{1/2}x_n\|\to\infty$ implies $|\langle (T+S)x_n,x_n\rangle|\to \infty$ as $n\to\infty$, 
a contradiction to~\eqref{eq.lm.in.W.TplusK}. Thus  $\big(\|A^{1/2}x_n\|\big)_{n\in\N}$ is bounded.

Since $x_n\stackrel{w}{\to}{0}$ and  $A^{-1/2}S$ is $A^{1/2}$-compact, it follows that $A^{-1/2}Sx_n\to 0$ as $n\to\infty$.
So we conclude $\langle S x_n,x_n\rangle=\langle A^{-1/2}Sx_n,A^{1/2}x_n\rangle \to 0$ and hence $\langle Tx_n,x_n\rangle\to \lm\in W_{\!e}(T)$  as $n\to\infty$.

``$\subset$":  Let $(x_n)_{n\in\N}$ satisfy $\|x_n\|=1$, $x_n\stackrel{w}{\to}0$ and
$$\langle Ax_n,x_n\rangle +\I \langle B x_n,x_n\rangle=\langle Tx_n,x_n\rangle \longrightarrow \lambda\in W_{\!e}(T).$$
The assumptions on $A$ and $B$ imply that $(\|A^{1/2}x_n\|\big)_{n\in\N} = (\langle Ax_n,x_n \rangle )_{n\in\N}$
is bound\-ed. As in the last step above, %by the boundedness of $\big(\|A^{1/2}x_n\|\big)_{n\in\N}$ and 
the compactness assumption on $S$ yields $\langle S x_n,x_n\rangle\!\to\! 0$ as $n\to\infty$. Therefore $\lm\in W_e(T+S)$.
\end{proof}

\begin{rem}\label{rempower}
All possible choices of $\alpha,\beta\geq 0$ for which $A^{-\alpha}SA^{-\beta}$ is compact implies $W_{\!e}(T)=W_e(T+S)$ are given by $\alpha\in [0,1/2]$, $\beta\in [0,1/2]$. In fact, 
Theorem~\ref{thmSqrt} proves the admissibility of $\alpha=\beta=1/2$, and thus also of all $\alpha\in [0,1/2]$, $\beta\in [0,1/2]$ which correspond to stronger conditions. 
Example~\ref{Ex2} below shows that $\beta>1/2$ is not eligible, and if we replace $S$ therein by its adjoint $S^*$, we see that $\alpha>1/2$ is not eligible either. 
%Similarly, if in  Example~\ref{Ex2} we perturb by the adjoint operator $S^*$ instead of $S$, we obtain a counterexample for .
\end{rem}

%In the following example we perturb a uniformly positive operator $T$ by another selfadjoint operator $K$ which is not $T$-compact but $T^{-1/2}S$ is %$T^{1/2}$-compact. Hence Theorem~\ref{thmSqrt} implies $W_e(T+S)=W_{\!e}(T)$.

The following example illustrates that $T^{-1/2}S T^{-1/2}$ may be compact, whereas $ST^{-1}$ need not be compact, 
even for a uniformly positive selfadjoint operator $T$ and a bounded selfadjoint operator $S$; in this case 
neither Proposition \ref{propsapert} nor Theorem~\ref{thmABUV} 
apply, but Theorem~\ref{thmSqrt} yields $W_e(T+S)=W_{\!e}(T)$.

\begin{example}\label{Ex1}
Let $\{e_k:\,k\in\N\}$ denote the standard orthonormal basis of  $l^2(\N)$.
In $l^2(\N) = \bigoplus_{n\in\N} M_n$ with $M_n\!:=\!{\rm span}\{e_{2n-1},e_{2n}\}$ 
we introduce two selfadjoint operators, identified with their block matrix representations %with respect to $\{e_k:\,k\in\N\}$, 
\be
\begin{aligned}
T&\!:=\!{\rm diag}\left(\!\begin{pmatrix} n^2 & 0 \\ 0 & 1\end{pmatrix}\!:\!n\in\N\!\right), \ \
\dom(T)\!:=\!\left\{(x_k)_{k\in\N}\in l^2(\N):\sum_{n=1}^{\infty}\!|n^2x_{2n-1}|^2\!<\!\infty\!\right\},\\
S&:={\rm diag}\left(\begin{pmatrix} 0 & 1 \\ 1 & 0 \end{pmatrix}\!:\!n\in\N\right), \ \ \dom(S):=l^2(\N).
\end{aligned}
\ee
The operator $T$ is uniformly positive and $S$ is bounded and selfadjoint. It is easy to check that $ST^{-1}$ is not compact whereas $T^{-1/2} S T^{-1/2}$ is compact.

Therefore, by Theorem~\ref{thmSqrt} and Theorem~\ref{thm.selfadjoint}, we have 
$$W_e(T+S)=W_{\!e}(T) ={\rm conv}(\widehat{\sigma}_e(T))\backslash\{\pm\infty\}=[1,\infty).$$
\end{example}

The following example shows that the essential numerical range need not be preserved if $S$ is $T$-compact, i.e.\ $ST^{-1}$ 
is compact rather than $T^{-1/2}ST^{-1/2}$ is compact. 
To obtain the latter we need that $\dom(T) \not\subset \dom(S) \cap \dom (S^*)$ by \cite[Thm.~3.5~(ii)]{MR2501173},
and to achieve the former we need that $S$ is $T$-compact, but does not satisfy the stronger
assumption $S=U+\I V$ with $T$-compact symmetric operators $U$, $V$ (compare Theorem~\ref{thmABUV}); note that the latter
necessitates $\dom(S)=\dom(U) \cap \dom(V)$.

\begin{example}\label{Ex2}
Let $T$  be the uniformly positive operator defined in Example~\ref{Ex1}, and define the operator 
$S$ in $l^2(\N)= \bigoplus_{n\in\N} M_n$ with $M_n\!:=\!{\rm span}\{e_{2n-1},e_{2n}\}$ by
\begin{align*}
S&\!:=\!{\rm diag}\left(\!\begin{pmatrix} 0 & 0 \\ n & 0\end{pmatrix}\!:\!n\in\N\!\right), \ \
\dom(S)\!:=\!\left\{(x_k)_{k\in\N}\in l^2(\N)\!:\!\sum_{n=1}^{\infty}\!|nx_{2n-1}|^2<\infty\right\}.
\end{align*}
It is easy to check that $ST^{-1}$ is compact, whereas $T^{-1/2} S T^{-1/2}$ is not;
further, $T^{-\alpha}ST^{-\beta}$ is compact if $\beta>1/2$ and $T^{-\alpha}S^*T^{-\beta}$ is compact if $\beta>1/2$, comp.\ Remark \ref{rempower}.
Note that here $\dom(T) \not\subset \dom (S^*)$ and hence the condition $\dom(T) \subset \dom(S) \cap \dom (S^*)$ is violated 
which would have implied the compactness of $T^{-1/2} S T^{-1/2}$ by \cite[Thm.~3.5~(ii)]{MR2501173}. In fact, $S$ is neither sectorial nor accretive and
$$
 \dom(S) \cap \dom(S^*) = \left\{(x_k)_{k\in\N}\in l^2(\N)\!:\!\sum_{k=1}^{\infty}\!|kx_{k}|^2<\infty\right\} \subsetneq \dom(S), 
$$
so there is no result to conclude that $S = U+\I V$ with symmetric $U$, $V$, let alone that $U$, $V$ are $T$-compact as required in 
Theorem~\ref{thmABUV}.

Indeed, here the essential numerical ranges of $T$ and $T+S$ do not coincide since
$W_{\!e}(T)\!=\!{\rm conv}(\widehat{\sigma}_e(T))\backslash\{\pm\infty\}\!=\![1,\infty)$ by Theorem~\ref{thm.selfadjoint},  see also Example~\ref{Ex1}, 
and we will show that
\begin{align}
\label{We2parabola}
W_e(T+S)%=\{ |\gamma|^2+1+\gamma:\,\gamma\in\C\}
=\left\{ \lm\in\C: \re\,\lm \geq \frac 34, |\im\,\lm| \leq \sqrt{\re\,\lm - \frac 34}\right\}
=:E,
\end{align}
so that $W_{\!e}(T)\subsetneq W_e(T+S)$. 

To prove \eqref{We2parabola}, we first note that 
the numerical ranges of the $2\times 2$-matrices $P_{M_n}(T\!+\!S)|_{M_n}$, $n\!\in\! \N$, are ellipses with foci $1$, $n^2$ and minor semi-axis~$n/2$,
\begin{align*}
  E_n\!:= \!W(P_{M_n}\!(T\!+\!S)|_{M_n}) \!=\! W \!\left( \!\begin{pmatrix} 1\!&\!\!\!0\\[-0.5mm]n\!&\!\!n^2 \end{pmatrix} \!\right) \!=\! 
	\left\{x\!+\!\I y \!\in\! \C \!:\! \frac{( x \!-\! \frac{n^2\!+1}2)^2}{\!(\frac n2)^2 \!+\! (\frac{n^2\!-1}2)^2} \!+\! \frac{y^2}{(\frac n2)^2\!} \!\leq\! 1 \!\right\}\!.
\end{align*}%
Solving for $y$ and letting $n\to \infty$, it is not difficult to check that the non-nested sequence of ellipses $E_n$, n$\in\N$, has the following convergence properties with respect 
to Kuratowski distance of closed (unbounded) subsets of $\C$, see e.g.\ \cite[Chapt.\ 4]{MR1491362},
\begin{align*}
\label{ku-conv} 
E_n \to E, \ \  n\to \infty, \quad U_m:=\bigcup_{n\in\N, n\ge m} E_n \to E, \ \ m\to \infty.
\vspace{-1mm}
\end{align*}
This means e.g.\ that the set of limits points of sequences $(\lambda_n)_{n\in\N}$ with $\lambda_n\in E_n$, $n\in\N$, and the set of its accumulation points coincide and are equal to $E$. 
Since every sequence $(f_n)_{n\in\N}$ with $f_n\in M_n$ satisfies $f_n\stackrel{w}{\to}0$, this proves, in particular, 
the inclusion ``$\supset$'' in \eqref{We2parabola}.

For the converse inclusion ``$\subset$'' in \eqref{We2parabola} 
we use that, by Theorem~\ref{thmequivdefforWe},  
\begin{align*}
W_e(T+S) = W_{e2}(T+S) &= W_{e2}({\rm diag}(P_{M_n}(T+S)|_{M_n}:\, n\geq m))\\
&\subset \ \ \overline{W({\rm diag}(P_{M_n}(T+S)|_{M_n}:\, n\geq m))} 
\vspace{-1mm}
\end{align*}
for arbitrary $m\in\N$ and hence
\[
W_e(T+S) \subset \bigcap_{m\in\N} \overline{W({\rm diag}(P_{M_n}(T+S)|_{M_n}:\, n\geq m))} = \bigcap_{m\in\N} 
{\rm conv}\, U_m. \vspace{-1mm}
\]
Since $U_m\to E$ implies ${\rm conv}\, U_m \to \overline{{\rm conv}\, E}=E$ as $m\to\infty$, see \cite[Prop.\ 4.30 (c)]{MR1491362}, and $U_m$, $m\in \N$, is a nested decreasing sequence, $U_m \!\supset\! U_{m'}$, $m \!\le\! m'$, it follows that $\bigcap_{m\in\N} {\rm conv}\, U_m = \lim_{m\to\infty}
{\rm conv}\, U_m =E$.
\end{example}

\begin{rem}
\label{neutral-delay-diff}
Note that the operators $T$ and $S$ in Example \ref{Ex2} are related to %have the same matrix representations as 
certain \emph{neutral delay differential expressions}, see e.g.~\cite[Chapt.~3]{delay}. 
More precisely, if we define 
$$(\tau_{T} f)(x):=-\frac{f''(x)+f''(-x)}{2}+\frac{f(x)-f(-x)}{2}, \quad (\tau_{S} f)(x):=-\frac{f'(x)-f'(-x)}{2}$$
and consider 
the realizations of $\tau_T$ and $\tau_S$ in $L^2(-\pi,\pi)$ with periodic boundary conditions and domain orthogonal to the constant functions, it is easy to check that the corresponding 
matrix representations in $l^2(\N)$ with respect to $\{\cos(k\cdot),\sin(k\cdot)\!:\!k\in\N\}$ are given by 
the infinite matrices $T$ and $S$, respectively, 
studied in Example~\ref{Ex2}.
\end{rem}

\section{The limiting essential numerical range of operator approximations}
\label{sec:limessnr}

In this section we introduce the notion of limiting essential numerical range $W_e\left((T_n)_{n\in\N}\right)$ of a sequence of operators $(T_n)_{n\in\N}$, a concept that is also new in the bounded case. We will establish conditions 
under which the limiting essential numerical range coincides with the essential numerical range $W_{\!e}(T)$ of the limit operator $T$ in generalized strong resolvent sense. 

Our main result is that $W_{\!e}(T)$ contains all pathologies that might occur for any such operator approximation; first spectral pollution where a sequence of eigenvalues of the approximating operator converges to a point 
$\lambda\notin \sigma(T)$,  and secondly failure of spectral inclusion where a true spectral point $\lambda\in \sigma(T)$ is not approximated.

To this end, we first provide some abstract notions for operator sequences, their spectra and convergence behaviour.  
Let $H_n\subset H$, $n\in\N,$ be closed subspaces and denote by 
$P_n=P_{H_n}:H\to H_n$, $n\in\N$, the orthogonal projections in $H$ onto them. 
Let $T: H \supset \dom(T) \to H$ and $T_n : H_n \supset \dom(T_n) \to H_n$, $n\in\N$, be linear operators in $H$ and $H_n$, $n\in\N$, respectively.

The following local notions of spectral inclusion and spectral exactness have their origin in~\cite{Bailey-1993} by Bailey et al.\ for selfadjoint operators, the notion of spectral pollution may be traced back to 
\cite{MR642457} by Rappaz for bounded non-compact operators (comp.\ \cite[Def.\ 2.2]{boegli-chap1}).

\begin{definition}\label{defspectralexactness}
\begin{enumerate}[label=\rm{\roman{*})}]
\item The \emph{limiting spectrum} of $(T_n)_{n\in\N}$ is defined as
\[
 \sigma((T_n)_{n\in\N})\!:=\! \{\lm\!\in\! \C\!:\! \exists\,I\!\subset\!\N\,\text{infinite},\,\exists\,\lm_n\!\in\!\sigma(T_n), \,n\!\in\! I, \lm_n\!\to\!\lm \}; \hspace{-14mm}
\]
a point $\lm \in \sigma(T)$ is called \emph{approximated} by $(T_n)_{n\in\N}$ if $\lm \in \sigma((T_n)_{n\in\N})$.
\item the \emph{set of spectral pollution} of $(T_n)_{n\in\N}$ is defined~as
\[
 \sigma_{poll}((T_n)_{n\in\N})\!:=\! \{\lm\!\notin\! \sigma(T)\!:\! \exists\,I\!\subset\!\N\,\text{infinite},\,\exists\,\lm_n\!\in\!\sigma(T_n), \,n\!\in\! I, \lm_n\!\to\!\lm \}; \hspace{-14mm}
\]
a point $\lm \in \sigma_{poll}((T_n)_{n\in\N})$ is called  \emph{spurious eigenvalue} of \vspace{0.7mm} $(T_n)_{n\in\N}$.
\item $(T_n)_{n\in\N}$ is called \emph{spectrally inclusive for $T$ in $\Lambda\subset\C$} if 
$$
  \sigma(T) \cap \Lambda \subset  \sigma((T_n)_{n\in\N}).
$$
\item  $(T_n)_{n\in\N}$ is called \emph{spectrally exact for $T$ in $\Lambda\subset\C$} if it is spectrally inclusive for $T$ in $\Lambda$ and no spectral pollution occurs in $\Lambda$, i.e.\
$$
 \sigma(T) \cap \Lambda \subset  \sigma((T_n)_{n\in\N}) \quad \mbox{and} \quad  \sigma_{poll}((T_n)_{n\in\N}) \cap \Lambda = \emptyset.
$$
\end{enumerate}
%\marginpar{\sab{i) should be ii)?}}
\noindent
If iii) and iv), respectively, hold for $\Lambda =\C$, then $(T_n)_{n\in\N}$ is called \emph{spectrally inclusive} or \emph{spectrally exact}, respectively.
\end{definition}

The following definition of generalized strong resolvent convergence is due to Weidmann~\cite[Section~9.3]{weid1} in the selfadjoint case.
%\marginpar{\ct{\footnotesize correct?}}
The limiting essential spectrum was introduced in~\cite{boegli-limitingess} and generalizes a notion from~\cite{boulton} for the Galerkin method of selfadjoint operators.

\begin{definition}\label{defresconvandsigmae}
%Let $T\in C(H)$ and $T_n\in C(H_n)$, $n\in\N$. 
\begin{itemize}
\item[\rm i)]
The operator sequence $(T_n)_{n\in\N}$ is said to \emph{converge in generalized strong resolvent sense} to $T$, 
$T_n\gsr T$,  if there exists $n_0\in\N$~with
$$\exists\ \lm_0\in\!\!\underset{n\geq n_0}{\bigcap}\!\rho(T_n)\cap\rho(T):\quad (T_n-\lm_0)^{-1}P_n\slong (T-\lm_0)^{-1}, \quad n\to\infty. 
\vspace{-2mm}
\hspace{-10mm} $$
% \marginpar{\ct{\footnotesize is this notion useful? appears in Boulton \& Hansen}}
% \item[\rm ii)]
% \ct{The \emph{limiting spectrum} of $(T_n)_{n\in\N}$ is defined as  
% $$
% \sigma\left((T_n)_{n\in\N}\right):=
% \left\{
% \lm\in\C: \exists\,\lm_n\!\in\!\sigma(T_n), n\!\in\!\N, \ \text{with} \ \lm_n \to \lm 
% \right\}.\vspace{-2mm}
% $$
% }
\item[\rm ii)]
The \emph{limiting essential spectrum} of $(T_n)_{n\in\N}$ is defined as  
$$\!\!\sigma_e\!\left((T_n)_{n\in\N}\right)\!:=\!\left\{\lm\!\in\!\C:\!\!\!
\begin{array}{c}
\exists\,I\!\subset\!\N\,\text{infinite},\,\exists\, (x_n)_{n\in I}\!\subset\! H, x_n\!\in\dom(T_n),\! 
\\[1mm]
\text{with} \ \|x_n\|\!=\!1,\ x_n\stackrel{w}{\to}0,\
\|(T_n\!-\!\lm)x_n\|\to 0 
\end{array} \!\!
\right\}\!.\hspace{-12mm}
$$
\end{itemize}
\end{definition}

\begin{rem}
\label{sigma_e_lim}
The limiting essential spectrum $\sigma_e\!\left((T_n)_{n\in\N}\right)$ is closed and, if 
$P_n\s I$,
\[
   T_n\gsr T  \quad \implies \quad
	\sigma_e(T) \subset \sigma_e\left((T_n)_{n\in\N}\right), 
\]	
while equality requires generalized norm resolvent convergence, see  \cite[Prop.~2.4,~2.7]{boegli-limitingess}.
\end{rem}

For a particular operator approximation $(T_n)_{n\in\N}$, the following local spectral exactness result from~\cite{boegli-limitingess} identifies sets to which spectral  
pollution is confined and outside of which isolated spectral points are spectrally included. 

\begin{theorem}[{\cite[Thm.~2.3]{boegli-limitingess}}]
\label{mainthmspectralexactness}
Let $\,\overline{\dom(T)}=H$, $\overline{\dom(T_n)}=H_n$, $n\in\N$, and $P_n\s I$.
\begin{enumerate}[label=\rm{\roman{*})}] 
\item 
If \,$T_n\gsr T$ and $\,T_n^*\gsr T^*$, 
then 
\beq 
\sigma_{poll}((T_n)_{n\in\N}) \subset\,
\sigma_e\!\left((T_n)_{n\in\N}\right)\cup\sigma_e\!\left((T_n^*)_{n\in\N}\right)^*,\label{eqbadset}
\eeq
and  every isolated $\lm\!\in\!\sigma(T)$ outside $\sigma_e\!\left((T_n)_{n\in\N}\right)\!\cup\!\sigma_e\!\left((T_n^*)_{n\in\N}\right)^*$ 
is approximated by $(T_n)_{n\in\N}$.\vspace{0.7mm}
\item 
If \,$T_n\gsr T$ and all $\,T_n$, $n\in\N$, have compact resolvents,  
then claim{\rm~i)} holds with $\sigma_e\!\left((T_n)_{n\in\N}\right)\cup\sigma_e\!\left((T_n^*)_{n\in\N}\right)^*$ replaced by $\sigma_e\!\left((T_n^*)_{n\in\N}\right)^*$.
\end{enumerate}
\end{theorem}

The following concept of limiting essential numerical range for operator sequences is new even in the case of bounded operators.

\begin{definition}\label{defWeTn}
We define  the \emph{limiting essential numerical range} of $(T_n)_{n\in\N}$ by 
$$W_e\left((T_n)_{n\in\N}\right):=\left\{\lm\in\C:\!\!\!
\begin{array}{c}\exists\,I\subset\N\,\text{infinite},\,\exists\,(x_n)_{n\in I}\!\subset\! H, x_n\!\in \dom(T_n), \\[1mm]  
\text{with} \ \|x_n\|=1,\ x_n\stackrel{w}{\to}0,\ \langle T_nx_n,x_n\rangle\to \lm\end{array}\!\!\right\}.
$$
\end{definition}

\vspace{1mm}

Clearly, $W_e\!\left((zT_n)_{n\in\N}\right)\!=\!zW_e\!\left((T_n)_{n\in\N}\right)$ and 
$W_e\!\left((T_n\!+\!z)_{n\in\N}\right)\!=\!W_e\!\left((T_n)_{n\in\N}\right)\!+\!z$ for $z\in\C$. 

The next results relating limiting essential spectrum, limiting essential numerical range and essential numerical range will be used 
in later sections.

\begin{prop}\label{propWeTnprop}
\begin{enumerate}
\item[\rm i)] 
The limiting essential numerical range $W_e\left((T_n)_{n\in\N}\right)$ is closed and convex with
${\rm conv}\,\sigma_e\!\left((T_n)_{n\in\N}\right)\subset W_e\left((T_n)_{n\in\N}\right)$, and, if $P_n\s I$,
\begin{equation}
\label{gsr-Welim}
 T_n\gsr T  \quad \implies \quad W_{\!e}(T)\subset W_e\left((T_n)_{n\in\N}\right).
\end{equation}
\item[\rm ii] If, for every $n\in\N$, $\dom(T_n)\cap \dom(T_n^*)$ is a core of \,$T_n^*$, then 
$${\rm conv}\big(\sigma_e\left((T_n)_{n\in\N}\right)\cup \sigma_e\left((T_n^*)_{n\in\N}\right)^*\big)\subset W_e\!\left((T_n)_{n\in\N}\right).$$
\end{enumerate}
\end{prop}

\begin{proof}
The first three claims in i) are proved in the same way as Proposition~\ref{propWeclosed-convex}; claim ii) is shown in an analogous way as Remark~\ref{remrelcompchar}~i).

In order to prove \eqref{gsr-Welim},   
let $\lm\in W_{\!e}(T)$. Then there exist $x_k\in\dom(T)$, $k\in\N,$ with $\|x_k\|=1$, $x_k\stackrel{w}{\to}0$ and $\langle (T-\lm)x_k,x_k\rangle\to 0$ as $k\to\infty$.
Using $P_n\s I$, $T_n\gsr T$ and choosing $\lm_0$ as in Def.\ \ref{defresconvandsigmae} i), we let $x_{k;n}:= (T_n-\lm_0)^{-1} P_n (T-\lm_0)x_k \in \dom(T_n)$, $k,n\in\N$. Then,
for every $k\in\N$, we have %we find $x_{k;n}\in\dom(T_n)$, $n\in\N$, so that 
$\|x_{k;n}-x_k\|\to 0$ and $\|T_nx_{k;n}-Tx_{k}\|\to 0$ as $n\to\infty$;
in particular, $\|x_{k;n} \| \to 1$ as $n\to\infty$. 
Hence  we can find a strictly increasing sequence $(n_k)_{k\in\N}\subset\N$ such that, for every $k\in\N$, the element $y_k:=x_{k;n_k}\in\dom(T_{n_k})$ satisfies 
$$\|y_k-x_k\|<\frac{1}{k\|Tx_k\|},\quad \|T_{n_k}y_k-Tx_k\|<\frac{1}{k}.$$
Then the sequence $(y_k)_{k\in\N}$ is bounded and bounded away from $0$ with
$$\left|\langle T_{n_k}y_k,y_k\rangle-\lm\right|
\leq \left|\langle T x_k,x_k\rangle-\lm\right| + \|T x_k\|\|y_k-x_k\|+\|T_{n_k}y_k-Tx_k\| \|y_k\|\tolong 0$$
as $k\to\infty$. Hence $\widetilde x_{n_k}:=y_k/\|y_k\|\in\dom(T_{n_k})$, $k\in\N$, satisfy $\|\widetilde x_{n_k}\|=1$, $\widetilde x_{n_k}\stackrel{w}{\to}0$~and 
$\langle T_{n_k}\widetilde x_{n_k},\widetilde x_{n_k}\rangle \to \lm$  as $k\to\infty$, which proves that $\lm\in W_e\left((T_n)_{n\in\N}\right)$.
\end{proof}

\begin{prop}\label{prop:limitingWe}
If $\,t$ is a sesquilinear form with domain $\dom(t)\subset H$ such that 
$$
\dom(T_n)\subset\dom(t), \quad 
\langle T_nx_n,x_n\rangle= t[x_n], \quad x_n\in\dom(T_n), \ n\in \N,
$$
then $W_e\!\left((T_n)_{n\in\N}\right)\!\subset\! W_{\!e}(t)$ and, if $\,\dom(T)$ is a core of $\,t$, 
then $W_e\!\left((T_n)_{n\in\N}\right)\!\subset\! W_{\!e}(T)$.
\end{prop}

\begin{proof}
The claims follow from Definition \ref{defWesesq} and the remarks thereafter.
\end{proof}

The following example shows that the sets $W_e\big((T_n)_{n\in\N}\big)$ and ${\rm conv}\,\sigma_e\big((T_n)_{n\in\N}\big)$
may be larger than $W_{\!e}(T)$, even if all operators are bounded with $T_n\!\s\!T$. 
It also shows that it is important \emph{not} to choose the subspaces $H_n$ unnecessarily large since this may artificially blow up the limiting sets.

\begin{example}
In $H\!=\!H_n\!=\!l^2(\N)$ with standard orthonormal basis $\left\{e_k:\,k\in\N\right\}$ 
con\-sider the operators $T:=I: l^2(\N) \!\to\! l^2(\N)$ and $T_n:l^2(\N) \!\to\! l^2(\N)$,  $n\!\in\!\N$, given~by
$$T_nx:=\sum_{k=1}^n \langle x,e_k\rangle e_k,\quad x\in l^2(\N).$$
Clearly, $T$ and $T_n$, $n\in\N$, are selfadjoint and bounded in $l^2(\N)$ with $T_n\s T$,
but we have the strict inclusions
\begin{equation}
\label{strict}
\sigma_e(T)\!=\!W_{\!e}(T)\!=\!\{1\}
\ \ \subsetneq\ \
\sigma_e\big((T_n)_{n\in\N}\big)\!=\!\{0,1\}
\ \ \subsetneq \ \
W_e\big((T_n)_{n\in\N}\big)\!=\![0,1].
\end{equation}
Here the equalities on the left are obvious. For the middle equality in \eqref{strict}, we note that
$\|e_{n}\|=1$, $e_{n}\stackrel{w}{\to}0$ and $\|(T_n-1)e_n\|=0$, $\|T_n e_{n+1}\|=0$ imply that $1,0 \in \sigma_e\left((T_n)_{n\in\N}\right)$;
vice versa,  $\sigma(T_n)=\{0,1\}$, $n\in \N$, implies that $\sigma_e\big((T_n)_{n\in\N}\big)\subset\{0,1\}$.
The last equality in \eqref{strict} follows from $W_e\left((T_n)_{n\in\N}\right)\subset \overline{W(T_n)} = W(T_n) = [0,1]$, $n\in\N$, 
the middle equality and Proposition \ref{propWeTnprop} i).

Note that if we consider the operators $T_n$ in $H_n:={\rm span}\{e_k:\,k=1,\dots,n\}$, we obtain 
$\sigma_e(T)=W_{\!e}(T)=\sigma_e\big((T_n)_{n\in\N}\big)=W_e\big((T_n)_{n\in\N}\big)=\{1\}$.
\end{example}

\section{Application I: Projection method}\label{sectiongalerkin}

In this section we focus on projection methods. We prove that, for any projection method, the essential numerical range $W_{\!e}(T)$ contains all possible spectral pollution
and that $W_{\!e}(T)$ is the smallest set with this property because arbitrary points in $W_{\!e}(T)$ can be arranged to be spurious eigenvalues. 

As in the previous section, for a closed subspace $V\subset H$ we denote by $P_V\!:H\to V$ the orthogonal pro\-jection in $H$ onto~$V$. 
If $V\subset\dom(T)$, then $T_V:=P_VT|_{V}$ denotes the compression of $T$ to $V$.

\begin{theorem}\label{thmgalerkin}
Assume $\overline{\dom(T)}\!=\!H$. 
Let $P_{H_n}:H\!\to\! H_n$, $n\!\in\!\N$, be orthogonal pro\-jections onto finite{-dimensional} subspaces $H_n\!\subset\! \dom(T)$ with $P_{H_n}\!\s\! I$.
If $\,T_{H_n}\!\gsr\! T$, then
\begin{enumerate}
\item[\rm i)] $W_e\left((T_{H_n})_{n\in\N}\right)=W_{\!e}(T)$,
\item[\rm ii)] spectral pollution is confined to $W_{\!e}(T)$, 
\item[\rm iii)] every isolated $\lm\in\sigma(T)$ outside $W_{\!e}(T)$ is approximated. % , there exists a sequence of $\lm_n\in\sigma(T_{H_n}), \,n\in\N,$ with $\lm_n\to\lm, \,n\to\infty$.
\end{enumerate}
\end{theorem}

\begin{proof}
i)
The equality follows from \eqref{gsr-Welim} in Proposition \ref{propWeTnprop} i) and from Proposition~\ref{prop:limitingWe} applied with the form $t$ associated with $T$, noting  
that, for every $n\in\N$,  $\dom(T_{H_n})= H_n\subset\dom(T)$ and $\langle T_{H_n}x_n,x_n\rangle=\langle Tx_n,x_n\rangle$ for $x_n\in H_n$.

% First, let $\lm\in W_e\left((T_{H_n})_{n\in\N}\right)$. By definition, there exist an infinite subset $I\subset\N$
%and $x_n\in H_n$, $n\in I,$ with $\|x_n\|=1$, $x_n\stackrel{w}{\to}0$  and $\langle T_{H_n}x_n,x_n\rangle\to \lm$.
%Since $P_{H_n} x_n=x_n\in H_n\subset\dom(T)$, $n\in\N$, we obtain 
%\begin{align*}
%\langle Tx_n,x_n\rangle&=\langle P_{H_n}Tx_n, x_n\rangle=\langle T_{H_n}x_n,x_n\rangle\tolong 0, \quad n\to\infty;
% \end{align*}
%hence $\lm\in W_{\!e}(T)$.
%
%Now let $\lm\in W_{\!e}(T)$. Then there exist $x_k\in\dom(T)$, $k\in \N,$ with $\|x_k\|=1$, $x_k\stackrel{w}{\to}0$  and $\langle Tx_k,x_k\rangle\to \lm$.
%Let $k\in\N$ be fixed. 
%Since $T_{H_n}\gsr T$, there exist $n_0\in\N$ and $\lm_0\in\underset{n\geq n_0}{\bigcap}\rho(T_{H_n})\cap\rho(T)$ such that $(T_{H_n}-\lm_0)^{-1}P_{H_n}\s(T-\lm_0)^{-1}$.
%It is easy to check that the elements $$x_{k:n}:=(T_{H_n}-\lm_0)^{-1}P_{H_n}(T-\lm_0)x_k\in H_n, \quad n\geq n_0,$$
%satisfy $x_{k;n}\to x_k$ and $T_{H_n}x_{k;n}\to Tx_k$ as $n\to\infty$.
%So we can construct a strictly increasing sequence $(n_k)_{k\in\N}\subset\N$ such that 
%$$y_k:=\frac{x_{k;n_k}}{\|x_{k;n_k}\|}\in H_{n_k}, \quad k\in\N,$$
%satisfy $\|y_k\|=1$,  $y_k\stackrel{w}{\to} 0$ and $\langle T_{H_{n_k}}y_k,y_k\rangle\to \lm$ as $k\to\infty$.
%Hence $\lm\in W_e\big((T_{H_n})_{n\in\N}\big)$.

ii), iii)  
By claim i) and Proposition~\ref{propWeTnprop}~i), we know that
$$W_{\!e}(T)=W_e\big((T_{H_n})_{n\in\N}\big)\supset \sigma_e\left((T_{H_n})_{n\in\N}\right)\cup\sigma_e\big((T_{H_n}^*)_{n\in\N}\big)^*.$$
Now the two assertions follow from Theorem~\ref{mainthmspectralexactness}~ii).
\end{proof}

\vspace{-3mm}

\begin{rem}
\label{proj-sigma_e_lim}
If $\,T_{H_n}\!\gsr\! T$ and the subspaces $H_n=\ran (P_n) \subset \dom (T)$, $n\in\N$, are invariant for $T$, then 
\[
\sigma_e\left((T_{H_n})_{n\in\N}\right) = \sigma_e(T).
\]
Here the inclusion `$\supset$' follows from the first assumption, see Remark \ref{sigma_e_lim}, while the inclusion `$\subset$' follows from Definition \ref{defresconvandsigmae}~iii) since in this case $\dom(T_n) = H_n \subset \dom(T)$ and $T_n = P_{H_n} T|_{H_n} = T|_{H_n}$, $n\in\N$.
\end{rem}

The following result, together with its more detailed versions Theorem~\ref{thmdescloux} and Theorem~\ref{thmdescloux2}, constitutes one of the key advances of this paper.
It shows that $W_{\!e}(T)$ is the smallest possible set that captures spectral pollution for projection methods.

The proof is split in two steps and shows even more. Given an arbitrary $\lm\!\in\! W_{\!e}(T)$ and 
finite-dimensional subspaces $V_n$, we construct subspaces $H_n \!=\! \widetilde V_n \!\oplus\! {\rm span}\, \{e_n\}$ with 
$\widetilde V_n$ close to $V_n$ so that $\lm$ is a spurious eigenvalue for the pro\-jection method onto $H_n$;
if $W(T)\!\neq\!\C$ or $\overline{\dom(T)\cap\dom(T^*)}\!=\!H$, we can even choose~$\widetilde V_n \!=\! V_n$.

\begin{theorem}\label{thmsummary}
Assume that $\overline{\dom(T)}=H$. 
Then, for any $\lm\in W_{\!e}(T)$ there exists a sequence of finite-dimensional subspaces $H_n\subset\dom(T)$, $n\in\N$, such that 
$$P_{H_n}\slong I,\quad {\rm dist}(\lm,\sigma(T_{H_n}))\tolong 0,\quad n\to\infty,$$
and hence, for this projection method, $\lm\in W_{\!e}(T)\backslash\sigma(T)$ is a spurious eigenvalue. 
\end{theorem}

In the first step of the proof of Theorem \ref{thmsummary} we show 
%The following theorem shows 
that arbitrary compact subsets of $W_{e1}(T)$ can be filled with spurious eigenvalues.
Since $W_{e1}(T)=W_{\!e}(T)$ if $W(T)\neq\C$ or $\overline{\dom(T)\cap\dom(T^*)}=H$, see Theorem~\ref{thmequivdefforWe}, 
this completes the proof of Theorem \ref{thmsummary} in this case.

\begin{theorem}\label{thmdescloux}
Assume that $\overline{\dom(T)}=H$. Let $V_n\subset \dom(T)$, $n\in\N,$ be  finite-dimensional subspaces such that $P_{V_n}\s I$.
Then, for any compact $\Omega\subset W_{e1}(T)$, there exist finite-dimensional subspaces $H_n\subset \dom(T)$, $n\in\N$,   
with $V_n\subset H_n$ and with the following properties:
$$P_{H_n}\s I, \quad \sup_{\lm\in\Omega}\,{\rm dist}(\lm,\sigma(T_{H_n}))\tolong 0, \quad n\to\infty,$$
and if $\Omega\subset {\rm int}\, W_{e1}(T)$ is a finite set, then $\sigma(T_{H_n})=\sigma(T_{V_n})\cup\Omega$.
If, in addition,
\begin{enumerate}[label=\rm{(\alph{*})}] 
\item $W(T)\neq\C$ or  $\dom(T)\subset \dom(T^*)$;
\item $\,T_{V_n}\!\!\gsr\! T$ with corresponding 
$\lambda_0\!\notin\! \overline{W(T)}$ if $\,W(T)\!\neq\!\C$ and $\lambda_0\!\notin\!\Omega$ otherwise,
\end{enumerate}
then the subspaces $H_n\subset \dom(T),\,n\in\N,$ can be constructed so that
$\,T_{H_n}\!\gsr\! T$. 
\end{theorem}

\begin{rem}\label{remdescloux}
Theorem~\ref{thmdescloux} contains various earlier results as special cases:
\begin{enumerate}
\item[\rm i)]
For bounded operators, assumption (a) holds automatically and assumption (b) is satisfied for any sequence of subspaces $(V_n)_{n\in\N}$ with 
$P_{V_n}\!\!\s\! I$, and hence~\cite[Theorem~3]{descloux}, \cite{pokrzywa} are contained in Theorem \ref{thmdescloux}.

\item[\rm ii)] For selfadjoint operators, conv $\sigma_e(T) = W_{\!e}(T)=W_{e1}(T)$ by Theorem~\ref{thm.selfadjoint}, 
and thus~\cite[Theorem~2.1]{levitin}, \cite[Theorem~1.1]{lewin-sere} are contained in Theorem~\ref{thmdescloux}.
\end{enumerate}
\end{rem}

For the proof of Theorem~\ref{thmdescloux} we need the following lemma.

\begin{lemma}\label{lemmadescloux}
Let $V\subset \dom(T)$ be a finite-dimensional subspace.
Then, for given $\lm\!\in\!W_{e1}(T)$ and $\eps>0$, there exist $x\in V^{\perp}\!\cap\dom(T)$, $\|x\|\!=\!1$, and $\mu\!\in\! B_{\eps}(\lm)$ such~that 
\begin{equation}\label{eq:lemmanew}
T_{V_x}=\begin{pmatrix}T_V & A\\ 0 & \mu I\end{pmatrix} \quad\text{in}\quad V_x:=V\oplus{\rm span}\{x\}
\end{equation}
with a linear operator $A:{\rm span}\{x\}\to V$ and therefore $\sigma(T_{V_x})=\sigma(T_V)\cup\{\mu\}$.
Moreover, we can choose $A=0$ if $\,\overline{\dom(T)}=H$ and $\dom(T)\subset \dom(T^*)$,
and we can choose $\mu=\lm$ if $\lm\in {\rm int}\,W_{e1}(T)$.
\end{lemma}

\begin{proof}
Let $\lm\in W_{e1}(T)$ and $\eps>0$ be fixed. 
By definition, see Theorem \ref{thmequivdefforWe}, $W_{e1}(T)$ is the intersection of $\overline{W(T|_{U^{\perp}\cap\dom(T)})}$ over all finite-dimensional subspaces $U\subset H$.
Hence if we choose
$$U:={\rm span}\big(V\cup \ran(T|_V)\big),$$
there exists $\mu\in B_{\eps}(\lm)$ such that 
\beq\label{eq.xin.Uperp}
\exists\,x\in U^{\perp}\cap\dom(T)\subset V^{\perp}\cap\dom(T), \, \|x\|=1: \quad \langle Tx,x\rangle=\mu.\eeq
Since $x\in U^{\perp}$, we have
$$\langle Tv,x\rangle=0 \quad v\in V,$$
which implies the representation~\eqref{eq:lemmanew}. %with $A=0$.
Clearly, the matrix representation of $T_{V_x}$ yields that $\sigma(T_{V_x})=\sigma(T_V)\cup\{\mu\}$.

If $\,\overline{\dom(T)}=H$ and $\dom(T)\subset \dom(T^*)$, we can even choose $U:={\rm span}\big(V\cup \ran(T|_V)\cup \ran(T^*|_V)\big)$. Then \eqref{eq.xin.Uperp} yields that 
$$\langle Tv,x\rangle=0, \quad \langle Tx,v\rangle=\langle x,T^*v\rangle=0, \quad v\in V,$$
which implies the representation~\eqref{eq:lemmanew} with $A=0$.

If $\lm\!\in\! {\rm int}\,W_{e1}(T)$, there exists $\delta\!>\!0$ with $B_{\delta}(\lm)\!\subset\!W_{e1}(T)$ and hence
$B_{\delta}(\lm)\subset\overline{W(T|_{U^{\perp}\cap\dom(T)})}$. This
implies $\lm\in W(T|_{U^{\perp}\cap\dom(T)})$ and so~\eqref{eq.xin.Uperp} holds for~$\mu=\lm$.
\end{proof}

\begin{proof}[Proof of Theorem~{\rm\ref{thmdescloux}}]
Let $n\in\N$. There exists a finite open covering  
$\{D_{k;n}:k=1,\dots,N_n\}$ of $\Omega$ by open disks of radius $1/n$.
By applying Lemma~\ref{lemmadescloux} inductively $N_n$ times with $\eps:=1/n$, we construct orthonormal elements
$x_{1;n},\dots,x_{N_n;n}\in V_n^{\bot}\cap\dom(T)$ 
and points $\mu_{k;n}\in D_{1;n},\dots,\mu_{N_n;n}\in D_{N_n;n}$ such that 
\beq \label{eq:Hn}
H_n:=V_n\oplus {\rm span}\{x_{1;n}\}\oplus\dots\oplus {\rm span}\{x_{N_n;n}\}
\eeq 
satisfies $P_{H_n}\s I$ and
$$\sigma(T_{H_n})=\sigma(T_{V_n})\cup\{\mu_{1;n},\dots,\mu_{N_n;n}\}.$$
%Note that $P_{V_n}\s I$ implies $P_{H_n}\s I$.

If $\Omega\subset {\rm int}\, W_{e1}(T)$ is a finite set, 
then we choose the covering so that the centre of each $D_{k;n}$ is a point in $\Omega$ and so $\sigma(T_{H_n})=\sigma(T_{V_n})\cup\Omega$.
For a general compact subset $\Omega\subset W_{e1}(T)$, by construction of the disks $D_{k;n},\,k=1,\dots,N_n$, we have 
\beq \label{distOmega}
\sup_{\lm\in\Omega}\,\dist(\lm,\sigma(T_{H_n}))\leq \frac{2}{n}\tolong 0, \quad n\to\infty.
\eeq

Now we show  $T_{H_n}\!\gsr\!T$  if assumptions (a) and (b) hold.
First we consider the case $W(T)\!\neq\!\C$ in (a) where $\lambda_0 \!\in\!\bigcap_{n\in\N}\rho(T_{V_n})\!\cap\!\rho(T)$ in (b) satisfies $\lambda_0\!\notin\!\overline{W(T)}$. 
Then, since $\sigma(T_{H_n})\subset W(T_{H_n})\subset W(T)$, 
we have $\lambda_0 \in \rho(T_{H_n})$ and
\beq \label{eq:urb}
\|(T_{H_n}-\lm_0)^{-1}\|\leq \frac{1}{{\rm dist}(\lm_0,W(T_{H_n}))}\leq  \frac{1}{{\rm dist}(\lm_0,W(T))},\quad n\in\N.
\eeq
Lemma~\ref{lemmadescloux} yields that the matrix representation of $T_{H_n}$ in $H_n$
given by \eqref{eq:Hn}
is upper triangular. Now assumption (b) implies that 
\begin{equation}\label{eq:gsr1}
(T_{H_n}-\lm_0)^{-1}P_{V_n}=(T_{V_n}-\lm_0)^{-1}P_{V_n}\stackrel{s}{\tolong} (T-\lm_0)^{-1},\quad n\to\infty.
\end{equation}
In addition, the uniform bound for the resolvents in \eqref{eq:urb} and $P_{H_n}\stackrel{s}{\to} I$ show that
\begin{equation}\label{eq:gsr2}
(T_{H_n}-\lm_0)^{-1}(P_{H_n}-P_{V_n})=(T_{H_n}-\lm_0)^{-1}P_{V_n}(P_{H_n}-I)\stackrel{s}{\tolong}0,\quad n\to\infty.
\end{equation}
Now~\eqref{eq:gsr1} and~\eqref{eq:gsr2} yield $T_{H_n}\!\gsr\!T$.

It remains to consider the case $\dom(T)\!\subset\! \dom(T^*)$ in (a) where $\lambda_0 \!\in\! \bigcap_{n\in\N}\rho(T_{V_n}\!)\cap\rho(T)$ in (b) satisfies $\lambda_0\!\notin\! \Omega$.
Then Lemma~\ref{lemmadescloux} implies that the representation of $T_{H_n}$ in $H_n$ given by \eqref{eq:Hn}
is block-diagonal. Hence
\begin{align*}
(T_{H_n}-\lm_0)^{-1}P_{H_n}
&=(T_{V_n}-\lm_0)^{-1}P_{V_n}+\sum_{k=1}^{N_n}(\mu_{k;n}-\lm_0)^{-1}P_{{\rm span}\{x_{k;n}\}}
.%\\&\slong (T-\lm_0)^{-1}.
\end{align*}
Since $\lm_0$ satisfies $(T_{V_n}-\lm_0)^{-1}P_{V_n}\s (T-\lm_0)^{-1}$, it suffices to show that
$$\sum_{k=1}^{N_n}(\mu_{k;n}-\lm_0)^{-1}P_{{\rm span}\{x_{k;n}\}}\slong 0, \quad n\to\infty.$$
Since $\lm_0\notin\Omega$ by assumption, we have ${\rm dist}(\lm_0,\Omega)>0$.
By \eqref{distOmega}, the eigenvalues $\mu_{k;n}\in\sigma(T_{H_n})$, $k=1,\dots,N_n$, lie in the $2/n$-neighbourhood of $\Omega$. 
If we choose $n\in\N$ so large that $2/n<{\rm dist}(\lm_0,\Omega)/2$,
then $|\mu_{k;n}-\lm_0|\geq {\rm dist}(\lm_0,\Omega)/2$. Hence, for every $x\in H$,
\begin{align*}
\bigg\|\sum_{k=1}^{N_n}(\mu_{k;n}\!-\!\lm_0)^{-1}P_{{\rm span}\{x_{k;n}\}}x\bigg\|^2\!\!
&= \sum_{k=1}^{N_n}|\mu_{k;n}\!-\!\lm_0|^{-2}\big\|P_{{\rm span}\{x_{k;n}\}}x\big\|^2\\
&\leq \frac{4}{{\rm dist}(\lm_0,\Omega)^2}\,\big\|P_{{\rm span}\{x_{1;n},\dots,x_{N_n;n}\}}x\big\|^2\\
&\leq \frac{4}{{\rm dist}(\lm_0,\Omega)^2}\,\big\|(I\!-\!P_{V_n})x\big\|^2\!\tolong 0, \quad n\to\infty. \hspace{2mm}
\qedhere
\end{align*}
\end{proof}

\vspace{2mm}

The next theorem is the second step in the proof of Theorem \ref{thmsummary}. It 
shows that, if $W_{e1}(T)\subsetneq W_{\!e}(T)$, it is even possible to produce spectral pollution in $W_{\!e}(T)\backslash W_{e1}(T)$ 
if we allow for a modification of the given subspaces $V_n$.  
%For general $T$, this is established in the next result.

%The $\widehat\delta$ appearing in~\eqref{eq:gap} denotes the gap between subspaces, see~\cite[Section~IV.2.1]{Kato}.

\begin{theorem}\label{thmdescloux2}
Assume that $\overline{\dom(T)}=H$. %, $\rho(T)\neq\emptyset$ and $W(T)=W_{\!e}(T)=\C$.
Let $V_n\subset \dom(T)$, $n\in\N$, be  finite-dim\-ensional subspaces such that $P_{V_n}\s I$ and
let $\eps_n>0$, $n\in\N$, with $\eps_n\to 0$ as $n\to\infty$. 
Then, for any $\lm\in W_{\!e}(T)\backslash\sigma(T)$ and every $n\in\N$, there exist a finite-dimensional subspace $\widetilde V_n\!\subset\!\dom(T)$  and 
$e_n\!\in\! \widetilde V_n^{\perp}\!\cap\dom(T)$, $\|e_n\|=1$,~with 
\begin{equation}\label{eq:gap}
  \| P_{\widetilde V_n} - P_{V_n} \|<\eps_n
\end{equation}
and such that %$H_n:=\widetilde V_n\oplus{\rm span}\{e_n\}$ satisfies
\begin{equation}\label{eq:rep}
T_{H_n}=\begin{pmatrix}T_{\widetilde V_n} & B_n \\ 0 & \lm I\end{pmatrix}\quad \text{in}\quad H_n 
:=\widetilde V_n\oplus{\rm span}\{e_n\}
\end{equation}
for some linear operator $B_n: \widetilde V_n\to{\rm span}\{e_n\}$. Hence
$$
P_{H_n}\slong I, \quad 
\sigma(T_{H_n})=\sigma(T_{\widetilde V_n})\cup\{\lm\},\quad n\in\N,
$$
so if $\lm\!\in\! W_{\!e}(T)\setminus\sigma(T)$, then $\lambda$ is a spurious eigenvalue for $(T_{H_n})_{n\in\N}$.
\end{theorem}

\begin{proof}
If $\,W_{\!e}(T)=W_{e1}(T)$, all claims follow from Theorem \ref{thmdescloux} with $\widetilde V_n\!=\!V_n$, $n\in\N$.
If $\,W_{e1}(T)\subsetneq W_{\!e}(T)$, then $W(T)\!=\!\C$ by Theorem~\ref{thmequivdefforWe} and hence $W_{\!e}(T)=W(T)=\C$ by
Corollary~\ref{corlines}~iv).

If $\rho(T)=\emptyset$, there is nothing to prove, so we may assume that $\rho(T)\neq\emptyset$, without loss of generality 
$0\in\rho(T)$. This and $\overline{\dom(T)}=H$ imply that also $\overline{\dom(T^2)}=H$.

Let $\lm\in W_{\!e}(T)=\C$ be arbitrary. Let $n\in\N$ be fixed and let $\delta_n>0$ be arbitrary.
Then there exists $W_n\subset D(T^2)$ with
$$\| P_{W_n} - P_{V_n} \|< \delta_n. $$
Since $W(T)\!=\!\C$, we know that $T$ cannot be a multiple of the identity on~any finite-codimensional subspace. Thus we can choose 
$W_{\!n}$ such that $W_{\!n} \!\cap\! T W_{\!n} \!=\! \{0\}$,~i.e.
\begin{equation}\label{eq:invsub}
\forall\,w\in W_n\backslash\{0\}:\quad Tw\notin W_n.
\end{equation}

Let $\{y_1,\dots,y_{N_n}\}$ be an orthonormal basis of $W_n$. 
By induction over $k\!=\!1,$ $\dots,N_n$ we can construct 
\begin{equation}\label{eq:xk}
x_k\in \big(W_n\cup TW_n\cup \{x_j:\,j<k\}\cup\{Tx_j:\,j<k\}\big)^{\perp}\cap\dom(T^2)
\end{equation}
with $\|x_k\|=1$; by taking an appropriate linear combination of such $x_k$, and using that $T$ is injective, we can also achieve that 
 \begin{equation}\label{eq:Txk}
 Tx_k\in \big(W_n\cup TW_n\cup \{x_j:\,j < k\}\cup\{Tx_j:\,j<k\}\big)^{\perp}.
 \end{equation}
For every $t>0$, we define the pairwise orthogonal elements
$$w_k(t):=y_k+t x_k\in\dom(T^2),\quad k=1,\dots,N_n,$$
and set $W_n(t):={\rm span}\{w_1(t),\dots,w_{N_n}(t)\}$. Note that $\ran(T|_{W_n(t)})\subset\dom(T)$. Further note that the set $\{Tx_j:\,j\leq N_n\}\cup\{Ty_j:\,j\leq N_n\}$ 
is linearly independent due to the injectivity of $T$ and by~\eqref{eq:xk}.

Next we prove, by induction over $k=1,\dots,N_n$ and using Lemma~\ref{newlem}~i), that for all but finitely many $t_1,\dots,t_k,s_1,\dots,s_k\in \C$, 
the set $\{w_j(t_j):j\leq k\}\cup\{Tx_j:\,j\leq k\}\cup\{Ty_j:\,j\leq N_n\}$ is linearly independent and 
\begin{equation}\label{eq:induction}
W(T|_{(\{w_j(t_j):\,j\leq k\}\cup\{Tw_j(s_j):\,j\leq k\})^{\perp}\cap\dom(T)})=\C.
\end{equation}
As in the proof of Theorem~\ref{thmequivdefforWe}, see the proof of Claim~2) therein, we will apply Lemma~\ref{newlem}~i) successively in a sequence of subspaces $X_1 \supset X_2 \supset \dots$ of $H$ of finite codimension to which $T$ is compressed.

Let $k=1$. Since $y_1,x_1\in\dom(T)$ are linearly independent, Lemma~\ref{newlem}~i) in $H$ yields that 
\begin{equation}\label{eq:t1}
W(T|_{w_1(t_1)^{\perp}\cap\dom(T)})\!=\!\C
\end{equation}
for all but finitely many $t_1\!\in\! \C$.
Using \eqref{eq:xk},~\eqref{eq:Txk} and the property that $W_n\cap TW_n = \{0\}$, 
it is not difficult to check that, 
for all but at most one $t_1\in\R$, the set $\{w_1(t_1),Tx_1\}\cup\{Ty_j:\,j\leq N_n\}$ is linearly independent. 
We fix such a~$t_1$ that also satisfies \eqref{eq:t1},
set $X_1:=w_1(t_1)^{\perp}$ and let $P_1:H\to X_1$ be the orthogonal projection in $H$ onto~$X_1$. Then,
since $T$ is injective, it follows that 
$P_1Ty_1$, $P_1Tx_1\in\dom(T)\cap X_1$  
are non-zero and linearly independent. 
Hence Lemma~\ref{newlem}~i) in $X_1$ shows that  $W(T|_{\{w_1(t_1),Tw_1(s_1)\}^{\perp}\cap\dom(T)})=\C$ for all but finitely many $s_1\in \C$.
This proves \eqref{eq:induction} for $k=1$. 

Now assume that 
the induction hypothesis holds for some $k \,\!\in\! \{1,\dots,N_n\!-\!1\}$ and fix some admissible $t_1,\dots,t_k,s_1,\dots,s_k\!\in\! \C$. 
Then \eqref{eq:xk}, \eqref{eq:Txk} and the property that $W_n\cap TW_n = \{0\}$ 
imply that $\{y_{k+1},x_{k+1}\}\cup \{w_j(t_j):j\!\leq\! k\} \cup\{Tw_j(s_j):j\!\leq\! k\}$ is linearly independent.
Set $X_{2k}\!:=(\{w_j(t_j):j\!\leq\! k\} \cup\{Tw_j(s_j):j\!\leq\! k\})^{\perp}\!$
and let $Q_k:H\to X_{2k}$ be the orthogonal projection in $H$ onto $X_{2k}$.
Then $Q_ky_{k+1}$, $Q_kx_{k+1}\in\dom(T)\cap X_{2k}$ are non-zero and linearly independent.  
Now Lemma~\ref{newlem}~i) in $X_{2k}$ yields 
\begin{equation}\label{eq:tk}
W(T|_{(\{w_j(t_j):j=1,\dots,k+1\}\cup\{Tw_j(s_j):j=1,\dots,k\})^{\perp}\cap\dom(T)})=\C
\end{equation}
for all but finitely many $t_{k+1}\in \C$.
By the linear independence induction hypothesis, using \eqref{eq:xk}, \eqref{eq:Txk} and the injectivity of $T$, one can prove that 
for all but at most one $t_{k+1}\!\in\!\C$, the set $\{w_j(t_j):j\!\leq\! k\!+\!1\}\cup\{Tx_j:j\!\leq\! k\!+\!1\}\cup\{Ty_j:j\!\leq\! N_n\}$ is linearly independent.
We fix such a $t_{k+1}$  that also satisfies~\eqref{eq:tk},
and let $P_{k+1}:H\!\to\! X_{2k+1}$ be the orthogonal projection in $H$ onto $X_{2k+1}\!:=(\{w_j(t_j):j\!\leq\! k\!+\!1\}\cup\{Tw_j(s_j):j\!\leq\! k\})^{\perp}$.
Then $P_{k+1}Ty_{k+1}$, $P_{k+1}Tx_{k+1}\in\dom(T) \cap X_{2k+1}$ are non-zero and linearly independent.
Finally, Lemma~\ref{newlem}~i) in $X_{2k+1}$ shows that~\eqref{eq:induction} holds for $k+1$ and for all but finitely many $s_{k+1}\in\C$. This proves the induction step. 

From~\eqref{eq:induction} with $k=N_n$ and letting $t_1=\dots=t_{N_n}=t$ and $s_1=\dots=s_{N_n}=t$, we conclude that, for all but finitely many $t\in \C$,
$$ 
  W(T|_{(W_n(t)\cup \ran(T|_{W_n(t)}))^{\perp}\cap\dom(T)})=\C.
$$
Thus for all but finitely many $t\in\C$, there exists $e(t)\!\in\! (W_n(t)\cup \ran(T|_{W_n(t)}))^{\perp}\cap\dom(T)$, $\|e(t)\|=1$, with 
$\langle Te(t),e(t)\rangle=\lm$. Note that $\langle T w,e(t)\rangle=0$ for every $w\in W_n(t)$. Now we choose $\delta_n$ and $t$ so small that 
\eqref{eq:gap} and \eqref{eq:rep} hold with $\widetilde V_n:=W_n(t)$ and $e_n:=e(t)$. 
From the representation \eqref{eq:rep}, it follows that $\sigma(T_{H_n})=\sigma(T_{\widetilde V_n})\cup\{\lm\}$. 
The property \eqref{eq:gap} and $\eps_n\to 0$ imply $\|P_{V_n}-P_{\widetilde V_n}\|\to 0$.
Together with $P_{V_n}\s I$, this yields $P_{\widetilde V_n}\s I$, and hence $P_{H_n}\s I$ since $\widetilde V_n \subset H_n$. 
\end{proof}

\begin{proof}[Proof of Theorem {\rm \ref{thmsummary}}]
Let $V_n\subset\! \dom(T)$, $n\!\in\!\N,$ be arbitrary finite-dimensional sub\-spaces with $P_{V_n}\!\s\! I$. 
If $W_{e1}(T)\!=\!W_{\!e}(T)$, we apply Theorem~\ref{thmdescloux}; if $W_{e1}(T)\!\subsetneq\! W_{\!e}(T)$ we 
%\marginpar{\tiny\ct{deleted `can' twice}}
apply Theorem~\ref{thmdescloux2}, to complete the proof of Theorem~\ref{thmsummary}.
\end{proof}

The next example gives an explicit construction of the subspaces $H_n$, $n\in\N$, in Theorem~\ref{thmdescloux} 
so that the corresponding projection method has a given point $\lambda \in W_{e1}(T) \backslash \sigma(T)$ (even 
$\lambda \in W_{\!e}(T) \backslash \sigma(T)$ if $W(T)\ne \C$) as a spurious eigenvalue.

\begin{example}\label{ex.delay}
Let $A:=T+S$ where $T$, $S$ are the neutral delay differential operators introduced in Remark \ref{neutral-delay-diff} with their matrix representations in Example~\ref{Ex2} with respect to ${\rm span}\{\cos(k\cdot),\sin(k\cdot):\,k\in\N\} \subset\dom(A)$.

It is not difficult to check that the spectrum of the lower triangular infinite matrix $A$ is given by its diagonal entries, 
$\sigma(A)=\{k^2:\,k\in\N\}$ and $\sigma_e(A)=\{1\}$. For the latter note that, while for $k\ge 2$ all eigenvalues $k^2$ are simple, $1$ is an eigenvalue of 
infinite geometric multiplicity (with one two-dimensional algebraic eigenspace), and $\sigma_e(A) \subset 
\bigcap_{K\in L(H), \atop K\,\text{compact}} \sigma(A+K)= \{1\}$. 

According to Example~\ref{Ex2} the essential numerical range of $A$ is given by 
$$
  W_e(A)= \left\{ \lm\in\C: \re\,\lm \geq \frac 34, |\im\,\lm| \leq \sqrt{\re\,\lm - \frac 34}\right\}
  %\{z\in\C:\,\re\,z\geq 3/4+(\im\,z)^2\}
  =\left\{|\gamma|^2+1+\gamma:\,\gamma\in\C\right\}.
$$
In particular, $W(A)\neq\C$ and hence assumption (a) of Theorem~\ref{thmdescloux} is satisfied.
For the projection method onto the subspaces 
$$V_n:={\rm span}\{\cos(k\cdot),\sin(k\cdot):\,k=1,\dots,n\}, \quad n\in\N,$$
Theorem~\ref{thmgalerkin}~i) shows that $W_e\left((A_{H_n})_{n\in\N}\right)=W_e(A)$, and 
it is easy to see that $\sigma(A_{V_n})=\{k^2:\,k=1,\dots,n\}$. Thus, 
for every $\lm_0\in\rho(A)$ we have $\lm_0\in\rho(A_{V_n})$, $n\in\N$, and
$$(A_{V_n}-\lm_0)^{-1}P_{V_n}\slong (A-\lm_0)^{-1}, \quad n\to\infty.$$
Hence also assumption (b) of Theorem~\ref{thmdescloux} is satisfied. 
According to Theorem~\ref{thmdescloux}, for every $\lm\in W_e(A)$, there exist finite-dimensional extensions $H_n\supset V_n$, 
$n\in\N$, and $\lm\in\sigma(A_{H_n})$, $n\in\N$, with $\lm_n\to\lm$.

In fact, if $\lambda \in W_e(A)$, then there exists $\gamma\in\C$ so that $\lm=|\gamma|^2+1+\gamma$. If we set 
$$f_n:=\frac{\overline{\gamma}}{n}\cos(n\cdot)+\sin(n\cdot), \quad n\in\N,$$
and $H_n:=V_n\oplus{\rm span}\{f_{n+1}\}$, then 
$$\sigma(A_{H_n})=\sigma(A_{V_n})\cup \{\lm_n\}, \quad \lm_n:=\frac{\langle A f_{n+1},f_{n+1}\rangle}{\|f_{n+1}\|^2}\tolong |\gamma|^2+1+\gamma=\lm.$$
Since the subspaces $V_n$, $n\in\N$, are invariant under $A$, Remark \ref{proj-sigma_e_lim} and the decomposition $H_n=V_n\oplus{\rm span}\{f_{n+1}\}$ yield that
$$\sigma_e\left((A_{H_n})_{n\in\N}\right)=\sigma_e(A)\cup\{\lim_{n\to\infty}\lambda_n\}
%=\{z\in\C:\,\re\,z= 3/4+(\im\,z)^2\}
=\{1\} \cup\{\lambda\}.$$
So, starting from a given projection method onto subspaces $V_n$, $n\in\N$, for an arbitrary point $\lambda \in W_e(A) \setminus \sigma(A)$ we have explicitly constructed a projection method onto subspaces $H_n \supset V_n$ with $\lambda$ as a point of spectral pollution.

Note that here the inclusion ${\rm conv}\,\sigma_e\left((A_{H_n})_{n\in\N}\right) \subsetneq W_e\left((A_{H_n})_{n\in\N}\right)$ is strict since, by Remark \ref{proj-sigma_e_lim} and Theorem~\ref{thmgalerkin}~i), 
$$
{\rm conv}\sigma_{\!e}\!\left((A_{H_{\!n}})_{n\in\N}\right) \!=\! {\rm conv}\sigma_{\!e}(A) \!=\! \{1\} \!\subsetneq\!  \left\{|\gamma|^2\!\!+\!1\!+\!\gamma\!:\!\gamma\!\in\!\C\right\} \!=\! W_{\!e}(A) \!=\!
W_{\!e}\!\left((A_{H_{\!n}})_{n\in\N}\right)\!.\!
$$
\end{example}

\smallskip

The following example shows that Theorem \ref{thmdescloux2} is sharp in the following sense. 
Without the modification of the subspaces $V_n$ it may happen that, if the inclusion $W_{e1}(T)\subset W_{\!e}(T)$ is strict, 
only points $\lambda\in W_{e1}(T)$ can be arranged to be spurious eigenvalues. Recall that $W_{e1}(T) \subsetneq W_{\!e}(T)$ necessitates $W_{\!e}(T)=W(T)=\C$.

\begin{example}\label{ex:nopollution}
In Example \ref{exequivWe} we considered a selfadjoint operator $T_0$ in $H$ with $\sigma(T_0)=\sigma_{e}(T_0)=\R$ and $S$ with $\dom(S)=\dom(T_0)$ 
is of the form $S=Q\Phi$ where $\Phi : H\to\C$ is an unbounded linear functional which is $T_0$-bounded and $Q:\C\to H$, $Qz=zg$ with fixed 
$g\in H \setminus \{0\}$. We showed that $W_{e1}(T) = \R \subsetneq \C = W_{\!e}(T)$. Moreover, since $S$ is $T_0$-compact, $\sigma(T)=\sigma_e(T)=\sigma_{e}(T_0)=\R$.
Here we wish to choose $g \in {\rm ker}\,T_0 \cap {\rm ker}\, \Phi$. This can be achieved, e.g.\ by choosing $T_0$ such that $\ker T_0 \ne \{0\}$, 
$\Phi := \langle T_0 \,\cdot\,,y \rangle$ with $y \notin \dom(T_0)$ and $g \in \ker T_0$.

Let $V_n\!\subset\! \dom(T)$, $n\!\in\! \N$, with $P_{V_n}\!\s\! I$ be such that $g\!\in\! V_n$, $n\!\in\!\N$. 
Suppose~that $\Omega\!\subset\! W_{\!e}(T)\backslash W_{e1}(T)\!=\!\C\backslash\R$ is compact, and assume there exist subspaces $H_n\!\supset\! V_n$ as in  Theo\-rem~\ref{thmdescloux} filling $\Omega$ with spectral pollution, 
i.e.\ $\sup_{\lambda\in\Omega}{\rm dist}(\lm,\sigma(T_{H_n}))\!\to\! 0$ as $n\!\to\!\infty$. Because $g\!\in\! V_n\!\subset\! H_n$, we can write $H_n\!=\!{\rm span}\{g\}\oplus U_{\!n}$. 
Then $Tg\!=\!0$, $S g \!=\! \Phi(g) g\!=\!0$ by the choice of $g$, $S_{U_{\!n}}\!=\!0$ since $U_n \!\perp\! g$ and so, because $T_0$ is selfadjoint, 
$$\sigma(T_{H_n})=\{0\} \cup\,\sigma(T_{U_{\!n}})\subset \{0\} \cup W(T_{U_{\!n}})= \{0\} \cup W(T_{0,{U_{\!n}}})\subset\R.$$
Since $\Omega\subset \C \backslash \R$ is compact, this contradicts $\sup_{\lambda\in\Omega}{\rm dist}(\lm,\sigma(T_{H_n}))\!\to\! 0$ as $n\!\to\!\infty$.
Hence no such subspaces $H_n$, $n\in\N$, can exist.
\end{example}

\section{Application II: Domain truncation method}\label{sectiondomaintruncation}

In this section we study spectral exactness of domain truncation methods for strongly elliptic partial
differential operators $A$ in $L^2(\R^d)$ 
with arbitrary dimension $d\in\N$. We show that, for domain truncation to bounded nested $\Omega_n$ exhausting $\R^d$ and Dirichlet conditions, 
spectral pollution is confined to the essential numerical range $W_e(A)$ and every isolated $\lambda \in \sigma(A)$ is approximated. 

More precisely, we consider a strongly elliptic 
differential operator $A$ of even order $2m\in\N$, induced by the quadratic form 
\beq
\label{qformA}
a[f]:=\sum_{\alpha, \beta \in\N_0^d\atop |\alpha|+|\beta|\leq 2m}
\big\langle Q_{\alpha,\beta}\frac{1}{\I^{|\alpha|}}D^{\alpha}f,\frac{1}{\I^{|\beta|}}D^{\beta}f\big\rangle,
\quad \dom(a):=H^{m}(\R^d),
\vspace{-1mm}
\eeq
with coefficients $Q_{\alpha,\beta}\in L^{\infty}(\R^d)$ for all $\alpha,\beta\in\N_0^d$ with $|\alpha|+|\beta|\leq 2m$
and constant leading coefficients, $Q_{\alpha,\beta}:=c_{\alpha,\beta}\in\C$ if $|\alpha|+|\beta|=2m$. This means the associated \emph{principal symbol} 
$$p_{2m}(\xi):=\sum_{\alpha,\beta \in\N_0^d\atop |\alpha|+|\beta|= 2m}c_{\alpha,\beta}\xi^{\alpha+\beta}, \quad \xi\in\R^d,$$
is independent of $x\in\R^d$ and satisfies 
\beq
\label{ass.elliptic}
\re p_{2m}(\xi)> 0, \quad \xi\in\R^d\backslash\{0\};
\eeq
since $p$ is homogeneous, this implies that 
$p_{2m}$ is sectorial, i.e.\ there exist $a_{2m}$, $b_{2m}\!\geq\! 0$ with
\beq
\label{eq.relbd}
  |\im p_{2m}(\xi)|\leq a_{2m}+b_{2m} \re p_{2m}(\xi), \quad \xi\in\R^d.
\eeq

Note that \eqref{qformA} allows for both divergence form (i.e.\ $Q_{\alpha,\beta}=0$ if $|\alpha|>m$ or $|\beta|>m$) and non-divergence form
(i.e.\ $Q_{\alpha,\beta}=0$ if $\beta>0$) with $L^\infty(\R^d)$-coefficients.

\begin{theorem}\label{thmtruncation}
Let $\Omega_n\subset\R^d$, $n\in\N$, be bounded nested domains exhausting $\R^d$ and, if $d\geq 2$, with boundaries of class~$C$.  
Then the $m$-sectorial operators $A$, $A_n$, $n\in\N$, associated with the densely defined, closed and sectorial forms $a$ and 
$a_n:=a|_{H_0^{m}(\Omega_n)}$  in $L^2(\R^d)$ and $L^2(\Omega_n)$, $n\in\N$, respectively, satisfy the following:
\begin{enumerate}
\item[\rm i)]
$A_n$, $n\!\in\!\N$, have compact resolvents, 
$A_n\!\gsrlong\! A$ as well as $A_n^*\!\gsrlong\! A^*$, and
\beq\label{eq.Weinclusion}
W_e\left((A_n)_{n\in\N}\right)=W_e(A).%\vspace{-3mm}
\eeq

\item[\rm ii)] 
Spectral pollution is confined to $W_e(A)$,
$$ 
  \sigma_{poll}((A_n)_{n\in\N}) \subset W_e(A).
$$

\item[\rm iii)] 
Every isolated $\lm\in\sigma(A)$ outside $W_e(A)$ is approximated by $(A_n)_{n\in\N}$.
\end{enumerate}
\end{theorem}

\vspace{1mm}

\begin{proof}
i)
First we define the principal part of $A$,
\beq \label{princpart}
T:=p_{2m}\left(\frac{1}{\I}D\right)=
\sum_{\alpha, \beta \in\N_0^d\atop |\alpha|+|\beta|= 2m} c_{\alpha,\beta}\frac{1}{\I^{N}}D^{\alpha+\beta},
\quad \dom(T):=H^{2m}(\R^d).
\eeq
Since $A$ is strongly elliptic, and hence \eqref{ass.elliptic}, \eqref{eq.relbd} hold, $T$ is $m$-sectorial and $C_0^\infty(\R^d)$ is a core of~$T$, see~\cite[Prop.~IX.6.4, Cor.~IX.6.7]{edmundsevans}.
The corresponding quadratic form 
$$t[f]:=\sum_{\alpha, \beta \in\N_0^d\atop |\alpha|+|\beta|= 2m}
\big\langle c_{\alpha,\beta}\frac{1}{\I^{|\alpha|}}D^{\alpha}f,\frac{1}{\I^{|\beta|}}D^{\beta}f\big\rangle,
\quad \dom(t):=H^{m}(\R^d),$$
is densely defined, closed and sectorial.
Since $Q_{\alpha,\beta}\in L^{\infty}(\R^d)$, Fourier analysis reveals that
the quadratic form $s:=a-t$ is $\re t$-bounded with relative bound~$0$, i.e.\ for every $\eps>0$ there exists $a_{\eps}>0$, without loss of generality $a_{\eps}\geq 1$, such that
\beq\label{eq:quad.bound}
|s[f]|\leq a_{\eps}\|f\|^2+\eps \re t[f],\quad f\in H^{m}(\R^d).
\eeq
Hence~\cite[Theorem~VI.3.4]{Kato} implies that $a$ is also densely defined, closed and sectorial, and the associated $m$-sectorial operator $A$ satisfies, for $\eps\in (0,1/2)$,
$$
\|(A-\lm)^{-1}-(T-\lm)^{-1}\|\leq \frac{2\eps}{(1-2\eps)|\lm|},\quad \lm\in\C,\ \re\lm\leq -\frac{a_{\eps}}{\eps}.
$$

Next we introduce $T_n$, $t_n$, $s_n$, $n\in\N$, in the same way as $T$, $t$, $s$ but with~domains
$$\dom(T_{n}):=H^{2m}(\Omega_n)\cap H_0^{m}(\Omega_n),\quad \dom(t_n)=\dom(s_n):=H_0^{m}(\Omega_n).$$
Note that, $H_0^{m}(\Omega_n)\subset H^{m}(\R^d)$ if we extend every function by zero outside $\Omega_n$. Therefore~\eqref{eq:quad.bound} and~\eqref{eq:res.diff} continue to hold if we replace $s,t,A,T$ by $s_n,t_n,A_n,T_n$. 
Let $f\in\dom(A)$. We construct $f_n\in\dom(A_n)$, $n\in\N$, so that
\beq \label{eq.discconv}
\|f_n-f\|\tolong 0, \quad \|A_nf_n-Af\|\tolong 0,\quad n\to\infty.\eeq
To this end, let $\eps\in (0,1/2)$ and $\lambda_{\eps}\in (-\infty,-a_{\eps}/\eps)$. Then
\beq\label{eq:res.diff}
\hspace{0mm}
\max\!\left\{\!\|(A\!-\!\lm_{\eps})^{-1}\!\!-\!(T\!-\!\lm_{\eps})^{-1}\|,\sup_{n\in\N}\|(A_n\!-\!\lm_{\eps})^{-1}\!\!-\!(T_n\!-\!\lm_{\eps})^{-1}\|\!\right\}\!
\leq\! 
\frac{2\eps}{(1\!-\!2\eps)|\lm_\eps|}.
\hspace{-3mm}
\eeq
For every $\phi\in C_0^{\infty}(\R^d)$ there exists $n_0(\phi)\in\N$ such that
$$\forall\,n\geq n_0(\phi):\quad \phi|_{\Omega_n}\in C_0^{\infty}(\Omega_n)\subset\dom(T_n), \quad T_n(\phi|_{\Omega_n})=(T\phi)|_{\Omega_n}.$$
The $m$-sectoriality of $T$, $T_n$, $n\in\N$, implies  that
$\sup_{n\in\N}\|(T_n-\lm_\eps)^{-1}\|<\infty$ and, using that $C_0^{\infty}(\R^d)$ is a core of $T$ and~\cite[Theorem~3.1]{boegli-chap1},
$$\forall\,g\in L^2(\R^d):\quad \left\|\left((T_n-\lm_\eps)^{-1}\chi_{\Omega_n}- (T-\lm_\eps)^{-1}\right)g\right\|\tolong 0, \quad n\to\infty.$$
Define 
\beq \label{fn}
f_n:=(A_n-\lm_\eps)^{-1}\chi_{\Omega_n}(A-\lm_\eps)f\in\dom(A_n),\quad n\in\N.
\eeq
Then~\eqref{eq:res.diff} and the inequalities $\lambda_{\eps}<-a_{\eps}/\eps\leq -1/\eps$ yield
\begin{align*}
\|f_n-f\|
&\leq \|\left((A_n-\lm_\eps)^{-1}-(T_n-\lm_\eps)^{-1}\right)\chi_{\Omega_n}(A-\lm_\eps)f\|\\
&\quad + \|\left((T_n-\lm_\eps)^{-1}\chi_{\Omega_n}-(T-\lm_\eps)^{-1}\right)(A-\lm_\eps)f\|\\
&\quad + \|\left((A-\lm_\eps)^{-1}-(T-\lm_\eps)^{-1}\right)(A-\lm_\eps)f\|\\
&\leq \frac{4\eps}{(1\!-\!2\eps)|\lm_\eps|}\|(A\!-\!\lm_\eps)f\|+ \left\|\left((T_n\!-\!\lm_\eps)^{-1}\chi_{\Omega_n}\!-\!(T\!-\!\lm_\eps)^{-1}\right)
(A\!-\!\lm_\eps)f\right\|\\
&\leq \frac{4\eps}{1\!-\!2\eps}(\eps\|Af\|+\|f\|)+ \left\|\left((T_n\!-\!\lm_\eps)^{-1}\chi_{\Omega_n}\!-\!(T\!-\!\lm_\eps)^{-1}\right)(A\!-\!\lm_\eps)f\right\|.
\end{align*}
By taking first $\eps$ small enough and then $n$ large enough, the right hand side can be made arbitrarily small, which proves the first convergence in~\eqref{eq.discconv}.
The second convergence in~\eqref{eq.discconv} follows from
\begin{align*}
\|A_nf_n-Af\|&=\|\chi_{\Omega_n}(A-\lm_\eps)f+\lambda_\eps f_n-Af\|\\
&\leq \|(I-\chi_{\Omega_n})(A-\lm_\eps)f\|+|\lm_\eps|\|f_n-f\|\tolong 0,\quad n\to\infty.
\end{align*}
Now fix $\eps \!\in\! (0,1/2)$ and $\lm_{\eps}\!<\!-a_\eps/\eps$ as in the construction of $f_n$ in \eqref{fn} satisfying~\eqref{eq.discconv}. Then the $m$-sectoriality of $A_n$, $n\in\N$, implies $\sup_{n\in\N}\|(A_n-\lm_\eps)^{-1}\|\!<\!\infty$.
Let $g\in L^2(\R^d)$ and define $f:=(A-\lm_\eps)^{-1}g\in\dom(A)$.
Then, by~\eqref{eq.discconv}, \eqref{fn},  
\begin{align*}
&\|(A_n-\lm_{\eps})^{-1}\chi_{\Omega_n}g-(A-\lm_{\eps})^{-1}g\|\\
%&=\|(A_n-\lm)^{-1}\chi_{\Omega_n}(A-\lm)f-f\|
&=\|(A_n-\lm_{\eps})^{-1}\chi_{\Omega_n}\left((A-\lm_{\eps})f-(A_n-\lm_{\eps})f_n\right)+f_n-f\|\\
&\leq \|(A_n-\lm_{\eps})^{-1}\|\left(\|Af-A_nf_n\|+|\lm_{\eps}|\|f-f_n\|\right)+\|f_n-f\|
\tolong 0,\quad n\to\infty.
\end{align*}
This proves $A_n\gsr A$.
 Since the adjoint quadratic form $a^*$ satisfies
\eqref{ass.elliptic}, \eqref{eq.relbd} as well, we analogously have $A_n^*\gsr A^*$.

The strong ellipticity of $A$ and~\cite[Proposition~IX.6.4]{edmundsevans} imply that 
$T_n$ is $H_0^{m}(\Omega_n)$-coercive. Hence, by the compactness of the embedding of $H^{1}(\Omega_n)$ in $L^2(\Omega_n)$, 
see e.g.\ \cite[Theorem~4.17]{edmundsevans}, the operator $T_n$ has compact resolvent, and so has $A_n$ by~\cite[Theorem~VI.3.4]{Kato} for every $n\in\N$.

Next we use that $\dom(A_n)\subset\dom(a_n)\subset\dom(a)$ and
\begin{equation}\label{eq:afn}
\forall\,f_n\in\dom(A_n):\quad \langle A_nf_n,f_n\rangle=a[f_n].
\end{equation}
Since  $\dom(A)$ is a core of $a$ by~\cite[Theorem~VI.2.1]{Kato},
it follows that $W(A_n)\subset \overline{W(A)}$ and hence, with Proposition~\ref{prop:limitingWe},
$$W_e\left((A_n)_{n\in\N}\right)\subset W_e(A).$$
Now equality in \eqref{eq.Weinclusion} follows from  Proposition~\ref{propWeTnprop}~i).

ii), iii) Let $\lm\in\C\backslash W_e(A)$.
Then~\eqref{eq.Weinclusion} and
Proposition~\ref{propWeTnprop}~ii) imply that $\lm\notin\sigma_e\left((A_n)_{n\in\N}\right)$.
Now the claim follows from Theorem~\ref{mainthmspectralexactness}~ii) applied to the adjoint operators  
if we note that $\sigma_{poll}((T_n^*)_{n\in\N})=\sigma_{poll}((T_n)_{n\in\N})^*$, which is immediate from Definition \ref{defspectralexactness}~ii).
\end{proof}

In some cases $W_e(A)$ can be determined explicitly, e.g.\ if $A$ is in non-divergence form or in divergence form, and the coefficients 
are asymptotically constant. In this case, although $A$ has complex coefficients and is not selfadjoint, the next proposition shows that $W_e(A)$ is the convex hull of the range of the asymptotic symbol and hence the convex hull of the essential spectrum of $A$.

\begin{prop}\label{prop.constcoeff}
Suppose that $A$ with domain $\dom(A)$ in Theorem{\rm~\ref{thmtruncation}} 
is either given in non-divergence form, i.e.\ $Q_{\alpha,\beta}=0$ if $\beta>0$, 
$$
A=\sum_{\alpha\in\N_0^d\atop |\alpha|\leq 2m}Q_{\alpha,0}\frac{1}{\I^{|\alpha|}}D^{\alpha}, 
\vspace{-1mm}
$$
or in divergence form, i.e.\ $Q_{\alpha,\beta}=0$ if $|\alpha|>m$ or $|\beta|>m$,
$$
\hspace{11mm} A = \hspace{-1.5mm} \sum_{\alpha, \beta \in\N_0^d\atop |\alpha|, |\beta|\leq m} \hspace{-1.5mm}
\frac{1}{\I^{|\alpha|+|\beta|}}D^{\beta} Q_{\alpha,\beta}D^{\alpha}, 
\vspace{-1mm}
$$
and that, in both cases, there exist $c_{\alpha,\beta}\in\C$ with
\beq\label{eq.decay}
Q_{\alpha,\beta}(x)\tolong c_{\alpha,\beta}, \quad |x|\to\infty, \qquad \alpha, \beta \in\N_0^d,\ |\alpha|+|\beta|< 2m.
\eeq
Then, if we denote the limiting symbol of $A$ by 
$$p_\infty(\xi):=p_{2m}(\xi)\ + \hspace{-3mm} \sum_{\alpha\in\N_0^d\atop |\alpha|+|\beta|<2m} \hspace{-3mm} c_{\alpha,\beta}\,\xi^{\alpha}, 
\quad \xi\in\R^d,
\vspace{-3mm}
$$
we have
$$
W_e(A)={\rm conv}\,\big\{p_\infty(\xi):\,\xi\in\R^d\big\} ={\rm conv}\, \sigma_e(A).
$$
\end{prop}

\begin{proof}
Let $T$ be the principal part of $A$, see \eqref{princpart} in the proof of Theorem~\ref{thmtruncation}, 
and define 
\beq \label{lower-order}
  A_\infty:=T+\hspace{-3mm}\sum_{\alpha,\beta\in\N_0^d\atop |\alpha|+|\beta|<2m}\hspace{-3mm}c_{\alpha,\beta}\frac{1}{\I^{|\alpha|+|\beta|}}D^{\alpha+\beta}\!\!\!,\quad\dom(A_\infty):=H^{2m}(\R^d), \quad A_0:=A\!-\!A_\infty.
\eeq
Since $A_\infty$ has constant coefficients, 
Fourier analysis and~\cite[Proposition~IX.6.4]{edmundsevans} yield $\overline{W(A_\infty)}={\rm conv}\,\sigma(A_\infty)$ and $\sigma(A_\infty)=\sigma_e(A_\infty)=\{p_\infty(\xi):\,\xi\in\R^d\}$. Now the sequence of inclusions $W_e(A_\infty)\subset\overline{W(A_\infty)}={\rm conv}\,\sigma_e(A_\infty)\subset W_e(A_\infty)$ implies that all sets therein coincide. 
Hence it remains to be shown that 
$W_e(A)=W_e(A_\infty)$ and $\sigma_e(A)=\sigma_e(A_\infty)$.

First we consider the case that $A$ is in non-divergence form. By definition \eqref{lower-order} the differential operator $A_0$ has order less than $2m$ and all the coefficients of its symbol $p_0(x,\xi) = \sum_{\alpha\in\N_0^d, |\alpha|<2m} ( Q_{\alpha,0}(x) - c_{\alpha,0})\xi^\alpha$ tend to $0$ for $|x|\to\infty$ by~\eqref{eq.decay}, and hence so do the coefficients of the real and imaginary part $\re p_0(x,\xi)$ and $\im p_0(x,\xi)$, respectively.
If we denote the differential operators induced by $\re p_0(x,\xi)$, $\im p_0(x,\xi)$ and $\re p_\infty(\xi)$ by $\re A_0$, $\im A_0$ and $\re A_\infty$, respectively, it follows that $\re A_0$ and $\im A_0$ are $\re A_\infty$-compact, see~\cite[Thm.~5.5.4]{Schechter} or~\cite[Thm.~IX.8.2]{edmundsevans}. 
Now Theorem~\ref{thmABUV} yields
\beq
\label{eq.Wediff}
W_e(A)=W_e(A_\infty + A_0)=W_e(A_\infty).
\eeq
In addition, we also have that $A_0$ is $A_\infty$-compact, which implies
\beq\label{eq.symbol}
\sigma_e(A)=\sigma_e(A_\infty)=\{p_\infty(\xi):\,\xi\in\R^d\}. 
\eeq
If $A$ is in divergence form, assuming $0\in \rho(\re A_\infty)$ after a possible shift of the spectral parameter, we can write 
$$
(\re A_\infty)^{-1/2}A_0(\re A_\infty)^{-1/2}=\sum_{\alpha,\beta\in\N_0^d\atop |\alpha|+|\beta|\leq 2m} F_{\beta}^*G_{\alpha,\beta}
\vspace{-4mm}
$$
with the bounded operators
$$
F_{\beta}:=\frac{1}{\I^{|\beta|}}D^{\beta}(\re T)^{-1/2}, \quad 
G_{\alpha,\beta}:=(Q_{\alpha,\beta}-c_{\alpha,\beta})\frac{1}{\I^{|\alpha|}}D^{\alpha}(\re T)^{-1/2}.
$$
Due to assumption \eqref{eq.decay}, $G_{\alpha,\beta}$ is compact, see \cite[Thm~IX.8.2]{edmundsevans}, and hence so is $(\re A_\infty)^{-1/2}A_0(\re A_\infty)^{-1/2}$. Now Theorem~\ref{thmSqrt} yields \eqref{eq.Wediff}.
Similarly, one can show that, for $\lambda \!\in\! \C$, $\re \lambda \!<\! 0$, $(A_\infty\!-\!\lambda)^{-1/2}A_0 (A_\infty\!-\!\lambda)^{-1/2}$ is compact, noting that~$(A_\infty-\lambda)^{1/2}$ exists because $A_\infty-\lambda$ is $m$-sectorial, see~\cite[Sect.~V.3.11]{Kato}). 
It is easy to check, using that $A=A_0+A_\infty$, that 
\begin{align*}
&(A-\lm)^{-1}- (A_\infty-\lm)^{-1}\\
&=-(A_\infty\!-\!\lm)^{-1}A_0(A_\infty\!-\!\lm)^{-1/2}(I\!+\!(A_\infty\!-\!\lm)^{-1/2}A_0(A_\infty\!-\!\lm)^{-1/2})^{-1}(A_\infty\!-\!\lm)^{-1/2}.
\end{align*}
which shows that the difference of the resolvents of $A$ and $A_\infty$ is compact, and so \eqref{eq.symbol} follows.
\end{proof}

\begin{example}
%Let $d\!=\!1$ and 
Consider the advection-diffusion type %perturbed constant-coefficient 
\vspace{-1mm} differential~operator 
$$
 A:=-\frac{\rd^2}{\rd x^2}+Q_1 \,\frac{\rd}{\rd x}+Q_0, \quad \dom(A):=H^{2}(\R),
\vspace{-1mm}
$$
with complex-valued coefficients $Q_1$, $Q_0 \!\in\! L^{\infty}(\R)$ such that 
$$ 
Q_1(x) \to -2, \quad Q_0(x)\to 0,  \qquad |x|\tolong\infty.
$$
Then we have, see \vspace{-1mm} \eqref{eq.symbol}, 
\beq \label{parabola-ex-last}
  \sigma_e(A)=\left\{\lm\in\C:\,\re\lm=\frac{(\im\lm)^2}{2}\right\}
\vspace{-1mm}
\eeq
and Theorem~\ref{thmtruncation}~ii), iii) together with 
Proposition~\ref{prop.constcoeff} yield that, for the truncated operators $A_n$, $n\in\N$, on intervals $(a_n,b_n)$ with $a_n \to -\infty$, $b_n \to \infty$ as $n\to\infty$, and Dirichlet conditions at $a_n$, $b_n$, 
\beq \label{We-last-ex}
\sigma_{poll}((A_n)_{n\in\N}) \subset 
 W_e(A)={\rm conv}\,\sigma_e(A)=\left\{\lm\in\C:\,\re\lm \ge \frac{(\im\lm)^2}{2}\right\}
\eeq
and every isolated $\lm\in\sigma(A)$ outside the parabolic region on the right hand side is approximated by \vspace{-4mm}$(A_n)_{n\in\N}$. 
\begin{figure}[htb]
\begin{center}
\includegraphics[width=0.58\textwidth]{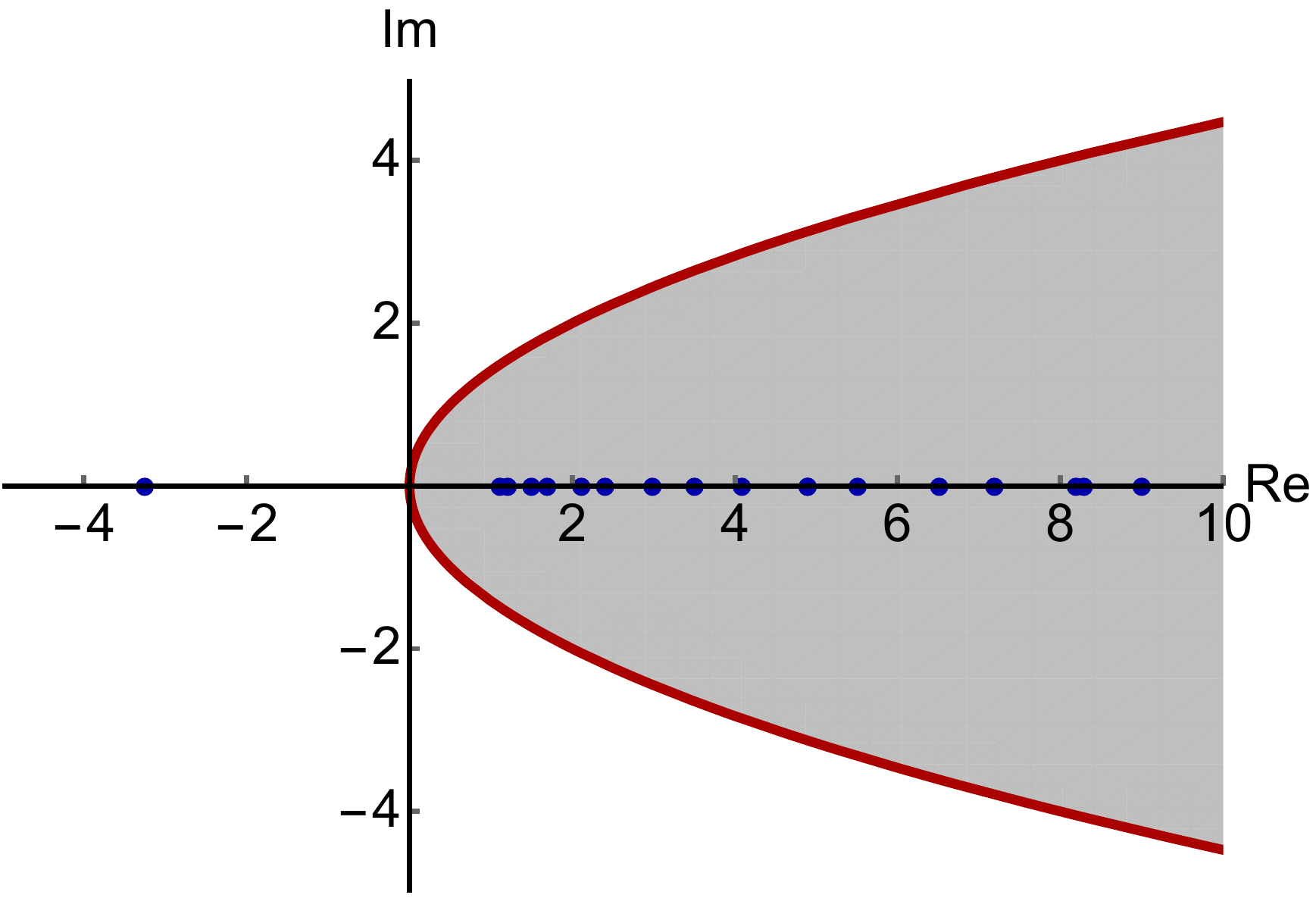}
\caption{\small Eigenvalues of $A_n$ with $Q_1(x)\!=\!-2$, $Q_0(x)\!=\!20\sin(x)\e^{-x^2}$ 
trun\-cated to $[-s_n,s_n]\!=\![-9,9]$ (blue/black points in~$\R$), $\sigma_e(A)$ (red/ black curve) 
and $W_e(A)$ (grey parabolic region with red/black~curve).}
\label{fig.CCDirichletPlane}
\end{center}
\end{figure}

If $Q_1$, $Q_0$ are real-valued, then 
interval truncation with Dirichlet boundary conditions can only produce real eigenvalues.
Indeed, in this case, if e.g.\  $Q_1 \in W^{1,\infty}(\R)$ each truncated operator $A_n$ can be transformed to an operator $\widetilde A_n$ 
still satisfying our assumptions which has the same eigenvalues 
without first order term and with real-valued potential, given by $\widetilde A_n = -\frac{\rd^2}{\rd x^2} + \widetilde Q_0$ with 
$\widetilde Q_0 = Q_0 - \frac 12 Q_1' + \frac 14 (Q_1)^2 \in L^\infty(\R)$ subject to Dirichlet boundary conditions. 
The transformed operator $\widetilde A_n$ has real numerical range satisfying  
\[
  \sigma_p(A_n) \subset W(\widetilde A_n) \subset \Big[ {\rm ess\,inf} \Big (Q_0 - \frac 12 Q_1' + \frac 14 (Q_1)^2 \Big),\infty \Big), \quad n \in \N,
\]
independently of $n$. Together with our new result \eqref{We-last-ex}, this shows that, in this case, 
spurious eigenvalues are confined to  
\begin{align}
  \sigma_{poll}((A_n)_{n\in\N}) 
	& \subset W_e(A) \cap \Big[ {\rm ess\,inf} \Big (Q_0 - \frac 12 Q_1' + \frac 14 (Q_1)^2 \Big),\infty \Big) \nonumber \\
	&= \Big[ \max \Big\{ 0,  {\rm ess\,inf} \Big (Q_0 - \frac 12 Q_1' + \frac 14 (Q_1)^2 \Big) \Big\} , \infty \Big), 
	\label{poll-last-ex}
\end{align}
but also that the approximation $(A_n)_{n\in\N}$ is not spectrally inclusive since no non-real spectral point  $\lambda \in\sigma(A)\backslash\R$, and
so, in particular, none of the non-zero points on the parabola $\sigma_e(A)$, is approximated. 

That the inclusion \eqref{poll-last-ex} is sharp follows if we consider the 
special case $Q_1\equiv -2$, $Q_0 \equiv 0$.
Here $\widetilde Q_0 \equiv 1$  so that \eqref{poll-last-ex} yields $\sigma_{poll}((A_n)_{n\in\N}) \subset [1,\infty)$.   
As remarked by Davies~\cite{davies-CCDiffOp}, the set of eigenvalues of $A_n$ is given by $\sigma(A_n)=\{1 + \frac{\pi^2 k^2}{4 s_n^2}: k \in \N\}$ 
and hence the set of accumulation points $\lambda = \lim_{n\to\infty} \lambda_n$ with $\lambda_n \in \sigma(A_n)$ is the whole interval 
$[1,\infty)$; since here $\sigma(A)=\sigma_e(A)$ is the parabola in \eqref{parabola-ex-last}, $[1,\infty)$ consists entirely of spurious eigenvalues.

Another interesting example is the special case $Q_1\equiv -2$, $Q_0(x):=20\sin(x)\e^{-x^2}$, $x\in\R$, considered in 
\cite{boegli-limitingess}. 
Here $\widetilde Q_0(x) = Q_0(x) + 1$, $x\in \R$, and ${\rm ess\,inf} \widetilde Q_0 \approx -6.933$ and hence  
\eqref{poll-last-ex} yields $\sigma_{poll}((A_n)_{n\in\N}) \subset [0,\infty)$. 
The eigenvalues of the truncated operator $A_n$ on the interval $[-s_n,s_n]$ with Dirichlet boundary conditions,  
which were computed numerically using a shooting method implemented in Wolfram Mathematica, are shown in Figure~\ref{fig.CCDirichlet}
for increasing values of  \vspace{-2mm}$s_n\in[0,9]$.
\begin{figure}[htb]
\begin{center}
\includegraphics[width=0.6\textwidth]{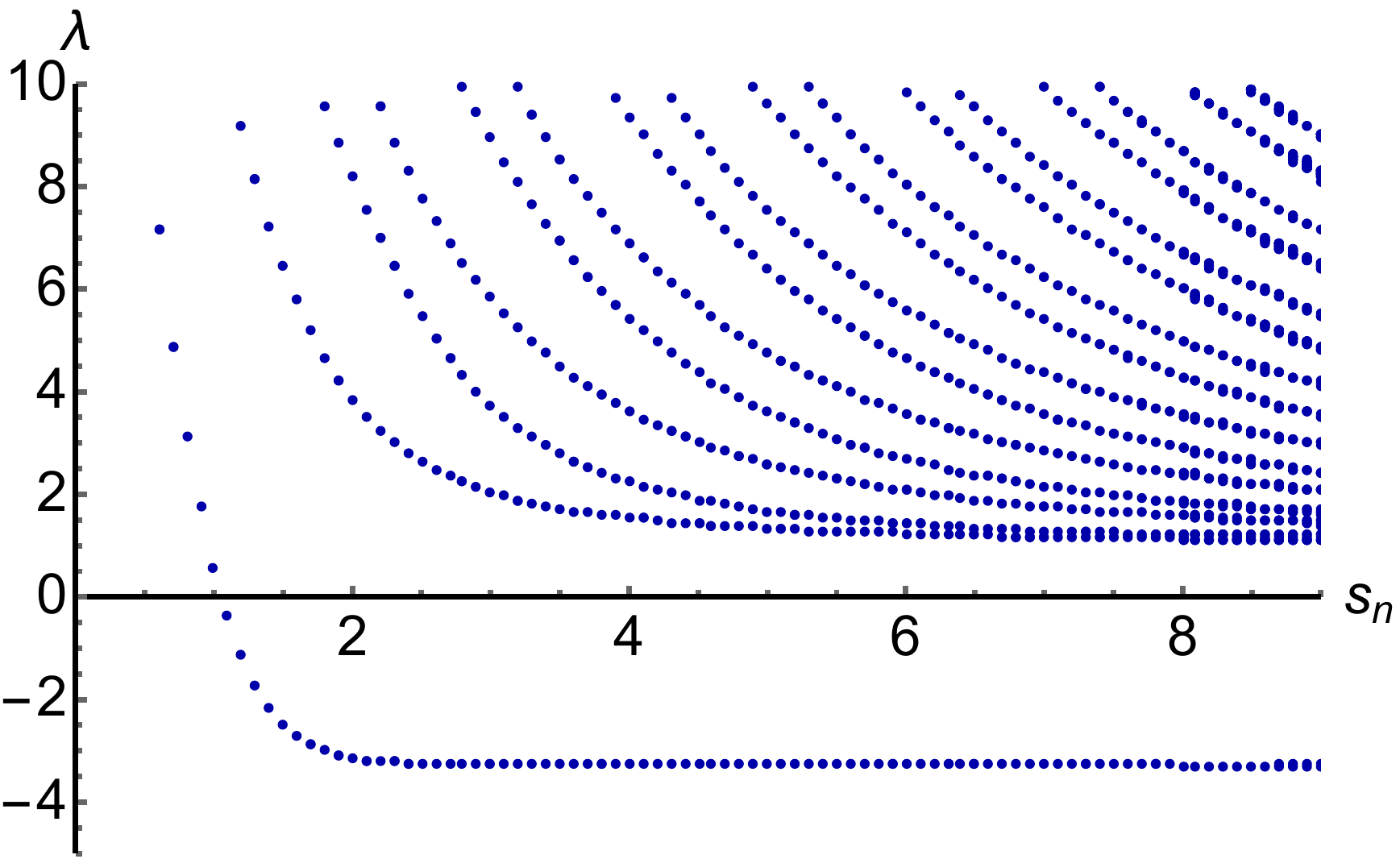}
\caption{\small Eigenvalues in the interval $[-5,10]$ of $A_n$ with $Q_1(x)=-2$, $Q_0(x)=20\sin(x)\e^{-x^2}$ 
truncated to $[-s_n,s_n]$ for different values of $s_n$.}
\label{fig.CCDirichlet}
\end{center} 
\end{figure}

%\marginpar{\tiny\ct{ours only shows `may'}}
Our result \eqref{We-last-ex} shows, first, that all accumulation points in $[0,\infty)$ may be~spurious and, 
secondly, that the accumulation point $\lm\approx -3.25$ which does \emph{not} belong to $W_e(A)$ is \emph{not} a spurious but a true eigenvalue, 
i.e.\ $\lm \in \sigma(A)$, see also Figure~\ref{fig.CCDirichletPlane}. 
This result agrees with the spectral exactness results of~\cite[Section~4.2]{boegli-limitingess} 
by which 
the only discrete eigenvalue of $A$ in the box $[-5,10]+[-5,5]\,\I $ is the point $\lm\approx -3.25$.
\end{example}

{\bf Acknowledgements.} 
% sorted?
%\marginpar{\tiny\ct{\bf are these thanks really still in place? Then we also would have to add Arlinksii ...}}
{\small 
%S.B.\ would like to thank L.\ Boulton for fruitful discussions that lead e.g.\ to Remark~\ref{rempower}. M.M.\ would like to thank W.\ Des 
%Evans for discussions about the use of results from his book \cite{edmundsevans}  in Theorem \ref{thmtruncation}.
The authors thank Y. Arlinskii, L.\ Boulton and W.\ Des Evans for fruitful discussions.  They
%The authors 
also gratefully acknowledge the support of the Swiss National Science Foundation (SNF), 
grant no.\ 200020\_146477 (S.B., C.T.) and Early Postdoc Mobility project P2BEP2\_159007 (S.B.).
}

\bibliography{mybib-jfa}{}
\bibliographystyle{acm}

\end{document}